\documentclass[reqno]{amsart}

\usepackage[latin1]{inputenc}
\usepackage[T1]{fontenc}
\usepackage[english]{babel}
\usepackage{latexsym}
\usepackage{amssymb}
\usepackage[backend = bibtex, sorting = nty, maxnames = 99]{biblatex}
\usepackage{enumitem}
\usepackage{mathtools}
\usepackage[hyperindex,breaklinks]{hyperref}
\mathtoolsset{centercolon}

\numberwithin{equation}{section} 

\setenumerate{leftmargin=*}

\allowdisplaybreaks

\newtheorem{theorem}{Theorem}[section]
\newtheorem{lem}[theorem]{Lemma}
\newtheorem{cor}[theorem]{Corollary}
\newtheorem{prop}[theorem]{Proposition}
\newtheorem{definition}[theorem]{Definition}

\theoremstyle{remark}
\newtheorem{rem}[theorem]{Remark}

\newcommand{\R}{\mathbb{R}}

\newcommand{\N}{\mathbb{N}}

\newcommand{\Ltwoa}{L^2(\Omega)}

\newcommand{\Ltwoh}{L^2(\R^3_+)}
\newcommand{\Ltwohdom}{L^2(G)}
\newcommand{\Ltwohn}[1]{\|#1\|_{\Ltwoh}}
\newcommand{\Ltwohndom}[1]{\|#1\|_{\Ltwohdom}}
\newcommand{\Ltwoan}[1]{\|#1\|_{L^2_\gamma(\Omega)}}
\newcommand{\Ltwoanwgamma}[1]{\|#1\|_{L^2(\Omega)}}
\newcommand{\Ltwohnt}[1]{\Gnorm{0}{#1}}

\newcommand{\Ltwodomh}{L^2(G)}

\newcommand{\Ltwodomanv}[2]{\|#1\|_{L^2_\gamma(J_{#2} \times G)}}

\newcommand{\Hh}[1]{H^{#1}(\R^3_+)}
\newcommand{\Hhdom}[1]{H^{#1}(G)}

\newcommand{\Ha}[1]{H^{#1}(\Omega)}
\newcommand{\Hadom}[1]{H^{#1}(J \times G)}
\newcommand{\Hagamma}[1]{H^{#1}_\gamma(\Omega)}

\newcommand{\HaPdom}[1]{H^{#1}(J_\tau \times G)}

\newcommand{\Hhn}[2]{\|#2\|_{\Hh{#1}}}
\newcommand{\Hhndom}[2]{\|#2\|_{\Hhdom{#1}}}
\newcommand{\Han}[2]{\|#2\|_{\Ha{#1}}}
\newcommand{\Handom}[2]{\|#2\|_{H^{#1}(J \times G)}}
\newcommand{\HanPdom}[2]{\|#2\|_{H^{#1}(J_\tau \times G)}}
\newcommand{\Hangamma}[2]{\|#2\|_{\Hagamma{#1}}}
\newcommand{\Hangammadom}[2]{\|#2\|_{H^{#1}_\gamma(J \times G)}}

\newcommand{\HangammaPdom}[2]{\|#2\|_{H^{#1}_\gamma(J_\tau \times G)}}

\newcommand{\G}[1]{G_{#1}(\Omega)}

\newcommand{\Gnorm}[2]{\|#2\|_{G_{#1, \gamma}(\Omega)}}
\newcommand{\Gnormdom}[2]{\|#2\|_{G_{#1, \gamma}(J \times G)}}
\newcommand{\Gnormdomwg}[2]{\|#2\|_{G_{#1}(J \times G)}}

\newcommand{\Gnormwg}[2]{\|#2\|_{G_{#1}(\Omega)}}

\newcommand{\Gdom}[1]{G_{#1}(J \times G)}
\newcommand{\GdomP}[1]{G_{#1}(J_\tau \times G)}

\newcommand{\Gdomvar}[1]{\tilde{G}_{#1}(J \times G)}
\newcommand{\GdomvarP}[1]{\tilde{G}_{#1}(J_\tau \times G)}

\newcommand{\Gdomnorm}[2]{\|#2\|_{G_{#1, \gamma}(J \times G)}}
\newcommand{\Gdomnormwg}[2]{\|#2\|_{G_{#1}(J \times G)}}
\newcommand{\GdomnormP}[2]{\|#2\|_{G_{#1, \gamma}(J_\tau \times G)}}

\newcommand{\GdomnormwgP}[2]{\|#2\|_{G_{#1}(J_\tau \times G)}}

\newcommand{\Gdomnormv}[3]{\|#2\|_{G_{#1,\gamma}(J_{#3} \times G)}}
\newcommand{\Gdomnormwgv}[3]{\|#2\|_{G_{#1}(J_{#3} \times G)}}

\newcommand{\Fuwl}[2]{F_{#1}^{\operatorname{#2}}(\Omega)}

\newcommand{\Fupdwl}[3]{F_{#1,#3}^{\operatorname{#2}}(\Omega)}

\newcommand{\Fuwlk}[3]{F_{#1,#3}^{\operatorname{#2}}(\Omega)}

\newcommand{\Fupdwlk}[4]{F_{#1,#4,#3}^{\operatorname{#2}}(\Omega)}

\newcommand{\Fdom}[1]{F_{#1}(J \times G)}
\newcommand{\Fdomk}[2]{F_{#1,#2}(J \times G)}
\newcommand{\FdomP}[1]{F_{#1}(J_\tau \times G)}

\newcommand{\Fdompdw}[2]{F_{#1,#2}(J \times G)}

\newcommand{\Fdomuwl}[2]{F_{#1}^{\operatorname{#2}}(J \times G)}

\newcommand{\Fdomupdwl}[3]{F_{#1,#3}^{\operatorname{#2}}(J \times G)}
\newcommand{\Fdompdk}[3]{F_{#1,#3,#2}(J \times G)}

\newcommand{\Fdomnorm}[2]{\|#2\|_{F_{#1}(J \times G)}}

\newcommand{\Fvarwkdom}[2]{F^0_{#1,#2}(G)}
\newcommand{\Fcoeff}[2]{F_{#1, \operatorname{coeff}}^{\operatorname{#2}}(\R^3_+)}

\newcommand{\Fvarnormdom}[2]{\|#2\|_{F^0_{#1}(G)}}
\newcommand{\Fnormdom}[2]{\|#2\|_{F_{#1}(J \times G)}}
\newcommand{\E}[1]{E_{#1}(J \times \partial \R^3_+)}
\newcommand{\Edom}[1]{E_{#1}(J \times \partial G)}
\newcommand{\Enormwg}[2]{\|#2\|_{E_{#1}(J \times \partial \R^3_+)}}
\newcommand{\Enorm}[2]{\|#2\|_{E_{#1, \gamma}(J \times \partial \R^3_+)}}
\newcommand{\Enormwgdom}[2]{\|#2\|_{E_{#1}(J \times \partial G)}}
\newcommand{\Enormdom}[2]{\|#2\|_{E_{#1, \gamma}(J \times \partial G)}}
\newcommand{\EnormdomP}[2]{\|#2\|_{E_{#1, \gamma}(J_\tau \times \partial G)}}

\newcommand{\ml}[3]{\mathcal{ML}^{#1}(#2,#3)}
\newcommand{\mlwl}[4]{\mathcal{ML}^{#1,\operatorname{#4}}(#2,#3)}
\newcommand{\mlpd}[3]{\mathcal{ML}_{\operatorname{pd}}^{#1}(#2,#3)}
\newcommand{\mlpdwl}[4]{\mathcal{ML}_{\operatorname{pd}}^{#1, \operatorname{#4}}(#2,#3)}
\newcommand{\curl}{\operatorname{curl}}

\renewcommand{\div}{\operatorname{div}}

\newcommand{\tr}{\operatorname{tr}}

\newcommand{\clJ}{\overline{J}}

\newcommand{\dist}{\operatorname{dist}}

\newcommand{\supp}{\operatorname{supp}}

\newcommand{\image}{\operatorname{im}}

\renewcommand{\epsilon}{\varepsilon}
\renewcommand{\phi}{\varphi}

\bibliography{LiteratureLocalWellposednessNonlinearMaxwell}

\title[Local wellposedness of nonlinear Maxwell equations]{Local wellposedness of nonlinear Maxwell equations with perfectly conducting boundary conditions}
\author{Martin Spitz}
\address[Martin Spitz]{Department of Mathematics, Karlsruhe Institute of Technology, Englerstr. 2, 76131 Karlsruhe}
\email{martin.spitz@kit.edu}
\subjclass[2010]{35L50, 35L60, 35Q61}
\keywords{nonlinear Maxwell equations, perfectly conducting boundary conditions, quasilinear initial 
boundary value problem, hyperbolic system, local wellposedness, blow-up criterion, continuous dependance}

\begin{document}
\begin{abstract}
 In this article we develop the local wellposedness theory for quasilinear Maxwell equations in $H^m$ for 
 all $m \geq 3$ on domains with perfectly conducting boundary conditions. The macroscopic Maxwell equations with 
 instantaneous material laws for the polarization and the magnetization lead to a quasilinear first order hyperbolic 
 system whose wellposedness in $H^3$ is not covered by the available results in this case. 
 We prove the existence and uniqueness of local solutions in $H^m$ with $m \geq 3$ of the corresponding initial boundary value problem 
 if the material laws and the data are accordingly regular and compatible. We further characterize finite time blowup in 
 terms of the Lipschitz norm and we show that the solutions depend continuously on their data. Finally, we establish
 the finite propagation speed of the solutions.
\end{abstract}
 \maketitle
 
 \section{Introduction and main result}
 \label{SectionIntroduction}
 Describing the theory of electromagnetism, the Maxwell equations are one of the fundamental partial 
 differential equations in physics. Equipped with instantaneous nonlinear material laws, 
 they form a quasilinear hyperbolic system. On the full space $\R^d$, for such systems 
 Kato has established a satisfactory local wellposedness theory in $H^s(\R^d)$ for $s > \frac{d}{2} + 1$, 
 see~\cite{KatoQuasilinearHyperbolicSystems}. However, on a domain $G \neq \R^3$, the Maxwell 
 equations with the boundary conditions of a perfect conductor do not belong to the classes of hyperbolic systems 
 for which one knows such a wellposedness theory.
 The available results need much more regularity and exhibit a loss of derivatives, 
 see~\cite{Gues}.
 
 In this work we provide a complete local wellposedness theory for quasilinear Maxwell equations on 
 domains with a perfectly conducting boundary. We prove the existence and uniqueness of solutions in 
 $H^m(G)$ for all $m \geq 3$ if the material laws and the data are accordingly regular and compatible, 
 provide a blow-up condition in the Lipschitz norm, and show the continuous dependance of the solutions on 
 the data, see also~\cite{SpitzDissertation}. 
 These results are based on a detailed regularity theory for the corresponding nonautonomous linear 
 equation which we developed in the companion paper~\cite{SpitzMaxwellLinear}.
 Here and in~\cite{SpitzMaxwellLinear}, we crucially use the special structure 
 of the Maxwell system.
 
The macroscopic Maxwell equations in a domain $G$ read 
\begin{align}
\label{EquationMaxwellSystem}
\begin{aligned}
 \partial_t \boldsymbol{D} &= \curl \boldsymbol{H} -  \boldsymbol{J}, \qquad  &&\text{for } x \in G, \quad  &&t \in (t_0,T), \\
 \partial_t \boldsymbol{B} &= -\curl \boldsymbol{E},  	 &&\text{for } x \in G, &&t \in (t_0,T),  \\
 \div \boldsymbol{D} &= \rho,  \quad \div \boldsymbol{B} = 0, &&\text{for } x \in G, &&t \in (t_0,T),  \\
 \boldsymbol{E} \times \nu &= 0, \quad  \boldsymbol{B} \cdot \nu = 0, &&\text{for } x \in \partial G, &&t \in (t_0,T),  \\
 \boldsymbol{E}(t_0) &= \boldsymbol{E}_0, \quad \boldsymbol{H}(t_0) = \boldsymbol{H}_0, &&\text{for } x \in G,
\end{aligned}
 \end{align}
for an initial time $t_0 \in \R$. Here $\boldsymbol{E}(t,x), \boldsymbol{D}(t,x) \in \R^3$ denote the 
electric and 
$\boldsymbol{H}(t,x),\boldsymbol{B}(t,x) \in \R^3$ the magnetic 
fields. The charge density $\rho(t,x)$ is related with the current density $\boldsymbol{J}(t,x) \in \R^3$  via
\begin{equation*}
 \rho(t) = \rho(t_0) - \int_{t_0}^t \div \boldsymbol{J}(s) ds
\end{equation*}
for all $t \geq t_0$. 
In~\eqref{EquationMaxwellSystem} we consider the Maxwell system with the boundary conditions of a 
perfect conductor, where $\nu$ denotes the outer normal unit vector of~$G$. 
System~\eqref{EquationMaxwellSystem} has to be complemented by constitutive relations 
between the electric fields and the magnetic fields, the so called material laws. We 
choose $\boldsymbol{E}$ and $\boldsymbol{H}$ as state variables and express 
$\boldsymbol{D}$ and $\boldsymbol{B}$ in terms of $\boldsymbol{E}$ and $\boldsymbol{H}$.
The actual form of these material laws is a question of modelling and different kinds have been 
considered in the literature. The so 
called retarded material laws  assume that the fields $\boldsymbol{D}$ and $\boldsymbol{B}$ depend also on the past of $\boldsymbol{E}$ and $\boldsymbol{H}$, 
see~\cite{BabinFigotin} and~\cite{RoachStratisYannacopoulos} for instance. 
In dynamical material laws 
the material response is modelled by additional evolution equations for the polarization or magnetization, see e.g.~\cite{AmmariHamdache}, 
\cite{DumasSueur}, \cite{Jochmann}, or~\cite{JolyMetivierRauch}.
In this work we concentrate on the instantaneous material laws, see~\cite{BuschEtAl} and~\cite{FabrizioMorro}. Here 
the fields $\boldsymbol{D}$ and $\boldsymbol{B}$ are given as local functions of $\boldsymbol{E}$ and 
$\boldsymbol{H}$, i.e., we assume that there are functions $\theta_1, \theta_2 \colon G \times \R^6 \rightarrow \R^3$ 
such that $\boldsymbol{D}(t,x) = \theta_1(x,\boldsymbol{E}(t,x), \boldsymbol{H}(t,x))$ and 
$\boldsymbol{B}(t,x) = \theta_2(x,\boldsymbol{E}(t,x), \boldsymbol{H}(t,x))$. The most prominent 
example is the so called Kerr nonlinearity, where
\begin{equation}
\label{EquationKerrNonlinearity}
	\boldsymbol{D} =  \boldsymbol{E} + \vartheta |\boldsymbol{E}|^2 \boldsymbol{E}, \quad \boldsymbol{B} = \boldsymbol{H},
\end{equation}
with $\vartheta \colon G \rightarrow \R^{3 \times 3}$ and the vacuum permittivity and permeability 
set equal to~$1$ for convenience. 
We further assume that the current density decomposes as $\boldsymbol{J} = \boldsymbol{J}_0 + \sigma_1(\boldsymbol{E}, \boldsymbol{H}) \boldsymbol{E}$, 
where $\boldsymbol{J}_0$ is an external current density and $\sigma_1$ denotes the conductivity. If we insert these material laws into~\eqref{EquationMaxwellSystem} 
and formally differentiate, we obtain
\begin{align*}
	(\partial_t \boldsymbol{D}, \partial_t \boldsymbol{B}) = \partial_{(\boldsymbol{E}, \boldsymbol{H})} \theta(x,\boldsymbol{E}, \boldsymbol{H}) \partial_t (\boldsymbol{E}, \boldsymbol{H}) = (\curl \boldsymbol{H} - \boldsymbol{J}, - \curl \boldsymbol{E})
\end{align*}
for the evolutionary part of~\eqref{EquationMaxwellSystem}. 
The resulting equation is a first order quasilinear hyperbolic system. 
In order to write~\eqref{EquationMaxwellSystem}  in the standard form of 
first order systems, we introduce the matrices
\begin{align*}
 J_1 = \begin{pmatrix}
	  0 &0 &0 \\
	  0 &0 &-1 \\
	  0 &1 &0
       \end{pmatrix},
        \quad
 J_2 =  \begin{pmatrix}
         0 &0 &1 \\
         0 &0 &0 \\
         -1 &0 &0
        \end{pmatrix},
        \quad
 J_3 = \begin{pmatrix}
        0 &-1 &0 \\
        1 &0 &0 \\
        0 &0 &0
       \end{pmatrix}
\end{align*}
and
\begin{align}
\label{EquationDefinitionOfAj}
 A_j^{\operatorname{co}} = \begin{pmatrix}
        0 & -J_j \\
        J_j &0 
       \end{pmatrix}
\end{align}
for $j = 1, 2, 3$. Observe that $\sum_{j = 1}^3 J_j \partial_j = \curl$. Writing $\chi$ for $\partial_{(\boldsymbol{E}, \boldsymbol{H})} \theta$, 
$f = (- \boldsymbol{J}_0,0)$, $\sigma = \begin{pmatrix} \sigma_1 &0 \\ 0 &0 \end{pmatrix}$, and 
using $u = (\boldsymbol{E}, \boldsymbol{H})$ as new variable, we finally obtain
\begin{equation}
\label{EquationMaxwellAsFirstOrderSystemEvolutionPart}
	\chi(u) \partial_t u + \sum_{j = 1}^3 A_j^{\operatorname{co}} \partial_j u + \sigma(u) u = f.
\end{equation}
Under mild regularity conditions, e.g., $\boldsymbol{D}, 
\boldsymbol{B} \in C([t_0,T], \Hhdom{1}) \cap C^1([t_0,T], \Ltwohdom)$ and 
$\operatorname{div} \boldsymbol{J} \in L^1((t_0,T), \Ltwoh)$, 
a solution 
$u = (\boldsymbol{E}, \boldsymbol{H})$ of~\eqref{EquationMaxwellAsFirstOrderSystemEvolutionPart}
satisfies the divergence conditions in~\eqref{EquationMaxwellSystem} if they hold at the initial time
$t_0$. Similarly, the second part of the boundary conditions, i.e., $\boldsymbol{B} \cdot \nu = 0$ on 
$(t_0,T) \times \partial G$ is true if $\boldsymbol{E} \times \nu = 0$ on $(t_0,T) \times \partial G$ 
and $\boldsymbol{B}(t_0) \cdot \nu = 0$ on $\partial G$. We refer to~\cite[Lemma~7.25]{SpitzDissertation} 
for details. Defining the matrix
\begin{equation}
\label{EquationDefinitionB}
	B = \begin{pmatrix}
      0 &\nu_3 &-\nu_2 &0 &0 &0 \\
      -\nu_3 &0 &\nu_1 &0 &0 &0 \\
      \nu_2 &-\nu_1 &0 &0 &0 &0 
     \end{pmatrix}
\end{equation}
on $\partial G$, we can then cast system~\eqref{EquationMaxwellSystem} into the equivalent first order quasilinear hyperbolic initial boundary value problem
 \begin{align}
 \label{EquationNonlinearIBVP} 
  \left\{\begin{aligned}
   \chi(u)\partial_t u + \sum_{j=1}^3 A_j^{\operatorname{co}} \partial_j u + \sigma(u) u &= f, \quad &&x \in G, \quad &&t \in J; \\
   B u &= g, &&x \in \partial G, &&t \in J; \\
   u(t_0) &= u_0, &&x \in  G,
 \end{aligned}\right.
\end{align}
with additional conditions for the initial value.
Here $J = (t_0, T)$ is an open interval. We allow for inhomogeneous boundary values 
which are not only interesting from the mathematical point of view, but also have physical relevance, 
see~\cite{DautrayLionsI}. We further
assume that $\chi$ is symmetric and at least locally positive definite, which includes~\eqref{EquationKerrNonlinearity}.
Such assumptions are quite standard already for linear Maxwell equations.

The initial value problem on the full space (without boundary conditions) corresponding to~\eqref{EquationNonlinearIBVP} has been solved by 
Kato in~\cite{KatoQuasilinearHyperbolicSystems} in a more general setting. 
Kato freezes a function $\hat{u}$ in the nonlinearities and employs a priori estimates 
for the corresponding linear problem previously obtained in~\cite{KatoLinearEvolutionEquationsI} 
and~\cite{KatoLinearEvolutionEquationsII}. Then a fixed point argument yields
a solution of the quasilinear problem. However, Kato works in a general functional analytic setting 
which does not cover the Maxwell equations with perfectly conducting boundary 
conditions. An alternative approach is to use energy techniques 
in order to derive the a priori estimates for the linear problem needed to apply a fixed point argument. 
This strategy was applied successfully to quasilinear hyperbolic initial boundary value problems with 
noncharacteristic boundary (i.e., where the boundary matrix $A(\nu) = \sum_{j = 1}^d A_j \nu_j$ is 
nonsingular), see~\cite{BenzoniGavage,MetivierEtAl, RauchMassey}. 
Systems with characteristic boundary pose several additional difficulties. In particular, no general theory 
for the corresponding linear problems is available and even a loss of regularity there is possible, 
see~\cite{MajdaOsher}. In~\cite{OhkuboLinear} an additional structural assumption is proposed in order 
to prevent this loss of regularity  and a quasilinear result is derived from it in~\cite{OhkuboQuasilinear}. 
However, the Maxwell system does not fulfill this structural assumption. 
A different approach is taken in~\cite{Gues} for quasilinear hyperbolic initial boundary value 
problems with characteristic boundary. 
The results there require high regularity (at least $H^6(G)$), are given in 
Sobolev-like spaces incorporating weights in the normal direction, and contain a loss of regularity. In~\cite{PicardZajaczkowski} 
the authors focus on Maxwell's equations, but treat different boundary conditions (belonging to 
the class treated in~\cite{MajdaOsher}, see also~\cite{CagnolEller}) than the perfectly conducting ones. 
Moreover, only the existence of a local solution of the quasilinear system is obtained there.

Somehow surprisingly, the quasilinear Maxwell system~\eqref{EquationMaxwellSystem} with perfectly conducting 
boundary has not yet been treated in the natural space $\Hhdom{3}$ and even the basic questions on local existence and uniqueness are still open. 
We will close this gap 
by providing a complete 
local wellposedness theory. We will prove that
\begin{enumerate}
	\item the system~\eqref{EquationNonlinearIBVP} has a unique maximal solution $u$ in the function space \linebreak
	      $\bigcap_{j = 0}^m C^j((T_{-}, T_+), \Hhdom{m-j})$ for all $m \in \N$ with $m \geq 3$ provided the data are sufficiently regular 
	      and compatible (in the sense of~\eqref{EquationNonlinearCompatibilityConditions} below) 
	      with the material law,
	\item finite existence time can be characterized by blowup in the Lipschitz-norm, 
	\item the solution depends continuously on the data.
\end{enumerate}
We refer to Theorem~\ref{TheoremLocalWellposednessNonlinear} for the precise statement. 
The derivation of global properties for~\eqref{EquationNonlinearIBVP} is a highly nontrivial task.
In particular, it is already known 
that global existence cannot be expected for all data. Blow-up examples in the Lipschitz-norm are 
given in~\cite{Majda}. For different boundary conditions than we consider, blow-up examples in the $H(\curl)$-norm are 
provided in~\cite{DAnconaNicaiseSchnaubelt}.

In order to prove the local wellposedness theorem, we follow the strategy mentioned above. We freeze 
a function $\hat{u}$ in the nonlinearities in~\eqref{EquationNonlinearIBVP} and employ energy 
estimates to set up a fixed point argument. However, energy techniques work in $L^2$-based spaces 
but require Lipschitz coefficients, see~\cite{Eller}. The solutions there are constructed in 
$C(\clJ, \Ltwohdom)$ but in view of our fixed point argument, we need that $\chi(u)$ is contained in 
$W^{1,\infty}(J \times G)$. This gap in integrability is bridged by Sobolev's embedding. If a solution $u$ 
belongs to $C(\clJ, \Hhdom{m}) \cap C^1(\clJ, \Hhdom{m-1})$ for a number $m \in \N_0$, then $\chi(u)$ is an element of $W^{1,\infty}(J \times G)$ 
if $m > \frac{3}{2} + 1$ and $\chi$ is reasonable regular. We thus require $m \geq 3$. Reasonable regular here means that $\chi$ belongs to $C^m(G \times \mathcal{U}, \R^{6 \times 6})$ 
and that $\chi$ and all of its derivatives up to order $m$ are bounded on $G \times \mathcal{U}_1$
for every compact subset $\mathcal{U}_1$ of $\mathcal{U}$, where $\mathcal{U}$ is a domain in~$\R^6$.
For later reference we introduce the spaces 
\begin{align*}
 \ml{m,n}{G}{\mathcal{U}} &= \{\theta \in C^m(G \!\times \! \mathcal{U}, \R^{n \times n}) \colon \text{For all } \alpha \in \N_0^9 \text{ with }
 |\alpha| \leq m \text{ and } \mathcal{U}_1 \Subset \mathcal{U} \!: \\
 &\hspace{12em} \sup_{(x,y) \in G \times \mathcal{U}_1} |\partial^\alpha \theta(x,y)| < \infty\}, \\
 \mlpd{m,n}{G}{\mathcal{U}} &= \{\theta \in \ml{m,n}{G}{\mathcal{U}} \colon \text{There exists } \eta > 0 \text{ with } 
 \theta \! = \! \theta^T \! \geq \! \eta \text{ on }  G \! \times \! \mathcal{U}\}
\end{align*}
for $n \in \N$. Actually, we only need $n = 1$ or $n = 6$. Although $C(\clJ, \Hhdom{m}) \cap C^1(\clJ, \Hhdom{m-1})$ embeds into 
$W^{1,\infty}(J \times G)$ for $m \geq 3$, the techniques we are going to apply to solve~\eqref{EquationNonlinearIBVP} 
require that its solution has the same amount of regularity in time 
as in space. We thus construct solutions of~\eqref{EquationNonlinearIBVP} in the function spaces
\begin{equation}
	\label{EquationDefinitionGm}
	\Gdom{m} = \bigcap_{j = 0}^m C^j(\clJ, \Hhdom{m-j})^6
\end{equation}
for all $m \in \N_0$, cf.~\cite{BenzoniGavage,MetivierEtAl,RauchMassey}.
(We do not indicate the dimension of $u$ below.)
 Defining the function $e_{-\gamma} \colon t \mapsto e^{- \gamma t}$, we equip the space $\Gdom{m}$ with the family of time-weighted norms
\begin{equation*}
	\Gnormdom{m}{v} = \max_{0 \leq j \leq m} \|e_{-\gamma} \partial_t^j v\|_{L^\infty(J, \Hhdom{m-j})}
\end{equation*}
for all $\gamma \geq 0$. In the case $\gamma = 0$, we also write $\Gnormdomwg{m}{v}$ instead of $\|v\|_{G_{m,0}(J \times G)}$.
Analogously, any time-space norm indexed by $\gamma$ means the usual norm complemented by
the time weight $e_{-\gamma}$ in the following.
We also need the spaces $\Gdomvar{m}$, consisting of those functions $v$ 
with $\partial^\alpha v \in L^\infty(J, \Ltwohdom)$ for all $\alpha \in \N_0^4$ with 
$0 \leq |\alpha| \leq m$. We equip them with the same family of norms as $\Gdom{m}$. 

The paper is organized as follows. In Section~\ref{SectionCalculusCompatibilityConditions} we first 
study the regularity properties of $\theta(u)$ for $\theta \in \ml{m}{G}{\mathcal{U}}$ and $u \in \Gdomvar{m}$. 
Based on Fa{\'a} di Bruno's formula we find suitable function spaces and provide corresponding estimates 
for these compositions. We further discuss the compatibility conditions. These are necessary conditions 
for the existence of a $\Gdom{m}$-solution, which arise since the differential equation and the 
boundary condition in~\eqref{EquationNonlinearIBVP} both yield information about $u$ and its 
time derivatives on $\{t_0\} \times \partial G$.

In Section~\ref{SectionLocalExistence} we then follow the strategy described above to deduce 
existence and uniqueness of a solution of~\eqref{EquationNonlinearIBVP}. We freeze a function 
$\hat{u}$ from $\Gdomvar{m}$ and study the arising linear problem
\begin{equation}
	\label{EquationIBVPIntroduction}
	 \left\{\begin{aligned}
   A_0 \partial_t u + \sum_{j=1}^3 A_j^{\operatorname{co}} \partial_j u + D u &= f, \quad &&x \in G, \quad &&t \in J; \\
   B u &= g, &&x \in \partial G, &&t \in J; \\
   u(t_0) &= u_0, &&x \in  G;
 \end{aligned}\right.
\end{equation}
where $A_0 = \chi(\hat{u})$ and $D = \sigma(\hat{u})$. Already the higher order 
energy estimates for this linear problem are difficult to obtain since the Maxwell system 
has a characteristic boundary (i.e., $A(\nu) = \sum_{j = 1}^3 A_j^{\operatorname{co}} \nu_j$ is singular). 
Here we rely on~\cite{SpitzMaxwellLinear}, where the structure of the Maxwell equations is heavily exploited in order 
to derive these estimates. We show how the results from~\cite{SpitzMaxwellLinear} can be employed to set up a fixed point 
argument which yields unique local solutions of the quasilinear problem~\eqref{EquationNonlinearIBVP}.
In order to characterize a finite maximal existence time in terms of the Lipschitz norm of the solution,
the a priori estimates from~\eqref{EquationIBVPIntroduction} 
are not good enough. We have to use in Section~\ref{SectionBlowUpCriteria} that the coefficient $A_0$ equals $\chi(u)$ with $u$ being 
the solution of~\eqref{EquationNonlinearIBVP}. Combining Moser-type inequalities and
estimates from~\cite{SpitzMaxwellLinear} relying on the structure of~\eqref{EquationIBVPIntroduction}, 
we can then control the $\Hhdom{m}$-norm of $u$ by its Lipschitz norm. 
In Section~\ref{SectionContinuousDependance} we study the continuous dependance of the solutions on the data. 
Once again the estimates from~\cite{SpitzMaxwellLinear} cannot be applied directly as they do not prevent
the loss of regularity due to the quasilinearity when the difference of two solutions is considered. 
We therefore have to combine the techniques from~\cite{SpitzMaxwellLinear} with certain decompositions 
of $u$ already used in the full space 
case, cf.~\cite{BahouriCheminDanchin}.
Finally, we also prove that solutions of~\eqref{EquationNonlinearIBVP} have finite propagation speed. 
Here, we use a simple and flexible method relying on weighted energy estimates for the linear problem, cf.~\cite{BahouriCheminDanchin}. 

\textbf{Standing notation:}
Let $m$ be a nonnegative integer. We denote by $G$ a domain in $\R^3$ with 
compact $C^{\max\{m,3\} + 2}$-boundary (the assumption that $\partial G$ is compact can be loosened, 
see~\cite{SpitzDissertation} for details) or the half-space $\R^3_+ = \{x \in \R^3 \colon x_3 > 0\}$.
The set $\mathcal{U}$ is a domain in $\R^6$.
We further write $L(A_0, \ldots, A_3, D)$ for 
 the differential operator $A_0 \partial_t + \sum_{j = 1}^3 A_j \partial_j + D$ and $\partial_0$ for 
 $\partial_t$.
 By $J$ we mean an open time interval and we set $\Omega = J \times \R^3_+$. Finally, the image of 
 a function $v$ is designated by $\image v$.

 \section{Calculus and compatibility conditions}
 \label{SectionCalculusCompatibilityConditions}
 
 In the study of quasilinear problems one often has to control compositions $\theta(v)$ in higher regularity 
in terms of $v$. Derivatives of $\theta(v)$ can be expressed by the so called Fa\'a di Bruno's 
formula, which is therefore widespread in the 
literature, see e.g.~\cite{BahouriCheminDanchin,BenzoniGavage}. More important than the formula itself are 
the estimates which follow from it. We provide both in the next lemma.
Its proof is an iterative application of the chain and product rule combined with Sobolev's embedding for the estimates 
and it works as the proof of~\cite[Lemma~7.1]{SpitzDissertation}. We work with functions $v$ taking values in $\R^6$ 
in the following.
 \begin{lem}
 \label{LemmaHigherOrderChainRule}
  Let $m \in \N$ and $\tilde{m} = \max\{m,3\}$.  Let $\mathcal{U}_1$ 
  be a compact subset of $\mathcal{U} \subseteq \R^6$.
 \begin{enumerate}[leftmargin = 2em]
  \item \label{ItemHigherOrderChainRuleInG} Let $\theta \in \ml{m,1}{G}{\mathcal{U}}$ and
    $v \in \Gdomvar{\tilde{m}}$ with $\image v \subseteq \mathcal{U}$.
	The function $\theta(v)$  belongs to the function space $W^{1,\infty}(J \times G)$ and all its derivatives up to order $m$
	are contained in $L^\infty(J, \Ltwohdom) + L^\infty(J \times G)$. 
	For $\gamma_1,\ldots,
	\gamma_j \in \N_0^4$ with $|\gamma_i| \leq m$, $1 \leq j \leq |\alpha|$, and $\alpha \in \N_0^4$ with 
	$1 \leq|\alpha| \leq m$ there  exist constants 
	$C(\alpha, \gamma_1, \ldots, \gamma_j)$ such that
	\begin{align}
	\label{EquationFormulaHigherOrderChainRuleInG}
	 \partial^\alpha \theta(v) = \sum_{\substack{\beta, \gamma \in \N_0^4, \beta_0 = 0 \\  \beta + \gamma = \alpha}} \,
	 \sum_{1 \leq j \leq |\gamma|} \, &\sum_{\substack{\gamma_1,\ldots,\gamma_j \in \N_0^4 \setminus\{0\} \\ \sum \gamma_i = \gamma}}
	 \, \sum_{l_1,\ldots,l_j = 1}^6
	  C(\alpha, \gamma_1, \ldots, \gamma_j) \\
	  &\hspace{4em}\cdot (\partial_{y_{l_j}} \cdots \partial_{y_{l_1}} \partial_x^{(\beta_1,\beta_2,\beta_3)} \theta)(v)
	      \prod_{i=1}^j \partial^{\gamma_i} v_{l_i}. \nonumber
	\end{align}
	Moreover, there exists a constant $C(\theta, m,  \mathcal{U}_1)$ such that
	\begin{align}
	\label{EquationEstimateForHigherOrderChainRuleInG}
	 \| \partial^{\alpha} \theta(v)\|_{L^\infty(J, \Ltwohdom) + L^\infty(J \times G)} \leq C(\theta, m, \mathcal{U}_1) (1 + \Gdomnormwg{\tilde{m}}{v})^{m} 
	\end{align}
	for all $\alpha \in \N_0^4$ with $|\alpha| \leq m$ and $v \in \Gdomvar{\tilde{m}}$ with $\image v \subseteq \mathcal{U}_1$.

  \item \label{ItemHigherOrderChainRuleInHh} Let $\theta \in \ml{m-1,1}{G}{\mathcal{U}}$ and $v \in \Hhdom{\tilde{m}-1}$ 
	  with $\image v \subseteq \mathcal{U}$. 
	  The composition $\theta(v)$ belongs to
	  $L^\infty(G)$ and all of its derivatives up to order $m$ are contained in $\Ltwohdom + L^\infty(G)$. We further have
	  \begin{align}
	  \label{EquationFormulaHigherOrderChainRuleInHh}
	   \partial^\alpha \theta(v) = \sum_{\substack{\beta, \gamma \in \N_0^3\\  \beta + \gamma = \alpha}} \, \sum_{1 \leq j \leq |\gamma|} \, \sum_{\substack{\gamma_1,\ldots,\gamma_j \in \N_0^3 \setminus\{0\} \\ \sum \gamma_i = \gamma}} 
	  \, &\sum_{l_1, \ldots, l_j = 1}^6 C_0(\alpha, \gamma_1, \ldots, \gamma_j) \nonumber\\
	  &\cdot (\partial_{y_{l_j}} \cdots \partial_{y_{l_1}} \partial_x^\beta \theta)( v)
	      \prod_{i=1}^j \partial^{\gamma_i} v_{l_i}
	  \end{align}
	  for all $v \in \Hhdom{\tilde{m}-1}$ and $\alpha \in \N_0^3$ with $0 < |\alpha|\leq m-1$, and the constants
	  \begin{align*}
	    C_0(\alpha, \gamma_1, \ldots, \gamma_j) = C((0,\alpha_1, \alpha_2,\alpha_3), (0,\gamma_{1,1}, \gamma_{1,2}, \gamma_{1,3}), \ldots, (0,\gamma_{j,1}, \gamma_{j,2}, \gamma_{j,3}))
	  \end{align*}
	  from assertion~\ref{ItemHigherOrderChainRuleInG}.
	  There further exists a constant $C_0(\theta, m,  \mathcal{U}_1)$ such that
	  \begin{align}
	   \label{EquationEstimateForHigherOrderChainRuleInH}
	   \|\partial^\alpha \theta(v)\|_{\Ltwohdom + L^\infty(G)} \leq C_0(\theta,m, \mathcal{U}_1) (1 + \Hhndom{\tilde{m}-1}{v})^{m-1} 
	  \end{align}
	  for all $\alpha \in \N_0^3$ with $|\alpha| \leq m-1$ and $v \in \Hhdom{\tilde{m}-1}$ with $\image v \subseteq \mathcal{U}_1$.
	  
  \item \label{ItemHigherOrderChainRuleMixed} Assume additionally that $m \geq 2$. 
	  Let $\theta \in \ml{m,1}{G}{\mathcal{U}}$, $t_0 \in \clJ$, and $r_0 > 0$. Then $\partial_t^j \theta(v)(t_0)$ belongs to 
	  $\Hhdom{m-j-1}$
	  and there is a constant $C(\theta,m, \mathcal{U}_1)$ such that
	  \begin{align*}
	   &\Hhndom{m-j-1}{\partial_t^j \theta(v)(t_0)} \\
	   &\leq C(\theta,m,\mathcal{U}_1) (1 + \max_{0 \leq l \leq j} \Hhndom{\tilde{m} - l -1}{\partial_t^l v(t_0)})^{m-2} \max_{0 \leq l \leq j} \Hhndom{\tilde{m} - l -1}{\partial_t^l v(t_0)}
	  \end{align*}
	    for all $j \in \{1,\ldots,m-1\}$ and $v \in \Gdomvar{\tilde{m}}$ with $\image v \subseteq \mathcal{U}$ and $\image v(t_0) \subseteq \mathcal{U}_1$.
 \end{enumerate}
\end{lem}

For the contraction property of the fixed point operator in Section~\ref{SectionLocalExistence} and the 
derivation of the continuous dependance, we need similar estimates for the differences of 
such compositions. They are given by the next corollary. Its proof relies on Lemma~\ref{LemmaHigherOrderChainRule} and 
works in the same way as the one of~\cite[Corollary~7.2]{SpitzDissertation}.
\begin{cor}
 \label{CorollaryEstimateForDifferenceHigherOrder}
 Let $m \in \N$, $\tilde{m} = \max\{m,3\}$, and $\gamma \geq 0$. Take 
 $\theta \in \ml{m-1,1}{G}{\mathcal{U}}$, $R > 0$, and 
 pick a compact subset $\mathcal{U}_1$ of $\mathcal{U}$.
 \begin{enumerate}[leftmargin = 2em]
 \item \label{ItemDifferenceInLtwoh}  
  Let $v_1, v_2 \in \Gdomvar{\tilde{m}-1}$ with the bounds  
  $\Gdomnormwg{\tilde{m}-1}{v_1}, \Gdomnormwg{\tilde{m}-1}{v_2} \leq R$ and $\image v_1, \image v_2 \subseteq \mathcal{U}_1$. Then 
  there exists a constant $C = C(\theta,m,R,\mathcal{U}_1)$ such that
  \begin{align}
  \label{EquationEstimateForDifferenceWithoutLinftyTerm}
   \Ltwohndom{(\partial^\alpha \theta(v_1))(t) - (\partial^\alpha \theta(v_2))(s)} \leq C \sum_{\substack{\beta \in \N_0^4 \\ 0 \leq |\beta| \leq \tilde{m}-1}}\Ltwohndom{\partial^\beta v_1(t) - \partial^\beta v_2(s)} 
  \end{align}
  for almost all $t \in J$ and almost all $s \in J$ if $\alpha \in \N_0^4$ with $0 \leq |\alpha| \leq m-2$. In 
  the case $|\alpha| = m-1$ and $m > 1$ we have the estimate
  \begin{align}
  \label{EquationEstimateForDifferenceWithLinftyTerm}
  &\|(\partial^\alpha \theta(v_1))(t) - (\partial^\alpha \theta(v_2))(s)\|_{\Ltwohdom + L^\infty(G)} \leq C \sum_{\substack{\beta \in \N_0^4 \\ 0 \leq |\beta| \leq \tilde{m}-1}}\Ltwohndom{\partial^\beta v_1(t) - \partial^\beta v_2(s)} \nonumber\\
  &\qquad + C  \sum_{0 \leq j \leq m-1} \, \sum_{\substack{0 \leq \gamma \leq \alpha, \gamma_0 = 0 \\ |\gamma| = m-1-j}} \, \sum_{l_1, \ldots, l_{j} = 1}^6 \|(\partial_{y_{l_{j}}} \ldots \partial_{y_{l_1}} \partial_x^{(\gamma_1,\gamma_2,\gamma_3)}\theta)(v_1)(t) \\
  &\hspace{17em} - (\partial_{y_{l_{j}}} \ldots \partial_{y_{l_1}} \partial_x^{(\gamma_1,\gamma_2,\gamma_3)}\theta)(v_2)(s)\|_{L^\infty(G)} \nonumber 
  \end{align}
  for almost all $t \in J$ and almost all $s \in J$. If  $\theta$ additionally belongs to $\ml{m,1}{G}{\mathcal{U}}$, 
  the estimate~\eqref{EquationEstimateForDifferenceWithoutLinftyTerm} is true for almost all $t \in J$ and almost all $s \in J$ in 
  the case $|\alpha| = m-1$.
  Finally, if $\alpha_0 = 0$, it is enough to sum in~\eqref{EquationEstimateForDifferenceWithoutLinftyTerm} 
  and~\eqref{EquationEstimateForDifferenceWithLinftyTerm} over those multiindices $\beta$ with $\beta_0 = 0$.
  \item \label{ItemDifferenceInG} Let $v_1, v_2 \in \Gdomvar{\tilde{m}-1}$ with the bounds  
  $\Gdomnormwg{\tilde{m}-1}{v_1}, \Gdomnormwg{\tilde{m}-1}{v_2} \leq R$ and $\image v_1, \image v_2 \subseteq \mathcal{U}_1$,
  and $\theta \in \ml{m,1}{G}{\mathcal{U}}$.
 Then there exists a 
 constant  $C = C(\theta, m,R, \mathcal{U}_1)$ such that
 \begin{align*}
  \Gdomnorm{m-1}{\theta(v_1) - \theta(v_2)} \leq C \Gdomnorm{\tilde{m}-1}{v_1 - v_2}
 \end{align*}
 for all $\gamma \geq 0$.
 \end{enumerate}
\end{cor}

Assume that $u \in \Gdom{m}$ is a solution of~\eqref{EquationNonlinearIBVP} with inhomogeneity $f \in \Hadom{m}$ and 
initial value $u_0 \in \Hhdom{m}$. Lemma~\ref{LemmaHigherOrderChainRule} implies that we can take $p-1$ time-derivatives of the evolution equation 
in~\eqref{EquationNonlinearIBVP}, insert $t_0 \in \clJ$, and solve for $\partial_t^p u(t_0)$. This procedure yields a closed 
expression for $\partial_t^p u(t_0)$ in terms of the material laws and the data for all 
$p \in \{0, \ldots, m\}$, which will be utterly important in the following.
\begin{definition}
 \label{DefinitionNonlinearSmpOperators}
 Let $J \subseteq \R$ be an open interval, $m \in \N$, 
 and $\chi \in \mlpd{m,6}{G}{\mathcal{U}}$ and $\sigma\in \ml{m,6}{G}{\mathcal{U}}$. 
 We inductively define the operators
\begin{align*}
 S_{\chi, \sigma, G, m,p} \colon \clJ \times \Hadom{\max\{m,3\}} \times H^{\max\{m,2\}}(G,\mathcal{U})  \rightarrow \Hhdom{m-p}
\end{align*}
by $S_{\chi, \sigma,G, m, 0}(t_0, f,u_0) = u_0$ and
\begin{align}
\label{EquationDefinitionHigherOrderInitialValuesNonlinear}
 &S_{\chi, \sigma,G, m,p }(t_0, f,u_0) = \chi(u_0)^{-1} \Big(\partial_t^{p-1}f(t_0) - \sum_{j=1}^3 A_j^{\operatorname{co}} \partial_j S_{\chi, \sigma,G,m,p-1}(t_0, f,u_0) \nonumber\\
    &\hspace{8em} - \sum_{l=1}^{p-1} \binom{p-1}{l} M_1^l(t_0, f, u_0) S_{\chi, \sigma,G,m,p-l}(t_0,f,u_0)  \\
   &\hspace{8em} - \sum_{l=0}^{p-1} \binom{p-1}{l} M_2^l(t_0, f, u_0) S_{\chi, \sigma,G,m,p-1-l}(t_0, f,u_0)\Big), \nonumber\\
 &M_k^p = \sum_{1 \leq j \leq p} \sum_{\substack{\gamma_1,\ldots,\gamma_j \in \N_0^4 \setminus\{0\} \\ \sum \gamma_i = (p,0,0,0)}} \sum_{l_1, \ldots, l_j = 1}^6
	  C((p,0,0,0), \gamma_1, \ldots, \gamma_j)  \nonumber\\
	  &\hspace{10em} \cdot  (\partial_{y_{l_j}} \cdots \partial_{y_{l_1}} \theta_k)(u_0)
	      \prod_{i=1}^j S_{\chi,\sigma,G,m,|\gamma_i|}(t_0,f,u_0)_{l_i} \label{EquationDefinitionMkp}
\end{align}
for $1 \leq p \leq m$, $k \in \{1,2\}$, where $\theta_1 = \chi$, $\theta_2 = \sigma$,  $M_2^0 = \sigma(u_0)$, 
and $C$ is the constant from Lemma~\ref{LemmaHigherOrderChainRule}. By $H^{\max\{m,2\}}(G,\mathcal{U})$ we mean 
those functions $u_0 \in \Hhdom{\max\{m,2\}}$ with $\overline{\image u_0} \subseteq \mathcal{U}$.
\end{definition}
Lemma~\ref{LemmaHigherOrderChainRule} then implies that
\begin{equation}
	\label{EquationTimeDerivativesOfSolutionEqualSChiSigmaG}
	\partial_t^p u(t_0) = S_{\chi, \sigma, G, m,p}(t_0, f, u_0) \quad \text{for all } p \in \{0, \ldots, m\}
\end{equation}
if $u \in \Gdom{m}$ is a solution of~\eqref{EquationNonlinearIBVP} with data $f \in \Hadom{m}$ and $u_0 \in \Hhdom{m}$.
The next lemma shows that the operators $S_{\chi, \sigma, G, m, p}$ indeed map into $\Hhdom{m-p}$ and it provides corresponding 
estimates. The proof relies on Lemma~\ref{LemmaHigherOrderChainRule},
 Corollary~\ref{CorollaryEstimateForDifferenceHigherOrder} and the product estimates from~\cite[Lemma~2.1]{SpitzMaxwellLinear}. We refer to~\cite[Lemma~7.7]{SpitzDissertation} for details.
\begin{lem}
 \label{LemmaNonlinearHigherOrderInitialValues}
 Let $J \subseteq \R$ be an open interval, $t_0 \in \clJ$,  $m \in \N$, and $\tilde{m} = \max\{m,3\}$.
 Take $\chi \in \mlpd{m,6}{G}{\mathcal{U}}$ and $\sigma \in \ml{m,6}{G}{\mathcal{U}}$ 
 Choose data $f, \tilde{f} \in \Hadom{\tilde{m}}$ and $u_0, \tilde{u}_0 \in \Hhdom{\tilde{m}}$ such that 
  $\overline{\image u_0}$ and $\overline{\image \tilde{u}_0}$ are 
 contained in $\mathcal{U}$.
 Take $r > 0$ such that
 \begin{align*}
  \sum_{j = 0}^{\tilde{m}-1} \Hhndom{\tilde{m}-j-1}{\partial_t^j f(t_0)} + \Hhndom{\tilde{m}}{u_0} \leq r, \\
  \sum_{j = 0}^{\tilde{m}-1} \Hhndom{\tilde{m}-j-1}{\partial_t^j \tilde{f}(t_0)} + \Hhndom{\tilde{m}}{\tilde{u}_0} \leq r.
 \end{align*}
  Then the function $S_{\chi,\sigma,G,m,p}(t_0,f,u_0)$ belongs to $\Hhdom{m-p}$ and for 
  a constant $C_1 = C_1(\chi,\sigma,m,r, \mathcal{U}_1)$ 
  it satisfies
  \begin{align*}
    \Hhndom{m-p}{S_{\chi,\sigma,G,m,p}(t_0,f,u_0)} \leq C_1 \Big( \sum_{j = 0}^{m-1} \Hhndom{m-j-1}{\partial_t^j f(t_0)} + \Hhndom{m}{u_0} \Big)
  \end{align*}
  for all $p \in \{0, \ldots, m\}$, where $\mathcal{U}_1$ is a compact subset of $\mathcal{U}$ with 
  $\image u_0 \subseteq \mathcal{U}_1$.
  Moreover, there is a constant 
  $C_2 = C_2(\chi,\sigma, m, r,\mathcal{U}_2)$ such that 
  \begin{align*}
   &\Hhndom{m-p}{S_{\chi,\sigma,G,m,p}(t_0,f,u_0) - S_{\chi,\sigma,G,m,p}(t_0, \tilde{f}, \tilde{u}_0)} \\
   &\leq C_2 \Big(  \sum_{j = 0}^{m-1} \Hhndom{m-j-1}{\partial_t^j f(t_0) - \partial_t^j \tilde{f}(t_0)} + \Hhndom{m}{u_0 - \tilde{u}_0} \Big)
  \end{align*}
  for all $p \in \{0,\ldots,m\}$, where $\mathcal{U}_2$ is a compact subset of $\mathcal{U}$ with 
  $\image u_0, \image \tilde{u}_0 \subseteq \mathcal{U}_2$.
\end{lem}
In~\cite{SpitzMaxwellLinear} a solution in $\Gdom{m}$ of the linear problem~\eqref{EquationIBVPIntroduction} 
was constructed for boundary data from the space
\begin{align*}
 &\Edom{m} = \bigcap_{j = 0}^m H^j(J, H^{m + \frac{1}{2} - j}(\partial G)), \\
 &\Enormwgdom{m}{g} = \max_{0 \leq j \leq m} \|\partial_t^j g\|_{L^2(J, H^{m + 1/2 - j}(\partial G))}.
\end{align*}
We thus also
take boundary data $g \in \Edom{m}$ for the nonlinear problem~\eqref{EquationNonlinearIBVP}. But then we can differentiate the boundary condition 
in~\eqref{EquationNonlinearIBVP} up to $m-1$-times in time and still evaluate in $t_0$ if $u$ belongs to $\Gdom{m}$. In 
combination with~\eqref{EquationTimeDerivativesOfSolutionEqualSChiSigmaG} we deduce the identities
\begin{equation}
\label{EquationNonlinearCompatibilityConditions}
	B S_{\chi, \sigma, G, m,p}(t_0, f, u_0) = \partial_t^p g(t_0) \quad \text{on } \partial G
\end{equation}
for all $p \in \{0, \ldots, m-1\}$, which are thus necessary conditions for the existence of a 
$\Gdom{m}$-solution of~\eqref{EquationNonlinearIBVP}. 
We say that the data tuple $(\chi, \sigma, t_0, B, f, g, u_0)$ fulfills the \emph{compatibility conditions}
of order $m$ if $\overline{\image u_0} \subseteq \mathcal{U}$ and~\eqref{EquationNonlinearCompatibilityConditions} holds for 
all $p \in \{0, \ldots, m-1\}$.
In the next lemma we relate the operators $S_{\chi,\sigma,G,m,p}$ with their linear counterparts from~\cite{SpitzMaxwellLinear}. Therefore, we have to recall some notation. In~\cite{SpitzMaxwellLinear} we introduced 
the function spaces
\begin{align*}
 &\Fdomk{m}{k} = \{A \in W^{1,\infty}(J \times G)^{k \times k} \colon  \partial^\alpha A \in L^\infty(J , \Ltwohdom)^{k \times k} 
 \text{ for all } \alpha \in \N_0^4 \\
 &\hspace{18em} \text{with } 1 \leq |\alpha| \leq m\}, \\
 &\Fdomnorm{m}{A} = \max\{\|A\|_{W^{1,\infty}(J \times G)}, \max_{1 \leq |\alpha| \leq m} \|\partial^\alpha A\|_{L^\infty(J, \Ltwohdom)}\},\\
 &\Fvarwkdom{m}{k} = \{A \in L^\infty(G)^{k \times k} \colon \partial^\alpha A \in \Ltwodomh^{k \times k} \text{ for all } \alpha \in \N_0^3  \text{ with }  1\leq |\alpha| \leq m\}, \\
  &\Fvarnormdom{m}{A} = \max\{\|A\|_{L^\infty(G)}, \max_{1 \leq |\alpha| \leq m}\|\partial^\alpha A\|_{\Ltwohdom}\}
\end{align*}
for $k \in \N$. 
Moreover, by $\Fdompdk{m}{\eta}{k}$ we mean those functions $A$ from $\Fdomk{m}{k}$ with 
$A(t,x) = A(t,x)^T \geq \eta$ for all $(t,x) \in J \times G$ and by $\Fdomupdwl{m}{c}{k}$ those 
which have a limit as $|(t,x)| \rightarrow \infty$.
If it is clear from the context which parameter $k$ we consider, we suppress it in the notation.
\begin{rem}
	\label{RemarkChiuInFm}
	As noted in Remark~1.2 in~\cite{SpitzMaxwellLinear} the linear theory allows for coefficients in $W^{1,\infty}(J \times G)$ 
	whose derivatives up to order $m$ are contained in $L^\infty(J, \Ltwohdom) + L^\infty(J \times G)$. In view of 
	Lemma~\ref{LemmaHigherOrderChainRule}, we can thus apply the linear theory with coefficients $\chi(\hat{u})$ and 
	$\sigma(\hat{u})$ and $\hat{u} \in \Gdomvar{\tilde{m}}$.
	However, the part of the derivatives in $L^\infty(J \times G)$ is easier to treat so that we concentrated on coefficients 
	from $\Fdom{m}$ in~\cite{SpitzMaxwellLinear}. The same is true for the nonlinear problem. In the proofs we will therefore 
	assume without loss of generality that $\chi$ and $\sigma$ from $\ml{m,6}{G}{\mathcal{U}}$ have 
	decaying space derivatives as $|x| \rightarrow \infty$.  More precisely, for 
	all multiindices $\alpha \in \N_0^9$ with $\alpha_4 = \ldots = \alpha_9 = 0$ and  $|\alpha| \leq m$, $R > 0$, $\mathcal{U}_1 \Subset \mathcal{U}$,
	and $v \in L^\infty(J, \Ltwohdom)$ with $\image v \subseteq \mathcal{U}_1$ and 
	$\|v\|_{L^\infty(J, \Ltwohdom)} \leq R$ we have
	\begin{align}
	\label{EquationPropertyForL2}
		&(\partial^\alpha \chi)(v), (\partial^\alpha \sigma)(v) \in L^\infty(J, \Ltwohdom), \nonumber\\
		&\|(\partial^\alpha \chi)(v)\|_{L^\infty(J, \Ltwohdom)} + \|(\partial^\alpha \sigma)(v)\|_{L^\infty(J, \Ltwohdom)} \leq C,
	\end{align}
	where $C = C(\chi,\sigma,m,R,\mathcal{U}_1)$.
	With this assumption we obtain from Lemma~\ref{LemmaHigherOrderChainRule} that $\chi(\hat{u})$ and $\sigma(\hat{u})$ 
	belong to $\Fdom{m}$.
	
	Finally, we point out that if $G$ is bounded, the above considerations are unnecessary since $\Ltwohdom + L^\infty(G) = \Ltwohdom$ in this case.
\end{rem}
If one has a $\Gdom{m}$-solution $u$ of the linear problem~\eqref{EquationIBVPIntroduction} with coefficients 
$A_0 \in \Fdompdw{\tilde{m}}{\eta}$ with $\eta > 0$ and $\tilde{m} = \max\{m,3\}$, $A_1, A_2, A_3 \in \Fdom{\tilde{m}}$ independent of time,
and $D \in \Fdom{\tilde{m}}$ (note that we allow more general spatial coefficients here), the same 
reasoning as above first gives a closed expression for $\partial_t^p u(t_0)$ in terms of the coefficients and the data, 
which we denote by $S_{G,m,p}(t_0, A_0, \ldots, A_3, D, f, u_0)$, 
and compatibility conditions on the boundary. We refer to (2.2) and~(2.4) in~\cite{SpitzMaxwellLinear} for the precise notion. 
We then say that the tuple $(t_0, A_0, \ldots, A_3, D, B, f, g, u_0)$ fulfills the linear compatibility 
conditions of order $m$ if the equations
\begin{equation}
 \label{EquationCompatibilityConditionPrecised}
 B S_{G,m,p}(t_0, A_0, \ldots, A_3, D, f, u_0) = \partial_t^p g(t_0) \quad \text{on } \partial G
\end{equation}
are satisfied for all $p \in \{0, \ldots, m-1\}$.
Since we want to apply the linear theory with coefficients $\chi(\hat{u})$ and $\sigma(\hat{u})$, we have to know in which way 
the compatibility conditions~\eqref{EquationNonlinearCompatibilityConditions} for the nonlinear problem imply the compatibility conditions~\eqref{EquationCompatibilityConditionPrecised} for the
resulting linear problem.
\begin{lem}
 \label{LemmaCorrespondenceLinearNonlinearInZero}
 Let $J \subseteq \R$ be an open interval, $t_0 \in \clJ$, and $m \in \N$ with $m \geq 3$.
 Take  $\chi \in \mlpd{m,6}{G}{\mathcal{U}}$ and $\sigma \in \ml{m,6}{G}{\mathcal{U}}$.
 Choose data $f \in \Hadom{m}$ and $u_0 \in \Hhdom{m}$ such that $\overline{\image u_0}$ is 
 contained in $\mathcal{U}$. Let $r  > 0$.
 Assume that $f$ and $u_0$ satisfy
 \begin{align*}
  \begin{aligned}
    &\Hhndom{m}{u_0} \leq r, \quad &&\max_{0 \leq j \leq m-1} \Hhndom{m-j-1}{\partial_t^j f(t_0)} \leq r, \\
    &\Gdomnormwg{m-1}{f} \leq r, &&\Handom{m}{f} \leq r.
  \end{aligned}
 \end{align*}
 \begin{enumerate}
  \item \label{ItemConnectionLinearNonlinerOperatorS}
	Let $\hat{u} \in \Gdomvar{m}$ with $\partial_t^p \hat{u}(t_0) = S_{\chi,\sigma,G,m,p}(t_0,f,u_0)$ 
	for $0 \leq p \leq m - 1$. Then $\hat{u}$ fulfills the equations
	\begin{align}
	    \label{EquationConnectionLinearAndNonlinearCompatibilityConditions}
	    S_{G,m,p}(t_0,\chi(\hat{u}),A_1^{\operatorname{co}}, A_2^{\operatorname{co}},A_3^{\operatorname{co}}, \sigma(\hat{u}), f, u_0) = S_{\chi,\sigma,G, m, p}(t_0, f, u_0)
	\end{align}
	for all $p \in \{0,\ldots,m\}$.
  \item \label{ItemNonEmptyFixedPointSpace}
	There is a constant $C(\chi,\sigma,m,r,\mathcal{U}_1) > 0$ 
	and a function $u$ in $\Gdom{m}$ realizing the initial conditions 
	\begin{align*}
	  \partial_t^p u(t_0) = S_{\chi, \sigma, G, m, p}(t_0,f,u_0)
	\end{align*}
	for all $p \in \{0,\ldots,m\}$ and it is bounded by
	\begin{align*}
	    \Gdomnormwg{m}{u} \leq C(\chi,\sigma,m,r,\mathcal{U}_1) \Big(\sum_{j = 0}^{m-1} \Hhndom{m-j-1}{\partial_t ^j f(t_0)} + \Hhndom{m}{u_0} \Big).
	\end{align*}   
	Here $\mathcal{U}_1$ denotes a compact subset of $\mathcal{U}$ with $\image u_0 \subseteq \mathcal{U}_1$.
 \end{enumerate}
\end{lem}
\begin{proof}
  \ref{ItemConnectionLinearNonlinerOperatorS} Assertion~\ref{ItemConnectionLinearNonlinerOperatorS} follows by 
 induction from the definition of the 
 operators $S_{G,m,p}$ in~\cite[(2.2)]{SpitzMaxwellLinear}, Lemma~\ref{LemmaHigherOrderChainRule}, and the 
 definition of $S_{\chi,\sigma,G,m,p}$ in~\eqref{EquationDefinitionHigherOrderInitialValuesNonlinear}.
 
 \ref{ItemNonEmptyFixedPointSpace}
 Since $S_{\chi, \sigma, G, m, p}(t_0,f,u_0)$ belongs to $\Hhdom{m-p}$ for all $p \in \{0, \ldots, m\}$, an extension theorem (see e.g. Lemma~2.34 in~\cite{SpitzDissertation}) yields the existence of a function $u$ with 
 $ \partial_t^p u(t_0) = S_{\chi, \sigma, G, m, p}(t_0,f,u_0)$ for all $p \in \{0, \ldots, m\}$ and
\begin{equation*}
 \Gdomnormwg{m}{u} \leq C \sum_{p = 0}^m \Hhndom{m-p}{S_{\chi, \sigma, G, m, p}(t_0,f,u_0)}.
\end{equation*} 
 Lemma~\ref{LemmaNonlinearHigherOrderInitialValues} then implies the assertion.
\end{proof}
 
 \section{Local existence}
 \label{SectionLocalExistence}
 In this section we prove existence and uniqueness of a solution of~\eqref{EquationNonlinearIBVP} by 
 a fixed point argument based on the a priori estimates and the regularity theory from~\cite{SpitzMaxwellLinear} 
 for the corresponding linear problem. By a solution 
of the nonlinear problem~\eqref{EquationNonlinearIBVP} we mean a function $u$ which belongs to 
$\bigcap_{j = 0}^m C^j(I, \Hhdom{m-j})$ 
with $\overline{\image u(t)} \subseteq \mathcal{U}$ for all $t \in I$ and which satisfies~\eqref{EquationNonlinearIBVP}. 
Here $I$ is an interval with $t_0 \in I$.
Since the main result from~\cite{SpitzMaxwellLinear} is 
 omnipresent in the following, we recall it in Theorem~\ref{TheoremExistenceAndUniquenessOnDomain} below.
 Prior to this, we want to stress that in~\cite{SpitzMaxwellLinear} the initial boundary value 
 problem~\eqref{EquationIBVPIntroduction} on general domains $G$ was reduced to a half-space problem 
 via local charts. The localization procedure and a subsequent transform lead to the study of
  \begin{equation}
  \label{IBVP}
\left\{\begin{aligned}
   A_0 \partial_t u + \sum_{j=1}^3 A_j \partial_j u + D u  &= f, \quad &&x \in \R^3_+, \quad &t \in J; \\
   B u &= g, \quad &&x \in \partial \R^3_+, &t \in J; \\
   u(t_0) &= u_0, \quad &&x \in \R^3_+;
\end{aligned}\right.
\end{equation}
with coefficients $A_0 \! \in \! \Fupdwlk{\tilde{m}}{\operatorname{c}}{\eta}{6}$, $D \! \in \! \Fuwlk{\tilde{m}}{c}{6}$, 
$A_3 = A_3^{\operatorname{co}}$, and $A_1, A_2 \!\in \! \Fcoeff{\tilde{m}}{cp}$, where 
\begin{align*}
 \Fcoeff{m}{cp} := \{&A \in \Fuwlk{m}{cp}{6} \colon \exists \mu_1, \mu_2, \mu_3 \in \Fuwlk{m}{cp}{1} 
 \text{ independent of time, } \\
 &\text{ constant outside of a compact set such that } A = \sum_{j = 1}^3 A_j^{\operatorname{co}} \mu_j\}.
\end{align*}
Moreover, in the boundary condition we have $B = B^{\operatorname{co}}$, where $B^{\operatorname{co}}$ 
is a constant matrix in $\R^{2 \times 6}$ with rank $2$. There further exists another constant matrix 
$C^{\operatorname{co}} \in \R^{2 \times 6}$ with rank $2$ such that
\begin{equation}
\label{EquationDecompositionOfA3co}
 A_3^{\operatorname{co}} = \frac{1}{2} \Big(C^{\operatorname{co} T} B^{\operatorname{co}} + B^{\operatorname{co} T} C^{\operatorname{co}} \Big).
\end{equation}
We refer to~\cite[Section~2]{SpitzMaxwellLinear} and~\cite[Chapter~5]{SpitzDissertation} for the details.

The main result from~\cite{SpitzMaxwellLinear} shows that the linear initial boundary value problem~\eqref{EquationIBVPIntroduction} 
respectively~\eqref{IBVP} has a unique solution in $\Gdom{m}$ if the coefficients and data are accordingly 
regular and compatible. Moreover, the $\Gdom{m}$-norm of the solution can be estimated by the 
corresponding norms of the data.
 \begin{theorem}
  \label{TheoremExistenceAndUniquenessOnDomain}
  Let $\eta > 0$, $m \in \N_0$, and $\tilde{m} = \max\{m,3\}$. Fix radii $r \geq  r_0 >0$. Take a domain $G$ 
  with compact $C^{\tilde{m} + 2}$-boundary or $G = \R^3_+$.
  Choose times $t_0 \in \R$, $T' > 0$ and $T \in (0, T')$ and set 
  $J = (t_0, t_0 + T)$. Take coefficients $A_0 \in \Fdomupdwl{\tilde{m}}{c}{\eta}$,   
  $D \in \Fdomuwl{\tilde{m}}{c}$, and $A_3 = A_3^{\operatorname{co}}$. If $G = \R^3_+$, pick 
  $A_1, A_2 \in \Fcoeff{\tilde{m}}{cp}$. Otherwise, let $A_1 = A_1^{\operatorname{co}}$ and 
  $A_2 = A_2^{\operatorname{co}}$. Assume the bounds
  \begin{align*}
    &\Fnormdom{\tilde{m}}{A_i} \leq r, \quad \Fnormdom{\tilde{m}}{D} \leq r, \\
    &\max \{\Fvarnormdom{\tilde{m}-1}{A_i(t_0)},\max_{1 \leq j \leq \tilde{m}-1} \Hhndom{\tilde{m}-j-1}{\partial_t^j A_0(t_0)}\} \leq r_0, \\
    &\max \{\Fvarnormdom{\tilde{m}-1}{D(t_0)},\max_{1 \leq j \leq \tilde{m}-1} \Hhndom{\tilde{m}-j-1}{\partial_t^j D(t_0)}\} \leq r_0,
  \end{align*}
  for all $i \in \{0, 1, 2\}$. Set $B = B^{\operatorname{co}}$ if $G = \R^3_+$ and define $B$ 
  as in~\eqref{EquationDefinitionB} else. Choose data $f \in \Hadom{m}$, $g \in \Edom{m}$, and $u_0 \in \Hhdom{m}$ 
  such that the tuple $(t_0, A_0, \ldots, A_3, D, B, f, g,u_0)$
  fulfills the linear compatibility conditions~\eqref{EquationCompatibilityConditionPrecised} of 
  order $m$. 
   
  Then the linear initial boundary value problem~\eqref{IBVP} respectively~\eqref{EquationIBVPIntroduction}
  has a unique solution $u$ in 
  $\Gdom{m}$. Moreover,  there is a number
 $\gamma_m = \gamma_m(\eta, r, T') \geq 1$ such that
 \begin{align}
  &\Gdomnorm{m}{u}^2  \leq (C_{m,0} + T C_m) e^{m C_1 T} \Big(  \sum_{j = 0}^{m-1} \Hhndom{m-1-j}{\partial_t^j f(t_0)}^2 + \Enormdom{m}{g}^2  \nonumber\\
      &\hspace{18em} + \Hhndom{m}{u_0}^2 \Big) + \frac{C_m}{\gamma}  \Hangammadom{m}{f}^2    \nonumber
 \end{align}
 for all $\gamma \geq \gamma_m$, where $C_i = C_i(\eta,  r,T') \geq 1$ and $C_{i,0} = C_{i,0}(\eta,r_0) \geq 1$ 
 for $i \in \{1,m\}$.
 \end{theorem}
 We point out that the scope of~\cite{SpitzMaxwellLinear} was to provide the regularity theory for~\eqref{EquationIBVPIntroduction}, 
 building up on the $L^2$-theory from~\cite{Eller}. The case $m = 0$ in Theorem~\ref{TheoremExistenceAndUniquenessOnDomain} 
 is already contained in~\cite{Eller}.
 We note that we need a further assumption on our material laws $\chi$ and $\sigma$ to guarantee that 
 $\chi(\hat{u})$ and $\sigma(\hat{u})$ have a limit at infinity, which is required in 
 Theorem~\ref{TheoremExistenceAndUniquenessOnDomain}. We therefore define
 \begin{align*}
  &\mlwl{m,n}{G}{\mathcal{U}}{c} = \{\theta \in \ml{m,n}{G}{\mathcal{U}} \colon \exists A \in \R^{n \times n} \text{ such that for all } \\
      &\hspace{8.5em} (x_k, y_k)_k \in (G \times \mathcal{U})^{\N} \text{ with } |x_k| \rightarrow \infty \text{ and } y_k \rightarrow 0: \\
      &\hspace{8.5em} \theta(x_k, y_k) \rightarrow A \text{ as } k \rightarrow \infty\}, \\
   &\mlpdwl{m,n}{G}{\mathcal{U}}{c} = \mlpd{m,n}{G}{\mathcal{U}} \cap \mlwl{m,n}{G}{\mathcal{U}}{c}.
 \end{align*}
 We point out that $\mlwl{m,n}{G}{\mathcal{U}}{c}$ coincides with $\ml{m,n}{G}{\mathcal{U}}$ 
 if $G$ is bounded. 
 Let $\chi \in \mlpdwl{m,6}{G}{\mathcal{U}}{c}$ and $\sigma \in \mlwl{m,6}{G}{\mathcal{U}}{c}$ 
 satisfy~\eqref{EquationPropertyForL2} and take a function $\hat{u} \in \Gdomvar{\tilde{m}}$. 
 Lemma~\ref{LemmaHigherOrderChainRule} and Sobolev's embedding then imply that 
 $\chi(\hat{u})$ is an element of 
 $\Fdomupdwl{\tilde{m}}{c}{\eta}$ for a number $\eta > 0$ and $\sigma(\hat{u})$ is contained in 
 $\Fdomuwl{\tilde{m}}{c}$.
 In the next lemma we prove the uniqueness of solutions of~\eqref{EquationNonlinearIBVP}. 
 \begin{lem}
 \label{LemmaUniquenessOfNonlinearSolution}
 Let $t_0 \in \R, T > 0$, and $J = (t_0, t_0 + T)$. Let $m \in \N$ 
 with $m \geq 3$. Take $\chi \in \mlpdwl{m,6}{G}{\mathcal{U}}{c}$ 
 and $\sigma \in \mlwl{m,6}{G}{\mathcal{U}}{c}$.
 Choose data $f \in \Hadom{m}$, $g \in \Edom{m}$, and $u_0 \in \Hhdom{m}$. 
 Let $u_1$ and $u_2$ be two solutions in $\Gdom{m}$ of~\eqref{EquationNonlinearIBVP} with inhomogeneity $f$, 
 boundary value $g$, and initial value $u_0$ at initial time $t_0$. Then $u_1 = u_2$.
\end{lem}
\begin{proof}
 As explained in Remark~\ref{RemarkChiuInFm},  we assume without loss of generality that $\chi$ and $\sigma$ 
 posess property~\eqref{EquationPropertyForL2}. 
 Set 
 \begin{align*}
  K = \{T_0 \in \clJ \colon u_1 = u_2 \text{ on } [t_0, T_0]\}.
 \end{align*}
 This set is nonempty since $u_1(t_0) = u_0 = u_2(t_0)$. Let $T_1 = \sup K$. The continuity of $u_1$ and 
 $u_2$ implies that the two functions coincide on $[t_0, T_1]$. 
 
 Since $u_1$ and $u_2$ are solutions of~\eqref{EquationNonlinearIBVP}, there is a compact 
 subset $\mathcal{U}_1 \subseteq \mathcal{U}$ such that $\image u_1, \image u_2 \subseteq \mathcal{U}_1$.
 We now assume that $T_1$ is not equal to $T$. We then take a time 
  $T_u \in (T_1, T]$ to be fixed below and we set $J_u = [T_1, T_u]$. 
  We observe   that $u_1$ and $u_2$ are both solutions of~\eqref{EquationNonlinearIBVP} 
 in $G_m(J_u \times G)$ with inhomogeneity $f$, boundary value $g$,  and initial value $u_1(T_1) = u_2(T_1)$. In particular, both functions 
 solve the linear initial boundary value problem~\eqref{EquationIBVPIntroduction} with data $f$, $g$, and $u_1(T_1)$ and differential 
 operator $L_1 := L(\chi(u_1),A_1^{\operatorname{co}}, A_2^{\operatorname{co}}, A_3^{\operatorname{co}}, \sigma(u_1))$ respectively 
 $L_2 := L(\chi(u_2),A_1^{\operatorname{co}}, A_2^{\operatorname{co}}, A_3^{\operatorname{co}}, \sigma( u_2))$. 
  Lemma~\ref{LemmaHigherOrderChainRule}  and Sobolev's embedding theorem
 yield that $\chi(u_1)$, $\chi(u_2)$, $\sigma(u_1)$, and $\sigma(u_2)$ are elements of $\Fdomuwl{3}{c}$. 
 Take $r > 0$ such that $\Gdomnormwgv{3}{u_1}{u} \leq r$.
  Lemma~\ref{LemmaHigherOrderChainRule} and Remark~\ref{RemarkChiuInFm} then provide 
 a radius $R = R(\chi,\sigma,r,\mathcal{U}_1)$ such that the bounds 
\begin{align*}
	&\max\{\|\chi(u_1)\|_{\Fdom{3}}, \|\sigma(u_1)\|_{\Fdom{3}} \} \leq R, \\
	&\max\{\Fvarnormdom{2}{\chi(u_1(T_1))}, \max_{1 \leq j \leq 2} \Hhndom{m-1-j}{\partial_t^j \chi(u_1)(T_1)}\} \leq R, \\
	&\max\{\Fvarnormdom{2}{\sigma(u_1(T_1))}, \max_{1 \leq j \leq 2} \Hhndom{m-1-j}{\partial_t^j \sigma(u_1)(T_1)}\} \leq R
\end{align*} 
 hold true.
 We further note that $\chi(u_1)$ is symmetric and uniformly positive definite.
 Therefore, Theorem~\ref{TheoremExistenceAndUniquenessOnDomain} for the differential operator $L_1$ can be applied 
 to $u_1 - u_2$. We take $\eta = \eta(\chi) > 0$ such that $\chi \geq \eta$ and set $\gamma = \gamma_{\ref{TheoremExistenceAndUniquenessOnDomain},0}(\eta, R)$, 
 where $\gamma_{\ref{TheoremExistenceAndUniquenessOnDomain},0}$ denotes the corresponding constant from Theorem~\ref{TheoremExistenceAndUniquenessOnDomain}.
 Theorem~\ref{TheoremExistenceAndUniquenessOnDomain} and Corollary~\ref{CorollaryEstimateForDifferenceHigherOrder}~\ref{ItemDifferenceInG} then show that
 \begin{align*}
  &\Gdomnormv{0}{u_1 - u_2}{u}^2 \leq C_{\ref{TheoremExistenceAndUniquenessOnDomain}}(\eta,R,T) \Ltwodomanv{L_1 u_1 - L_1 u_2}{u}^2 \\
  &= C(\chi,\sigma,r,T,\mathcal{U}_1)  \Ltwodomanv{f - \chi(u_1) \partial_t u_2 - \sigma(u_1) u_2 + \chi(u_2) \partial_t u_2 + \sigma(u_2) u_2 - f}{u}^2 \\
  &\leq C(\chi,\sigma,r,T,\mathcal{U}_1) (T_u - T_1) \|\partial_t u_2\|_{L^\infty(J_u \times G)}^2 \Gdomnormv{0}{\chi(u_1)- \chi(u_2)}{u}^2 \\
      &\quad + C(\chi,\sigma,r,T,\mathcal{U}_1)(T_u - T_1) \|u_2\|_{L^\infty(J_u \times G)}^2 \Gdomnormv{0}{\sigma(u_1)- \sigma(u_2)}{u}^2 \\
  &\leq C(\chi,\sigma,r,T, \mathcal{U}_1)  (\Gdomnormwgv{2}{\partial_t u_2}{u}^2 + \Gdomnormwgv{2}{u_2}{u}^2) (T_u - T_1) \Gdomnormv{0}{u_1 - u_2}{u}^2,
 \end{align*}
 where $C_{\ref{TheoremExistenceAndUniquenessOnDomain}}$ is the corresponding constant from Theorem~\ref{TheoremExistenceAndUniquenessOnDomain}. 
 Fixing the generic constant in the last line of the above estimate, we choose $T_u > T_1$ so small that 
 \begin{align*}
  C(\chi,\sigma,r,T,\mathcal{U}_1)  (\Gdomnormwgv{2}{\partial_t u_2}{u}^2 + \Gdomnormwgv{2}{u_2}{u}^2) (T_u - T_1) \leq \frac{1}{2}.
 \end{align*}
  Hence,
 \begin{align*}
  \Gdomnormv{0}{u_1 - u_2}{u} = 0,
 \end{align*}
 implying $u_1 = u_2$ on $[T_1, T_u]$ and thus on $[t_0, T_u]$. This result contradicts the definition 
 of $T_1$. We conclude that $T_1 = T$, i.e., $u_1 = u_2$ on $J$.
\end{proof}

 
 Finally, we can combine all the preparations in order to prove the local existence of solutions 
 of~\eqref{EquationNonlinearIBVP} using Banach's fixed point theorem. For the self-mapping and 
 the contraction property we heavily rely on Theorem~\ref{TheoremExistenceAndUniquenessOnDomain}. 
 Special care in the treatment of the constants is needed to close the argument. In particular, the 
 structure of the constants in Theorem~\ref{TheoremExistenceAndUniquenessOnDomain} is crucial here.
 \begin{theorem}
 \label{TheoremLocalExistenceNonlinear}
 Let $t_0 \in \R$, $T > 0$,  $J = (t_0,t_0 + T)$, and $m \in \N$ with $m \geq 3$. 
 Take $\chi \in \mlpdwl{m,6}{G}{\mathcal{U}}{c}$ and $\sigma \in \mlwl{m,6}{G}{\mathcal{U}}{c}$.
 Let \begin{align*}
   	B(x) = \begin{pmatrix}
             0 &\nu_3(x) &-\nu_2(x) &0 &0 &0 \\
      -\nu_3(x) &0 &\nu_1(x) &0 &0 &0 \\
      \nu_2(x) &-\nu_1(x) &0 &0 &0 &0 
            \end{pmatrix}
   \end{align*}
   for all $x \in \partial G$, where $\nu$ denotes the unit outer normal vector of $\partial G$.
 Choose an inhomogeneity $f \in \Hadom{m}$, boundary value $g \in \Edom{m}$, and initial value $u_0 \in \Hhdom{m}$ with $\overline{\image u_0} \subseteq \mathcal{U}$
 such that the tuple $(\chi, \sigma, t_0, B, f, g, u_0)$ fulfills 
 the nonlinear compatibility conditions~\eqref{EquationNonlinearCompatibilityConditions} of order $m$.
 Choose a radius $r > 0$
 satisfying
 \begin{align}
 \label{EquationDataSmallerRadiusInLocalExistenceTheorem}
  &\sum_{j = 0}^{m-1} \Hhndom{m-1-j}{\partial_t^j f(t_0)}^2 + \Enormwgdom{m}{g}^2 + \Hhndom{m}{u_0}^2 + \Handom{m}{f}^2 \leq r^2.
 \end{align}
 Take a number $\kappa > 0$ with 
 \begin{equation*}
  \dist(\{u_0(x) \colon x \in G\}, \partial \mathcal{U}) > \kappa.
 \end{equation*}
 Then there exists a time
 $\tau = \tau(\chi,\sigma,m,T,r,\kappa) > 0$ such that the nonlinear initial boundary value 
 problem~\eqref{EquationNonlinearIBVP} with data $f$, $g$, and $u_0$ 
 has a unique 
 solution $u$ on $[t_0, t_0 + \tau]$ which belongs to $G_m(J_\tau \times G)$, 
 where $J_\tau = (t_0, t_0 + \tau)$.
\end{theorem}
\begin{proof}
Without loss of generality we assume $t_0 = 0$ and that~\eqref{EquationPropertyForL2} holds true 
for $\chi$ and $\sigma$, cf. Remark~\ref{RemarkChiuInFm}. 
If $f = 0$, $g = 0$, and $u_0 = 0$, then $u = 0$ is a $\Gdom{m}$-solution of~\eqref{EquationNonlinearIBVP} 
and it is unique by Lemma~\ref{LemmaUniquenessOfNonlinearSolution}. So in the following we assume 
$\Handom{m}{f} + \Enormwgdom{m}{g} + \Hhndom{m}{u_0} > 0$.
Recall that the map $S_{\chi,\sigma,G,m,p}$ was defined 
in~\eqref{EquationDefinitionHigherOrderInitialValuesNonlinear} for $0 \leq p \leq m$.
Let $\tau \in (0,T]$. We set $J_\tau = (0, \tau)$ and 
\begin{equation}
\label{EquationDefinitionUkappa}
 \mathcal{U}_{\kappa} = \{y \in \mathcal{U} \colon \dist(y, \partial \mathcal{U}) \geq \kappa \} \cap \overline{B}_{2C_{\operatorname{Sob}} r}(0),
\end{equation}
where $C_{\operatorname{Sob}}$ denotes the constant for the Sobolev embedding from $\Hhdom{2}$ into $L^\infty(G)$. 
Then $\mathcal{U}_\kappa$ is compact and $\overline{\image u_0}$ is contained in $\mathcal{U}_\kappa$.

I) Let $R > 0$. We set
 \begin{align*}
  B_R(J_\tau) := \{v \in \GdomvarP{m} \colon &\GdomnormwgP{m}{v} \leq R, \, \|v - u_0\|_{L^\infty(J_\tau \times G)} \leq \kappa/2, \\
  &\partial_t^j v(0) = S_{\chi,\sigma,G,m,j}(0,f,u_0) \text{ for } 0 \leq j \leq m-1\}
 \end{align*}
 and equip it with the metric $d(v_1,v_2) = \GdomnormwgP{m-1}{v_1 - v_2}$. We first show that $B_R(J_\tau)$
 is a complete metric space. Recall that 
 $\GdomvarP{m}$ is continuously embedded in $\GdomP{m-1}$ so that $B_R(J_\tau)$ is well defined. Moreover, 
 Lemma~\ref{LemmaCorrespondenceLinearNonlinearInZero}~\ref{ItemNonEmptyFixedPointSpace} shows that 
 $B_R(J_\tau)$ is nonempty for all 
 $R > C_{\ref{LemmaCorrespondenceLinearNonlinearInZero}\ref{ItemNonEmptyFixedPointSpace}}( \chi, \sigma, m, r, \mathcal{U}_\kappa) \cdot (m+1)r $.
 
 Let $(v_n)_n$ be a Cauchy sequence in $(B_R(J_\tau), d)$. The functions $v_n$ then tend to $v$ in $\GdomP{m-1}$ as $n \rightarrow \infty$, 
 and hence $v$ satisfies $\partial_t^j v(0) = S_{\chi,\sigma,G,m,j}(0,f,u_0)$ for $0 \leq j \leq m-1$ 
 and $\GdomnormwgP{m-1}{v} \leq R$.
 Let $\alpha \in \N_0^4$ with $|\alpha| = m$. The 
 sequence $(\partial^\alpha v_n)_n$ is bounded in $L^\infty(J_\tau, \Ltwohdom) = (L^1(J_\tau, \Ltwohdom))^*$.
 The Banach-Alaoglu 
 Theorem thus gives a subsequence (again denoted by $(\partial^\alpha v_n)_n$) with 
 $\sigma^*$-limit $v_\alpha$
 in $L^\infty(J_\tau, \Ltwohdom)$. It is straightforward to check that 
 $\partial^\alpha v = v_\alpha$. In particular, $v$
 belongs to $\GdomvarP{m}$ with norm less or equal $R$.
 Finally, as $m \geq 3$, we infer 
 \begin{align*}
  \|v - u_0\|_{L^\infty(J_\tau \times G)} &\leq \|v - v_n\|_{L^\infty(J_\tau \times G)} + \|v_n - u_0\|_{L^\infty(J_\tau \times G)} \\
  &\leq C_{\operatorname{Sob}} \GdomnormwgP{2}{v - v_n} + \kappa \longrightarrow \kappa
 \end{align*}
 as $n \rightarrow \infty$.
 We conclude that $v$ again belongs to $B_R(J_\tau)$.
 
II) 
Let $\hat{u} \in B_R(J_\tau)$. Take $\eta  = \eta(\chi) > 0$ such that $\chi \geq \eta$.
Then $\chi(\hat{u})$ is contained in $F_{m,\eta}^{\operatorname{c}}(J_\tau \times G)$ 
and $\sigma(\hat{u})$ is an element of $F_{m}^{\operatorname{c}}(J_\tau \times G)$ 
by Lemma~\ref{LemmaHigherOrderChainRule}, Remark~\ref{RemarkChiuInFm},
and Sobolev's embedding. Lemma~\ref{LemmaCorrespondenceLinearNonlinearInZero}~\ref{ItemConnectionLinearNonlinerOperatorS} 
and the compatibility of $(\chi,\sigma,t_0,B,f,g,u_0)$ imply that the tuple 
$(t_0,\chi(\hat{u}), A_1^{\operatorname{co}}, A_2^{\operatorname{co}}, A_3^{\operatorname{co}},\sigma(\hat{u}),B,f,g,u_0)$ 
fulfills the linear compatibility conditions~\eqref{EquationCompatibilityConditionPrecised}.
Theorem~\ref{TheoremExistenceAndUniquenessOnDomain} yields a solution $u \in \GdomP{m}$ of 
the linear initial boundary value
problem~\eqref{EquationIBVPIntroduction} with differential operator $L(\chi(\hat{u}),A_1^{\operatorname{co}}, A_2^{\operatorname{co}}, A_3^{\operatorname{co}}, \sigma(\hat{u}))$, 
inhomogeneity $f$, boundary value $g$, and initial value $u_0$. 
 One thus defines a mapping $\Phi \colon \hat{u} \mapsto u$ from $B_R(J_\tau)$ to $\GdomP{m}$. We want to prove 
that $\Phi$ also maps $B_R(J_\tau)$ into $B_R(J_\tau)$ for a suitable radius $R$ and a sufficiently small time interval $J_\tau$.

For this purpose take numbers $\tau \in (0,T]$ and $R > C_{\ref{LemmaCorrespondenceLinearNonlinearInZero}\ref{ItemNonEmptyFixedPointSpace}}( \chi, \sigma, m, T,r) (m+1) r$ 
which will be fixed below. Let $\hat{u} \in B_R(J_\tau)$. We first note that there is a constant
$C_{\ref{LemmaNonlinearHigherOrderInitialValues}}(\chi,\sigma,m,r,\mathcal{U}_\kappa)$
such that
\begin{align}
\label{EquationBoundForLinearHigherOrderInitialValue}
 \Hhn{m-p}{S_{\chi,\sigma,G,m,p}(0,f,u_0)} \leq C_{\ref{LemmaNonlinearHigherOrderInitialValues}}(\chi,\sigma,m,r,\mathcal{U}_\kappa)
\end{align}
for all $p \in \{0,\ldots,m\}$ due to Lemma~\ref{LemmaNonlinearHigherOrderInitialValues}.
Lemma~\ref{LemmaHigherOrderChainRule}~\ref{ItemHigherOrderChainRuleInHh}  
further provides a constant $C_{\ref{LemmaHigherOrderChainRule}\ref{ItemHigherOrderChainRuleInHh}}$ such that
\begin{align*}
 \Fvarnormdom{m-1}{\chi(\hat{u})(0)} &= \Fvarnormdom{m-1}{\chi(u_0)} \leq C_{\ref{LemmaHigherOrderChainRule}\ref{ItemHigherOrderChainRuleInHh}}(\chi,m,r,\mathcal{U}_\kappa), \\
 \Fvarnormdom{m-1}{\sigma(\hat{u})(0)} &= \Fvarnormdom{m-1}{\sigma(u_0)} \leq C_{\ref{LemmaHigherOrderChainRule}\ref{ItemHigherOrderChainRuleInHh}}(\sigma,m,r,\mathcal{U}_\kappa).
\end{align*}
Note that $\overline{\image \hat{u}}$ is contained in the compact set
\begin{equation}
\label{EquationDefinitionTildeUkappa}
 \tilde{\mathcal{U}}_\kappa = \mathcal{U}_\kappa + \overline{B}(0,\kappa/2) \subseteq \mathcal{U}
\end{equation}
as $\hat{u} \in B_R(J_\tau)$.
From Lemma~\ref{LemmaHigherOrderChainRule}~\ref{ItemHigherOrderChainRuleMixed}
and~\eqref{EquationBoundForLinearHigherOrderInitialValue} we deduce the bound
\begin{align*}
 &\Hhndom{m-l-1}{\partial_t^l \chi(\hat{u})(0)} \leq C_{\ref{LemmaHigherOrderChainRule}\ref{ItemHigherOrderChainRuleMixed}}(\chi, m, \mathcal{U}_\kappa) (1 + \max_{0 \leq k \leq l} \Hhndom{m-k-1}{\partial_t^k \hat{u}(0)})^{m-1} \\
 &= C_{\ref{LemmaHigherOrderChainRule}\ref{ItemHigherOrderChainRuleMixed}}(\chi, m, \mathcal{U}_\kappa) (1 + \max_{0 \leq k \leq l} \Hhndom{m-k-1}{S_{\chi,\sigma,G,m,k}(0,f,u_0)})^{m-1} \\
 &\leq C_{\ref{LemmaHigherOrderChainRule}\ref{ItemHigherOrderChainRuleMixed}}(\chi, m, \mathcal{U}_\kappa) (1 + C_{\ref{LemmaNonlinearHigherOrderInitialValues}}(\chi,\sigma,m,r,\mathcal{U}_\kappa))^{m-1}, \\
 &\Hhndom{m-l-1}{\partial_t^l \sigma(\hat{u})(0)} \leq C_{\ref{LemmaHigherOrderChainRule}\ref{ItemHigherOrderChainRuleMixed}}(\sigma, m,  \mathcal{U}_\kappa) (1 + C_{\ref{LemmaNonlinearHigherOrderInitialValues}}(\chi,\sigma,m,r,\mathcal{U}_\kappa))^{m-1}
\end{align*}
for all $l \in \{1,\ldots,m-1\}$. We thus find a radius $r_0 = r_0(\chi,\sigma,m,r,\kappa)$ such that
\begin{align}
\label{EquationPrerequisitesInZeroForAPrioriEstimates}
 &\max\{\Fvarnormdom{m-1}{\chi(\hat{u})(0)}, \max_{1 \leq l \leq m-1} \Hhndom{m-l-1}{\partial_t^l \chi(\hat{u})(0)}\} \leq r_0, \nonumber \\
 &\max\{\Fvarnormdom{m-1}{\sigma(\hat{u})(0)}, \max_{1 \leq l \leq m-1} \Hhndom{m-l-1}{\partial_t^l \sigma(\hat{u})(0)}\} \leq r_0.
\end{align}
As $\hat{u}$ belongs to $B_R(J_\tau)$, Lemma~\ref{LemmaHigherOrderChainRule}~\ref{ItemHigherOrderChainRuleInG} 
gives
\begin{align*}
 \Fnormdom{m}{\chi(\hat{u})}, \Fnormdom{m}{\sigma(\hat{u})} &\leq C_{\ref{LemmaHigherOrderChainRule}\ref{ItemHigherOrderChainRuleInG}}(\chi,\sigma,m,\tilde{\mathcal{U}}_\kappa)(1+R)^m.
\end{align*}
We thus obtain a radius $R_1 = R_1(\chi, \sigma, m, R, \kappa)$ with 
\begin{align}
 \label{EquationPrerequisitesInGmForAPrioriEstimates}
 \Fnormdom{m}{\chi(\hat{u})} \leq R_1 \qquad \text{and} \qquad \Fnormdom{m}{\sigma( \hat{u})} \leq R_1.
\end{align}

We next define the constant $C_{m,0} = C_{m,0}(\chi,\sigma,r,\kappa)$ by
\begin{align}
\label{EquationConstantInZeroFromAPrioriEstimates}
 C_{m,0}(\chi,\sigma,r,\kappa) &= C_{\ref{TheoremExistenceAndUniquenessOnDomain},m,0}(\eta(\chi),r_0(\chi,\sigma,m,r,\kappa)),
\end{align}
where $C_{\ref{TheoremExistenceAndUniquenessOnDomain},m,0}$ denotes the constant $C_{m,0}$ from Theorem~\ref{TheoremExistenceAndUniquenessOnDomain}.
We set the radius $R = R(\chi,\sigma,m,r,\kappa)$ for $B_R(J_\tau)$ to be
\begin{align}
 \label{EquationRadiusForFixedPointSpace}
 R(\chi,\sigma,m,r,\kappa) = \max\Big\{\sqrt{6 \, C_{m,0}(\chi,\sigma,r,\kappa)}\, r, \, C_{\ref{LemmaCorrespondenceLinearNonlinearInZero}\ref{ItemNonEmptyFixedPointSpace}}( \chi, \sigma, m, r,\mathcal{U}_\kappa)(m+1) r + 1\Big\}.
\end{align}
We further introduce the constants
\begin{align}
 \gamma_m = \gamma_m(\chi, \sigma, T,r,\kappa) &:= \gamma_{\ref{TheoremExistenceAndUniquenessOnDomain},m}(\eta(\chi),R_1(\chi,\sigma,m,R(\chi,\sigma,m,r,\kappa),\kappa),T), \label{EquationGammaFromAPrioriEstimates} \\
 C_m = C_m(\chi,\sigma,T,r) &:= C_{\ref{TheoremExistenceAndUniquenessOnDomain},m}(\eta(\chi), R_1(\chi,\sigma,m,R(\chi,\sigma,m,r,\kappa),\kappa),T), \label{EquationConstantFromAPrioriEstimates}
\end{align}
where $\gamma_{\ref{TheoremExistenceAndUniquenessOnDomain},m}$ and 
$C_{\ref{TheoremExistenceAndUniquenessOnDomain},m}$ denote the corresponding constants from Theorem~\ref{TheoremExistenceAndUniquenessOnDomain}.
 Let
\begin{align*}
 C_{\ref{CorollaryEstimateForDifferenceHigherOrder}\ref{ItemDifferenceInG}}(\theta,m, R, \tilde{\mathcal{U}}_\kappa)
\end{align*}
denote the constant arising from the application of 
Corollary~\ref{CorollaryEstimateForDifferenceHigherOrder}~\ref{ItemDifferenceInG} 
to the components of $\theta \in \ml{m,6}{G}{\mathcal{U}}$.

With these constants at hand we define the parameter $\gamma = \gamma(\chi, \sigma, m, T, r,\kappa)$ and the time step $\tau = \tau(\chi, \sigma, m, T, r,\kappa)$ 
by
\begin{align}
 \gamma &= \max\Big\{\gamma_m, \,C_{m,0}^{-1} C_m \Big\} ,  \label{EquationDefinitionGammaForFixedPointArgument} \\
  \tau &= \min\Big\{T, \, (2 \gamma + m C_{\ref{TheoremExistenceAndUniquenessOnDomain},1})^{-1} \log{2}, \, C_m^{-1} C_{m,0}, (C_{\operatorname{Sob}} R)^{-1} \kappa, \label{EquationDefinitionOfTauForFixedPointArgument}\\
      &\hspace{4em} [32  R^2 C_{m,0} C_{P}^2(C_{\ref{CorollaryEstimateForDifferenceHigherOrder}\ref{ItemDifferenceInG}}^2(\chi,m, R, \tilde{\mathcal{U}}_\kappa) 
	+ C_{\ref{CorollaryEstimateForDifferenceHigherOrder}\ref{ItemDifferenceInG}}^2(\sigma,m,R,\tilde{\mathcal{U}}_\kappa))]^{-1}\Big\}, \nonumber 
\end{align}
where $C_{P}$ and $C_{\ref{TheoremExistenceAndUniquenessOnDomain},1}$ denote 
the corresponding constants from~\cite[Lemma~2.1]{SpitzMaxwellLinear} and Theorem~\ref{TheoremExistenceAndUniquenessOnDomain}
respectively. Observe that $\gamma$ and $\tau$ actually only depend on $\chi$, $\sigma$, $m$, $T$, $r$, and $\kappa$ as $C_{m,0}$, 
$C_m$, $C_{\ref{TheoremExistenceAndUniquenessOnDomain},1}$, and $R$ only depend on these quantities (see~\eqref{EquationConstantInZeroFromAPrioriEstimates} 
to~\eqref{EquationConstantFromAPrioriEstimates}).
For later reference we note that the choice of $\gamma$ and $\tau$ implies
\begin{align}
 &\gamma \geq \gamma_m, \label{EquationGammaLargerGammam}\\
 &\frac{C_m}{\gamma} \leq C_{m,0}, \label{EquationGammaAndConstants} \\
  &\tau \leq T \label{EquationTimeStepSmallerT}, \\
  &(2 \gamma + m C_{\ref{TheoremExistenceAndUniquenessOnDomain},1}) \, \tau \leq \log{2} \label{EquationExponentSmallerLog2}, \\
  &C_m \tau \leq C_{m,0} \label{EquationCmTauSmallerC0}, \\
  &C_{\operatorname{Sob}} R \tau \leq \kappa, \label{EquationCsobRtauSmallerKappa} \\
  &4 C_{m,0} C_{P}^2 C_{\ref{CorollaryEstimateForDifferenceHigherOrder}\ref{ItemDifferenceInG}}^2(\theta,m, 6, R,\tilde{\mathcal{U}}_\kappa) R^2 \, \tau \leq \frac{1}{8},
  \hspace{3em} \theta \in \{\chi,\sigma\}. \label{EquationTauLongTerm}
\end{align}

III) 
We want to bound the function $\Phi(\hat{u})$ from step~II) by means of Theorem~\ref{TheoremExistenceAndUniquenessOnDomain}. In view of the 
estimates~\eqref{EquationPrerequisitesInZeroForAPrioriEstimates} and~\eqref{EquationPrerequisitesInGmForAPrioriEstimates},
the definitions of $C_{m,0}$, $\gamma_m$, and $C_m$ in~\eqref{EquationConstantInZeroFromAPrioriEstimates}, \eqref{EquationGammaFromAPrioriEstimates}, 
and~\eqref{EquationConstantFromAPrioriEstimates}, respectively, fit to the assertion of Theorem~\ref{TheoremExistenceAndUniquenessOnDomain}.
Using also~\eqref{EquationGammaLargerGammam} and~\eqref{EquationTimeStepSmallerT}, we arrive at the inequality
\begin{align*}
 &\GdomnormwgP{m}{\Phi(\hat{u})}^2 \leq e^{2\gamma\tau} \GdomnormP{m}{\Phi(\hat{u})}^2 \\
 &\leq (C_{m,0} + \tau C_m) e^{(2\gamma + m C_{\ref{TheoremExistenceAndUniquenessOnDomain},1}) \tau}\Big(\sum_{j = 0}^{m-1} \Hhndom{m-1-j}{\partial_t^j f(0)}^2  \\
 &\hspace{2em} + \EnormdomP{m}{g}^2+ \Hhndom{m}{u_0}^2\Big)  + \frac{C_m}{\gamma} e^{2\gamma  \tau} \HangammaPdom{m}{f}^2.
\end{align*}
Observe that~\eqref{EquationDataSmallerRadiusInLocalExistenceTheorem} yields
\begin{align*}
  \HangammaPdom{m}{f}^2 \leq \HanPdom{m}{f}^2 \leq \Handom{m}{f}^2 \leq r^2
\end{align*}
and analogously $\EnormdomP{m}{g}^2 \leq \Enormwgdom{m}{g}^2$. Employing~\eqref{EquationCmTauSmallerC0}, \eqref{EquationGammaAndConstants}, 
\eqref{EquationExponentSmallerLog2}, \eqref{EquationDataSmallerRadiusInLocalExistenceTheorem}, and~\eqref{EquationRadiusForFixedPointSpace}, 
we then deduce the inequalities
\begin{align*}
 \GdomnormwgP{m}{\Phi(\hat{u})}^2 &\leq (C_{m,0} + C_{m,0}) e^{\log{2}}  r^2 + C_{m,0} e^{\log{2}} r^2 
 = 6 C_{m,0} r^2 \leq R^2, \\
 \GdomnormwgP{m}{\Phi(\hat{u})} &\leq R.
\end{align*}

Since $\Phi(\hat{u})$ belongs to $\GdomP{m}$, identity~(2.1) in~\cite{SpitzMaxwellLinear} (which is 
the linear counterpart to~\eqref{EquationTimeDerivativesOfSolutionEqualSChiSigmaG})
shows that 
\begin{align*}
    \partial_t^p \Phi(\hat{u})(0) = S_{G,m,p}(0, \chi(\hat{u}),A_1^{\operatorname{co}}, A_2^{\operatorname{co}},A_3^{\operatorname{co}}, \sigma(\hat{u}), f, u_0)
\end{align*}    
for all $p \in \{0,\ldots,m\}$. On the other hand, as an element of $B_R(J_\tau)$, the function $\hat{u}$ satisfies 
$\partial_t^p \hat{u}(0) = S_{\chi,\sigma,G,m,p}(0,f,u_0)$ for all $p \in \{0,\ldots,m-1\}$. 
Lemma~\ref{LemmaCorrespondenceLinearNonlinearInZero}~\ref{ItemConnectionLinearNonlinerOperatorS} thus yields
\begin{align*}
 \partial_t^p \Phi(\hat{u})(0) = S_{G,m,p}(0, \chi(\hat{u}), A_1^{\operatorname{co}}, A_2^{\operatorname{co}},A_3^{\operatorname{co}}, \sigma(\hat{u}), f, u_0) = S_{\chi,\sigma,G,m,p}(0,f,u_0)
\end{align*}
for all $p \in \{0,\ldots,m-1\}$.
We further  estimate
 \begin{align*}
  &\|\Phi(\hat{u}) - u_0\|_{L^\infty(J_\tau \times G)} = \Big\| \Phi(\hat{u})(0) + \int_0^t \partial_t \Phi(\hat{u})(s) ds - u_0 \Big\|_{L^\infty(J_\tau \times G)} \\
  &=  \Big\| \int_0^t \partial_t \Phi(\hat{u})(s) ds \Big\|_{L^\infty(J_\tau \times G)} 
  \leq C_{\operatorname{Sob}} \sup_{t \in (0,\tau)} \int_{0}^t \Hhndom{2}{\partial_t \Phi(\hat{u})(s)} ds \\
  &\leq C_{\operatorname{Sob}} \tau \GdomnormwgP{2}{\partial_t \Phi(\hat{u})} \leq C_{\operatorname{Sob}} \tau R \leq \kappa
 \end{align*}
 for all $\hat{u} \in B_R(J_\tau)$, where we used that $\Phi(\hat{u})(0) = u_0$ for $\hat{u} \in B_R(J_\tau)$
 and~\eqref{EquationCsobRtauSmallerKappa}.
We conclude that $\Phi(\hat{u})$ belongs to $B_R(J_\tau)$, i.e., $\Phi$ maps $B_R(J_\tau)$ into itself.

IV) Let $\hat{u}_1, \hat{u}_2 \in B_R(J_\tau)$. Since
$\chi(\hat{u}_i)$ and $\sigma(\hat{u}_i)$ belong to 
$\FdomP{m}$ for $i \in \{1,2\}$, Lemma~2.1 of~\cite{SpitzMaxwellLinear} implies that $\chi(\hat{u}_i) \partial_t \Phi(\hat{u}_2)$ 
and $\sigma(\hat{u}_i) \Phi(\hat{u}_2)$ are elements of $\GdomvarP{m-1} \hookrightarrow \HaPdom{m-1}$ for $i \in \{1,2\}$.
The function $\Phi(\hat{u}_2)$ thus fulfills 
\begin{align*}
&L(\chi(\hat{u}_1),A_1^{\operatorname{co}}, A_2^{\operatorname{co}},A_3^{\operatorname{co}}, \sigma(\hat{u}_1)) \Phi(\hat{u}_2) \\
&= \chi(\hat{u}_1) \partial_t \Phi(\hat{u}_2) + \sigma(\hat{u}_1) \Phi(\hat{u}_2) - \chi(\hat{u}_2) \partial_t \Phi(\hat{u}_2) - \sigma(\hat{u}_2) \Phi(\hat{u}_2)  \\
  &\hspace{2em} + L(\chi(\hat{u}_2),A_1^{\operatorname{co}}, A_2^{\operatorname{co}},A_3^{\operatorname{co}}, \sigma(\hat{u}_2)) \Phi(\hat{u}_2) \\
&= (\chi(\hat{u}_1) - \chi(\hat{u}_2)) \partial_t \Phi(\hat{u}_2) + (\sigma(\hat{u}_1) - \sigma(\hat{u}_2)) \Phi(\hat{u}_2) + f
\end{align*}
and this function belongs to  $\GdomvarP{m-1} \hookrightarrow \HaPdom{m-1}$.
We further stress that $\Phi(\hat{u}_1)(0) = u_0 = \Phi(\hat{u}_2)(0)$.

As in step III), properties~\eqref{EquationPrerequisitesInZeroForAPrioriEstimates}, \eqref{EquationPrerequisitesInGmForAPrioriEstimates},
\eqref{EquationConstantInZeroFromAPrioriEstimates}, \eqref{EquationGammaFromAPrioriEstimates}, \eqref{EquationConstantFromAPrioriEstimates}, 
 \eqref{EquationGammaLargerGammam}, and~\eqref{EquationTimeStepSmallerT} 
allow us to apply Theorem~\ref{TheoremExistenceAndUniquenessOnDomain} with differential operator 
$L(\chi(\hat{u}_1),A_1^{\operatorname{co}}, A_2^{\operatorname{co}},A_3^{\operatorname{co}}, \sigma( \hat{u}_1))$ 
and parameter $\gamma$ on $J_\tau \times G$. We thus obtain the inequality
\begin{align*}
 &\GdomnormwgP{m-1}{\Phi(\hat{u}_1) - \Phi(\hat{u}_2)}^2 \leq e^{2 \gamma \tau} \GdomnormP{m-1}{\Phi(\hat{u}_1) - \Phi(\hat{u}_2)}^2 \\
 &\leq (C_{m,0} + \tau C_m) e^{(2\gamma + m C_{\ref{TheoremExistenceAndUniquenessOnDomain},1})\tau} \sum_{j=0}^{m-2} \Hhndom{m-2-j}{\partial_t^j(f - L \Phi(\hat{u}_2))(0)}^2 \\
 &\hspace{3em} + \frac{C_m}{\gamma} e^{(2\gamma + m C_{\ref{TheoremExistenceAndUniquenessOnDomain},1})\tau} \HangammaPdom{m-1}{f - L \Phi(\hat{u}_2)}^2 \\
 &= (C_{m,0} + \tau C_m) e^{(2\gamma + m C_{\ref{TheoremExistenceAndUniquenessOnDomain},1})\tau} \sum_{j=0}^{m-2}\Hhndom{m-2-j}{\partial_t^j((\chi(\hat{u}_1) - \chi(\hat{u}_2)) \partial_t \Phi(\hat{u}_2))(0) \\
 &\hspace{15em} + \partial_t^j((\sigma(\hat{u}_1) - \sigma(\hat{u}_2)) \Phi(\hat{u}_2))(0)}^2 \\
 &\hspace{2em} + \frac{C_m}{\gamma} e^{(2\gamma + m C_{\ref{TheoremExistenceAndUniquenessOnDomain},1})\tau} \HangammaPdom{m-1}{(\chi(\hat{u}_1) - \chi(\hat{u}_2)) \partial_t \Phi(\hat{u}_2) \\
 &\hspace{15em}+ (\sigma(\hat{u}_1) - \sigma(\hat{u}_2)) \Phi(\hat{u}_2)}^2.
\end{align*}
Lemma~\ref{LemmaHigherOrderChainRule} and the equalities
\begin{align*}
 \partial_t^l \hat{u}_1(0) = S_{\chi, \sigma,G, m,l}(0,f,u_0) = \partial_t^l \hat{u}_2(0)
\end{align*}
for all $l \in \{0, \ldots, m-1\}$ imply that the terms in the sum vanish.
Employing~\eqref{EquationExponentSmallerLog2}, we then deduce
\begin{align}
\label{EquationContractionFirst}
 \GdomnormwgP{m-1}{\Phi(\hat{u}_1) - \Phi(\hat{u}_2)}^2 &\leq 
 4 C_m \frac{1}{\gamma}  \HangammaPdom{m-1}{(\chi(\hat{u}_1) - \chi(\hat{u}_2)) \partial_t \Phi(\hat{u}_2)}^2 \nonumber\\
 &\hspace{1em}+ 4 C_m  \frac{1}{\gamma} \HangammaPdom{m-1}{(\sigma(\hat{u}_1) - \sigma(\hat{u}_2)) \Phi(\hat{u}_2)}^2 \nonumber\\
 &=: I_1 + I_2.
\end{align}
Before going on, we point out that we know from step II) that $\Phi(\hat{u}_2)$ is an element of $B_R(J_\tau)$ and 
hence
\begin{align}
\label{EquationEstimateForGnormOfTimeDerivative}
 \GdomnormwgP{m-1}{\partial_t \Phi(\hat{u}_2)} \leq \GdomnormwgP{m}{\Phi(\hat{u}_2)} \leq R.
\end{align}
We now treat the first summand. Lemma~2.1 of~\cite{SpitzMaxwellLinear}, 
estimate~\eqref{EquationEstimateForGnormOfTimeDerivative}, and Corollary~\ref{CorollaryEstimateForDifferenceHigherOrder}~\ref{ItemDifferenceInG}
show that
\begin{align*}
 I_1 &\leq 4 C_m \frac{1}{\gamma} \tau \GdomnormP{m-1}{(\chi(\hat{u}_1) - \chi(\hat{u}_2)) \partial_t \Phi(\hat{u}_2)}^2 \\
      &\leq 4 C_m \frac{1}{\gamma} \tau C_{P}^2 \GdomnormP{m-1}{\chi(\hat{u}_1) - \chi(\hat{u}_2)}^2 \GdomnormwgP{m-1}{\partial_t \Phi(\hat{u}_2)}^2 \\
      &\leq 4 C_m \frac{1}{\gamma} C_{P}^2  C_{\ref{CorollaryEstimateForDifferenceHigherOrder}\ref{ItemDifferenceInG}}^2(\chi,m,R,\tilde{\mathcal{U}}_\kappa) R^2 \tau \GdomnormP{m-1}{\hat{u}_1 - \hat{u}_2}^2.
\end{align*}
Exploiting~\eqref{EquationGammaAndConstants} and~\eqref{EquationTauLongTerm}, we finally arrive at
\begin{align}
 \label{EquationEstimateForContractionThirdSummand}
 I_1 &\leq \frac{1}{8} \GdomnormP{m-1}{\hat{u}_1 - \hat{u}_2}^2 \leq \frac{1}{8} \GdomnormwgP{m-1}{\hat{u}_1 - \hat{u}_2}^2.
\end{align}
Analogously, we obtain 
\begin{align}
 \label{EquationEstimateForContractionFourthSummand}
 I_2 \leq \frac{1}{8} \GdomnormwgP{m-1}{\hat{u}_1 - \hat{u}_2}^2.
\end{align}
Estimates~\eqref{EquationContractionFirst}, \eqref{EquationEstimateForContractionThirdSummand}, 
and~\eqref{EquationEstimateForContractionFourthSummand} imply
\begin{align*}
  \GdomnormwgP{m-1}{\Phi(\hat{u}_1) - \Phi(\hat{u}_2)} &\leq \frac{1}{2} \GdomnormwgP{m-1}{\hat{u}_1- \hat{u}_2}.
\end{align*}
We conclude that $\Phi$ is a strict contraction on $B_R(J_\tau)$.

V) 
Banach's fixed point theorem thus
gives a fixed point $u \in B_R(J_\tau)$, i.e., $\Phi(u) = u$. By definition of $\Phi$,
this means that $u \in \GdomP{m}$ is a solution of~\eqref{EquationNonlinearIBVP}. 
Lemma~\ref{LemmaUniquenessOfNonlinearSolution} shows that 
$u$ is the only one on $[0,\tau]$.
\end{proof}


\begin{rem}
 \label{RemarkNegativeTimes}
 \begin{enumerate}
  \item Using time reversion and adapting coefficients and data accordingly, we can transfer
  the result of Theorem~\ref{TheoremLocalExistenceNonlinear} to the negative time direction, 
  see~\cite[Remark~3.3]{SpitzDissertation} for details.
  \item Standard techniques show that the restriction and the concatenation of solutions of~\eqref{EquationIBVPIntroduction} 
  are again solutions of~\eqref{EquationIBVPIntroduction}. For the precise statements and the proofs we 
  refer to Lemma~7.13 and Lemma~7.14 in~\cite{SpitzDissertation}.
 \end{enumerate}
\end{rem}

Theorem~\ref{TheoremLocalExistenceNonlinear} and Remark~\ref{RemarkNegativeTimes} show that 
the definition of a maximal solution makes sense.
\begin{definition}
 \label{DefinitionMaximalIntervalOfExistence}
 Let $t_0 \in \R$ and $m \in \N$ with $m \geq 3$. 
 Take $\chi \in \mlpdwl{m,6}{G}{\mathcal{U}}{c}$ and $\sigma \in\mlwl{m,6}{G}{\mathcal{U}}{c}$. 
 Choose data $f \in H^m((-T,T) \times G)$, $g \in E_m((-T,T) \times G)$, and $u_0 \in \Hhdom{m}$ 
 for  all $T > 0$  and define $B$ as in 
 Theorem~\ref{TheoremLocalExistenceNonlinear}. Assume that the tuple $(\chi,\sigma, t_0, B, f,g, u_0)$ fulfills the compatibility 
 conditions~\eqref{EquationNonlinearCompatibilityConditions} of order $m$. We introduce
 \begin{align*}
  T_+(m, t_0, f,g, u_0) &= \sup \{\tau \geq t_0 \colon \exists \, G_m  \text{-solution of } \eqref{EquationNonlinearIBVP} 
      \text{ on } [t_0,  \tau]\}, \\
  T_{-}(m, t_0, f,g, u_0) &= \inf \{\tau \leq t_0 \colon \exists \, G_m  \text{-solution of } \eqref{EquationNonlinearIBVP} 
      \text{ on } [\tau, t_0]\}.
 \end{align*}
 The interval $(T_{-}(m,t_0,f,g,u_0) , T_+(m,t_0,f,g,u_0)) =: I_{max}(m,t_0,f,g,u_0)$ is called 
 the maximal interval of existence.
\end{definition}

The name maximal interval of existence is justified by the next proposition. It states that there is 
a unique solution of~\eqref{EquationNonlinearIBVP} on the maximal interval of existence 
which cannot be extended beyond this interval. This solution is also called the maximal solution in 
the following. The proof works with standard techniques, see~\cite[Proposition~7.16]{SpitzDissertation} 
for details.
\begin{prop}
 \label{PropositionMaximalExistence}
Let $t_0 \in \R$ and $m \in \N$ with $m \geq 3$. 
Take $\chi \in \mlpdwl{m,6}{G}{\mathcal{U}}{c}$ and $\sigma \in\mlwl{m,6}{G}{\mathcal{U}}{c}$.
 Choose data $f \in H^m((-T,T) \times G)$, $g \in E_m((-T,T) \times G)$, and $u_0 \in \Hhdom{m}$ 
 for  all $T > 0$ and define $B$ as in Theorem~\ref{TheoremLocalExistenceNonlinear}.
 Assume that the tuple $(\chi,\sigma, t_0, B, f,g, u_0)$ fulfills the compatibility 
 conditions~\eqref{EquationNonlinearCompatibilityConditions} of order $m$. 
 Then there exists a unique maximal solution $u \in \bigcap_{j=0}^m C^j(I_{max}, \Hhdom{m-j})$ of~\eqref{EquationNonlinearIBVP}
 on $I_{max}$ which cannot be extended beyond this interval.
 \end{prop}

  \section{Blow-up criteria}
 \label{SectionBlowUpCriteria}
 
 We next want to characterize finite maximal existence times, i.e., the situation when 
 $T_+ < \infty$, by a blow-up condition. 
 As it is usually the case when the solution is constructed via Banach's fixed point theorem, 
 the construction allows to derive such a blow-up condition in the norm which controls the initial value. 
 In our case this is the $\Hhdom{m}$-norm.
 \begin{lem}
   \label{LemmaBlowUpCriterion}
    Let $t_0 \in \R$ and $m \in \N$ with $m \geq 3$. 
 Take $\chi \in \mlpdwl{m,6}{G}{\mathcal{U}}{c}$ and $\sigma \in \mlwl{m,6}{G}{\mathcal{U}}{c}$.
 Choose data $f \in H^m((-T,T) \times G)$, $g \in E_m((-T,T) \times G)$, and $u_0 \in \Hhdom{m}$ 
 for  all $T > 0$ and define $B$ as in Theorem~\ref{TheoremLocalExistenceNonlinear}. 
 Assume that the tuple $(\chi,\sigma, t_0, B, f,g, u_0)$ fulfills the compatibility 
 conditions~\eqref{EquationNonlinearCompatibilityConditions} of order $m$.   Let $u$ be the maximal solution of~\eqref{EquationNonlinearIBVP} on $I_{max}$ 
   provided by Proposition~\ref{PropositionMaximalExistence}. If 
   $T_+ = T_+(m, t_0, f,g, u_0) < \infty$, then one of the following blow-up properties
   \begin{enumerate}
    \item \label{ItemFirstBlowUpAlternativeConvergenceToBoundary} $\liminf_{t \nearrow T_+} \dist(\{u(t,x) \colon x \in G\}, \partial \mathcal{U}) = 0$, 
    \item \label{ItemFirstBlowUpAlternativeExplosionOfHmnorm} $\lim_{t \nearrow T_+} \Hhn{m}{u(t)} = \infty$
   \end{enumerate}
   occurs.
   The analogous result is true for $T_{-}(m, t_0, f, g, u_0)$.
  \end{lem}
  \begin{proof}
   Let $T_+ < \infty$ and assume that condition~\ref{ItemFirstBlowUpAlternativeConvergenceToBoundary} does not hold.
   This means that there exists $\kappa > 0$ such that 
   \begin{equation*}
    \dist(\{u(t,x) \colon x \in G\}, \partial \mathcal{U}) > \kappa
   \end{equation*}
   for all $t \in (t_0, T_+)$.
   Assume that there exists a sequence $(t_n)_n$ converging from below to the maximal existence time
   $T_+$ such that 
   $\rho := \sup_{n \in \N} \Hhn{m}{u(t_n)}$ is finite. Fix a time $T' > T_+$ and take a radius 
   $r > \rho$ with 
   \begin{equation*}
    m \|f\|_{G_{m-1}((t_0, T') \times G)}^2 + \|g\|_{E_m((t_0,T') \times G)}^2 + \rho^2 + \|f\|_{H^m((t_0, T') \times G)}^2 < r^2. 
   \end{equation*}
   Then pick an index $N \in \N$ such that 
   \begin{align*}
      t_N + \tau(\chi,\sigma,m, T' - t_0, r, \kappa) > T_+,
   \end{align*}
   for the time step $\tau = \tau(\chi,\sigma,m,T' - t_0, r,\kappa)$ from Theorem~\ref{TheoremLocalExistenceNonlinear}.
   Identity~\eqref{EquationTimeDerivativesOfSolutionEqualSChiSigmaG} shows that
   the tuple $(\chi, \sigma, t_N,B, f,g, u(t_N))$ fulfills the compatibility 
   conditions~\eqref{EquationNonlinearCompatibilityConditions} of order $m$.
   Since the distance between $\image u(t_N)$ and $\partial \mathcal{U}$ is larger or equal than
   $\kappa$, Theorem~\ref{TheoremLocalExistenceNonlinear} thus gives a $G_m$-solution $v$
   of~\eqref{EquationNonlinearIBVP} with inhomogeneity $f$, boundary value $g$, and initial value $u(t_N)$ at $t_N$ on $[t_N, t_N + \tau]$.
   Setting $w(t) := u(t)$ if $t \in [t_0, t_N]$ and $w(t) := v(t)$ if $t \in [t_N, t_N + \tau]$, we obtain a $G_m$-solution 
   of~\eqref{EquationNonlinearIBVP} with data $f$, $g$, and $u_0$ on $[t_0, t_N + \tau]$ by Remark~\ref{RemarkNegativeTimes}.
   This contradicts the definition of $T_+$ since $t_N + \tau > T_+$.
   The assertion for $T_{-}$ is proven analogously.
  \end{proof}
  
  The blow-up criterion above can be improved. In fact we will show that if $T_+ < \infty$ (and 
  the solution does not come arbitrarily close to $\partial \mathcal{U}$), then the spatial Lipschitz norm 
  of the solution has to blow up when one approaches $T_+$, see 
  Theorem~\ref{TheoremLocalWellposednessNonlinear}~\ref{ItemBlowUp} below. There are several examples of quasilinear systems, both 
on the full space and on domains, where the blow-up condition is given in terms of the Lipschitz-norm of the 
solution, see e.g.~\cite{BahouriCheminDanchin, BealeKatoMajda, BenzoniGavage, Klainerman,  KlainermanPonce,  Majda,  MetivierEtAl}.
  This improvement (in comparison with the $\Hhdom{m}$-norm) is possible as one can exploit that 
  a solution $u$ of the nonlinear problem~\eqref{EquationNonlinearIBVP} solves the linear problem~\eqref{EquationIBVPIntroduction} 
  with coefficients $\chi(u)$ and $\sigma(u)$. Deriving estimates for the derivatives of $u$, 
  one can then use so-called Moser-type estimates, see the proof of Proposition~\ref{PropositionEstimateForNonlinearSolutionOnDomain} 
  below. These estimates, introduced in~\cite{Moser} and 
  based on the Gagliardo-Nirenberg estimates from~\cite{NirenbergOnEllipticPartialDifferentialEquations}, 
  are an efficient tool to estimate products of derivatives of $u$. However, 
  as our material laws $\chi$ and $\sigma$ do also depend on the space variable $x$, we cannot use them 
  in their standard form (see e.g. \cite{KlainermanMajda, Majda} and also~\cite{SpitzDissertation} for domains, where we treated 
  a slightly simpler case). 
  But the proof of the version below still follows the standard ideas already used in~\cite{Moser}.
  \begin{lem}
   \label{LemmaMoserTypeInequalities}
   Let $T > 0$, $J = (0,T)$, and $m \in \N$ with $m \geq 3$.
   Let $\theta \in \ml{m,1}{G}{\mathcal{U}}$ and $v \in \Gdom{m}$. Assume that there is a number $\zeta_0 > 0$ 
   and a compact subset $\mathcal{U}_1$ of $\mathcal{U}$ such that $\|v\|_{W^{1,\infty}(J \times G)} \leq \zeta_0$ 
   and $\image v \subseteq \mathcal{U}_1$. Then there is a constant $C = C(\theta, \zeta_0, \mathcal{U}_1)$ such that
   \begin{equation*}
    \|\partial^\beta \theta(u) \partial_t \partial^{\alpha - \beta} u\|_{L^2(J \times G)} 
      + \|\partial^\beta \theta(u) \partial^{\alpha - \beta} u\|_{L^2(J \times G)} \leq C \|v\|_{H^{|\alpha|}(J \times G)}
   \end{equation*}
   for all $0 < \beta \leq \alpha$ and $\alpha \in \N_0^4$ with $|\alpha| \leq m$.
  \end{lem}
  We will employ this lemma in the proof of the next proposition. There it has to be combined with 
  a technique developed in~\cite{SpitzMaxwellLinear} to control the derivatives in normal direction 
  of solutions 
  of~\eqref{EquationNonlinearIBVP} although this system has a characteristic boundary. For later 
  reference, we recall the key result in this direction. It is a simplified version 
  of~\cite[Proposition~3.3 and Remark~4.11]{SpitzMaxwellLinear} and relies heavily on the structure of the Maxwell 
  equations.
  \begin{lem}
\label{LemmaCentralEstimateInNormalDirection}
Let $T' > 0$, $\eta > 0$, $\gamma \geq 1$, and $r \geq r_0 >0$. Pick $T \in (0, T')$ and set $J = (0,T)$.
Take $A_0 \in \Fupdwl{0}{\operatorname{c}}{\eta}$, $A_1, A_2 \in \Fcoeff{0}{\operatorname{cp}}$, $A_3 = A_3^{\operatorname{co}}$,
and $D \in \Fuwl{0}{\operatorname{c}}$ with 
\begin{align*}
	&\|A_i\|_{W^{1,\infty}(\Omega)} \leq r, \quad \|D\|_{W^{1,\infty}(\Omega)} \leq r, \\
	&\|A_i(0)\|_{L^\infty(\R^3_+)} \leq r_0,  \quad \|D(0)\|_{L^\infty(\R^3_+)} \leq r_0
\end{align*}
for all $i \in \{0, \ldots, 2\}$. 
Choose $f \in \Ha{1}$ and $u_0 \in \Hh{1}$.
Let $u \in \G{1}$ solve the initial value problem
 \begin{equation}
  \label{IVP}
\left\{\begin{aligned}
   A_0 \partial_t u + \sum_{j=1}^3 A_j \partial_j u + D u  &= f, \quad &&x \in \R^3_+, \quad &t \in J; \\
   u(0) &= u_0, \quad &&x \in \R^3_+.
\end{aligned}\right.
\end{equation}
Then there are constants $C_{1,0} = C_{1,0}(\eta,r_0) \geq 1$ and $C_1 = C_1(\eta, r, T') \geq 1$ 
such that
\begin{align}
 \label{EquationFirstOrderFinalVariant}
    \Gnorm{0}{\nabla u}^2 &\leq e^{C_1 T}  \Big((C_{1,0} + T C_1)\Big(\sum_{j = 0}^2 \Ltwohnt{\partial_j u}^2  
      +   \Ltwohn{f(0)}^2 + \Hhn{1}{u_0}^2 \! \Big)\nonumber \\
    & \hspace{4em}  + \frac{C_1}{\gamma} \Hangamma{1}{f}^2   \Big).
\end{align}
\end{lem}

  
  We can now prove the main result of this section, showing that we control the 
  $\Hhdom{m}$-norm of a solution as soon as we control its spatial Lipschitz-norm. For the proof we 
  differentiate~\eqref{EquationNonlinearIBVP} and employ the basic $L^2$-estimate 
  from Theorem~\ref{TheoremExistenceAndUniquenessOnDomain}
  to the derivative of $u$. 
  The Moser-type estimate from Lemma~\ref{LemmaMoserTypeInequalities} allow us to treat the arising 
  inhomogeneities in such a way that the Gronwall lemma yields the desired estimate. However, 
  this approach only works in tangential directions. In order to bound the derivatives of $u$ 
  containing a derivative in normal direction, we have to combine the above approach with 
  Lemma~\ref{LemmaCentralEstimateInNormalDirection}.
  \begin{prop}
	\label{PropositionEstimateForNonlinearSolutionOnDomain}
	Let $m \in \N$ with $m \geq 3$ and $t_0 \in \R$. Take functions 
	$\chi \in \mlpdwl{m,6}{G}{\mathcal{U}}{c}$ and $\sigma \in \mlwl{m,6}{G}{\mathcal{U}}{c}$. Let
   \begin{align*}
   	B(x) = \begin{pmatrix}
             0 &\nu_3(x) &-\nu_2(x) &0 &0 &0 \\
      -\nu_3(x) &0 &\nu_1(x) &0 &0 &0 \\
      \nu_2(x) &-\nu_1(x) &0 &0 &0 &0 
            \end{pmatrix}
   \end{align*}
   for all $x \in \partial G$, where $\nu$ denotes the unit outer normal vector of $\partial G$.
    Choose data $u_0 \in \Hhdom{m}$, $g \in E_m((-T,T) \times \partial G)$, 
   and $f \in H^m((-T,T) \times G)$ 
   for all $T > 0$ such that the tuple $(\chi,\sigma,t_0, B, f,g,u_0)$ fulfills the compatibility 
   conditions~\eqref{EquationNonlinearCompatibilityConditions}
   of order $m$.
   Let $u$ denote the maximal solution of~\eqref{EquationNonlinearIBVP}
    provided by Proposition~\ref{PropositionMaximalExistence} on $(T_{-},T_+)$. We introduce the quantity 
    \begin{align*}
     \omega(T) = \sup_{t \in (t_0,T)} \| u(t) \|_{W^{1,\infty}(G)}
    \end{align*}
    for every $T \in (t_0,T_+)$. We further take $r > 0$ with 
    \begin{align*}
     &\sum_{j=0}^{m-1} \Hhndom{m-j-1}{\partial_t^j f(t_0)} \! + \!  \|g\|_{E_m((t_0, T_+) \times \partial G)} \! + \! \Hhndom{m}{u_0} \! + \! \|f\|_{H^m((t_0, T_+) \times G)} \leq r.
    \end{align*}
    We set $T^* = T_+$ if $T_+ < \infty$ and take any $T^* > t_0$ if $T_+ = \infty$. Let $\omega_0 > 0$ and let 
    $\mathcal{U}_1$ be a compact subset of $\mathcal{U}$.

    Then there exists a constant $C = C(\chi,\sigma,m,r,\omega_0,\mathcal{U}_1,T^* - t_0)$ such that
    \begin{align*}
     \|u\|_{G_m((t_0,T)\times G)}^2 &\leq C \Big( \sum_{j=0}^{m-1} \Hhndom{m-1-j}{\partial_t^j f(t_0)}^2  + \Hhndom{m}{u_0}^2 + \|g\|_{E_m((t_0, T) \times \partial G)}^2 \\
     &\hspace{4em} + \|f\|_{H^m((t_0, T) \times G)}^2 \Big)
    \end{align*}
    for all $T \in (t_0,T^*)$ which have the property that  $\omega(T) \leq \omega_0$ and $\image u(t) \subseteq \mathcal{U}_1$ for all $t \in [t_0,T]$. The analogous result is true on $(T_{-},t_0)$.
\end{prop}
\begin{proof}
   Without loss of generality we assume $t_0 = 0$ and that $\chi$ and $\sigma$ have 
   property~\eqref{EquationPropertyForL2}, cf.~Remark~\ref{RemarkChiuInFm}.
   Let $\omega_0 > 0$ and $\mathcal{U}_1$ be a compact subset of $\mathcal{U}$. If $\omega(T) > \omega_0$ or if 
   the set $\{u(t,x) \colon (t,x) \in [t_0,T] \times G\}$ is not contained in $\mathcal{U}_1$ for all $T \in (0, T^*)$, 
   there is nothing to prove. Otherwise we fix $T' \in (0,T^*)$ with $\omega(T') \leq \omega_0$ and $\image u(t) \subseteq \mathcal{U}_1$ 
   for all $t \in [t_0,T']$.
   Let $T \in (0,T']$ be arbitrary
   and denote $(0,T) \times \R^3_+$ 
   by $\Omega$. Note that $\omega(T) \leq \omega(T') \leq \omega_0$ and $\image u(t) \subseteq \mathcal{U}_1$ for all 
   $t \in [t_0, T]$.
   
   To derive the improved estimates, we have to study the problem on the half-space. To that 
   purpose, we apply the localization procedure already used in the proof of
   Theorem~\ref{TheoremExistenceAndUniquenessOnDomain}, 
   see~\cite[Section~2]{SpitzMaxwellLinear} and~\cite[Chapter~5]{SpitzDissertation}. To streamline 
   the argument, we do not perform the localization here but assume that $G = \R^3_+$ and that 
   we have spatial coefficients $A_1, A_2 \in \Fcoeff{m}{cp}$ and $A_3 = A_3^{\operatorname{co}}$. 
   The full space case is easier and treated similarly. 
   We refer to~\cite[page~9]{SpitzMaxwellLinear} and~\cite[Proposition~7.20]{SpitzDissertation} for the details.
   
    We pick a number $\eta = \eta(\chi) > 0$ such that $\chi \geq \eta$. Consequently, 
    there is a constant $C$ with $|\chi(\xi)^{-1}| \leq C \eta^{-1}$
    for all $\xi \in \R^6$. Since the function $u$ solves~\eqref{EquationNonlinearIBVP}, we infer
    \begin{align}
    \label{EquationEstimatingLipschitzNormAgainstOmega}
     \|\partial_t u\|_{L^\infty(\Omega)} &\leq \|\chi(u)^{-1} f\|_{L^\infty(\Omega)} + \! \sum_{j = 1}^3 \|\chi(u)^{-1} A_j \partial_j u\|_{L^\infty(\Omega)} 
	+ \|\chi(u)^{-1} \sigma(u) u\|_{L^\infty(\Omega)} \nonumber\\
	&\leq C(\eta,\sigma,\mathcal{U}_1) \Big( \Han{m}{f} + 3 \,\omega(T) + \omega(T)  \Big), \nonumber\\
    \|u\|_{W^{1,\infty}(\Omega)} &\leq \|\partial_t u\|_{L^\infty(\Omega)} +  \omega(T)
	\leq C_{\ref{EquationEstimatingLipschitzNormAgainstOmega}}(\chi,\sigma,r,\omega_0, \mathcal{U}_1).
    \end{align}
    In the following we will frequently apply~\eqref{EquationEstimatingLipschitzNormAgainstOmega} 
    without further reference.
    
    I) We set
    \begin{align}
     \label{EquationDefinitionfalphaigalpha}
      f_\alpha &= \partial^\alpha f - \sum_{0 < \beta \leq \alpha} \binom{\alpha}{\beta} \partial^\beta \chi(u)  \partial_t \partial^{\alpha - \beta} u - \sum_{j = 1}^2 \sum_{0 < \beta \leq \alpha} \binom{\alpha}{\beta} \partial^\beta A_j \partial_j \partial^{\alpha - \beta} u  \nonumber \\
     &\quad - \sum_{0 <\beta \leq \alpha} \binom{\alpha}{\beta} \partial^\beta \sigma(u) \partial^{\alpha - \beta} u 
    \end{align}
    for all $\alpha \in \N_0^4$ with $|\alpha| \leq m$. As $u$ is a solution of~\eqref{EquationNonlinearIBVP}, 
    the function
    $\partial^\alpha  u$ solves the linear initial value problem
    \begin{equation}
       \label{EquationIVPForDerivative}
       \left\{
      \begin{aligned}
	 &L(\chi(u),A_1, A_2, A_3, \sigma(u)) v = f_{\alpha}, \quad &&x \in \R^3_+, \quad &&t \in (0,T); \\
	 &v(0) = \partial^{(0,\alpha_1,\alpha_2,\alpha_3)} S_{\chi,\sigma, \R^3_+, m, \alpha_0}(0,f,u_0), \quad &&x \in \R^3_+;
      \end{aligned} \right.
    \end{equation}
    for all $\alpha \in \N_0^4$ with $|\alpha| \leq m$. Moreover, if additionally $\alpha_3 = 0$, it
    solves 
    \begin{equation}
       \label{EquationIBVPForDerivative}
       \left\{
      \begin{aligned}
	 &L(\chi(u),A_1, A_2, A_3, \sigma(u)) v = f_{\alpha}, \quad &&x \in \R^3_+, \quad &&t \in (0,T); \\
	 &B v = \partial^\alpha g, \quad &&x \in \partial \R^3_+, \quad &&t \in (0,T); \\
	 &v(0) = \partial^{(0,\alpha_1,\alpha_2,0)} S_{\chi,\sigma, \R^3_+, m, \alpha_0}(0,f,u_0), \quad &&x \in \R^3_+.
      \end{aligned} \right.
    \end{equation}
    Here we used that $\partial_t^j u(0) = S_{\chi,\sigma, \R^3_+, m, j}(0,f,u_0)$
    for all $j \in \{0, \ldots, m\}$ by~\eqref{EquationTimeDerivativesOfSolutionEqualSChiSigmaG}.
    
    In view of~\eqref{EquationIBVPForDerivative} respectively~\eqref{EquationIVPForDerivative}, we want to apply 
    Theorem~\ref{TheoremExistenceAndUniquenessOnDomain} respectively Lemma~\ref{LemmaCentralEstimateInNormalDirection} 
    to $\partial^\alpha u$. We thus need estimates for $\Ltwoanwgamma{f_\alpha}$ for all $\alpha \in \N_0^4$ 
    with $|\alpha| \leq m$,  $\Han{1}{f_\alpha}$ for all $\alpha \in \N_0^4$ with $|\alpha| \leq m-1$, and 
    $\Ltwohn{f_{\alpha}(0)}$ for all $\alpha \in \N_0^4$ with $|\alpha| \leq m-1$. We start with the estimate for $\Ltwoanwgamma{f_\alpha}$. 
    Take $\alpha \in \N_0^4$ with $|\alpha| \leq m$. Let $\beta \in \N_0^4$ with 
    $0 < \beta \leq \alpha$. Lemma~\ref{LemmaMoserTypeInequalities} then implies that
    \begin{align}
     \label{EquationFinalEstimateForFalpha}
     \Ltwoanwgamma{f_\alpha} &\leq  \Han{|\alpha|}{f}  + C(\chi,\sigma,r,\omega_0,\mathcal{U}_1) \Han{|\alpha|}{u}.
    \end{align}
    
    Next, we want to estimate $\Han{1}{f_\alpha}$ for $\alpha \in \N_0^4$ with $|\alpha| \leq m-1$. So 
    fix such a multi-index. Let $k \in \{0,\ldots,3\}$ and set 
    $\alpha^k = \alpha + e_k$. A straightforward computation, see e.g.~(3.6) in~\cite{SpitzDissertation}, shows the formula 
    \begin{equation}
    \label{EquationfalphakpartialkfalphaMinusErrorTerms}
      f_{\alpha^k} = \partial_k f_\alpha -  \partial_k \chi(u)  \partial_t \partial^{\alpha} u - \sum_{j = 1}^2 \partial_k A_j \partial_j \partial^\alpha u - \partial_k \sigma(u) \partial^{\alpha} u.
    \end{equation}
    Combined with~\eqref{EquationFinalEstimateForFalpha}, Lemma~\ref{LemmaMoserTypeInequalities} 
    now yields the inequality
    \begin{align}
     \label{EquationEstimateForH1NormFalpha}
     \Han{1}{f_\alpha} &\leq C(\chi,\sigma,r,\omega_0,\mathcal{U}_1) \Big(\Han{|\alpha| + 1}{f}  +  \Han{|\alpha|+1}{u}\Big).
    \end{align}
    
    It remains to estimate $\Ltwohn{f_\alpha(0)}$ for $\alpha \in \N_0^4$ with $|\alpha| \leq m-1$. To that purpose 
    we first insert $t = 0$ in the 
    definition of $f_\alpha$ in~\eqref{EquationDefinitionfalphaigalpha}. The product estimates from Lemma~2.1 in~\cite{SpitzMaxwellLinear}, 
    the fact that $\partial_t^j u(0) = S_{\chi,\sigma,\R^3_+,m,j}(0,f,u_0)$ for all $j \in \{0, \ldots, m\}$ 
    by~\eqref{EquationTimeDerivativesOfSolutionEqualSChiSigmaG}, and Lemma~\ref{LemmaNonlinearHigherOrderInitialValues} 
    then lead to the bound
    \begin{align}
    \label{EquationEstimateForFalphaInZero}
     &\Ltwohn{f_{\alpha}(0)} \leq  C(\chi,\sigma,m,r,\omega_0, \mathcal{U}_1) \Big(\sum_{l = 0}^{|\alpha|} \Hhn{|\alpha|-l}{\partial_t^l f(0)} + \Hhn{|\alpha|+1}{u_0}\Big),
    \end{align}
    see Proposition~7.20 in~\cite{SpitzDissertation} for the details.
    
    II) We will show inductively that there are constants 
    $C_{k} \!= \!C_k(\chi,\sigma,m,r,\omega_0,\mathcal{U}_1,\!T^*\!)$ such that 
    \begin{align}
    \label{EquationEstimateDerivativeForInductionStep}
     \Gnormwg{0}{\partial^\alpha u}^2 &\leq C_k \Big(\sum_{j = 0}^{k-1} \Hhn{k-1-j}{\partial_t^j f(0)}^2 + \Enormwg{k}{g}^2 + \Hhn{k}{u_0}^2 \nonumber \\
     &\hspace{3em} + \Han{k}{f}^2 \Big)
    \end{align}
    for all $\alpha \in \N_0^4$ with $|\alpha| = k$ and $k \in \{0,\ldots,m\}$.
    
     We first apply Lemma~\ref{LemmaHigherOrderChainRule}~\ref{ItemHigherOrderChainRuleInG} 
    and~\ref{ItemHigherOrderChainRuleMixed} to obtain a radius $R_1 = R_1(\chi,\sigma,r,\omega_0, \mathcal{U}_1)$ with
    \begin{align*}
     &\|\chi(u)\|_{W^{1,\infty}(\Omega)} + \|\sigma(u)\|_{W^{1,\infty}(\Omega)}  \leq R_1(\chi,\sigma,r,\omega_0, \mathcal{U}_1), \\
     &\|\chi(u(0))\|_{L^\infty(\R^3_+)} + \|\sigma(u(0))\|_{L^\infty(\R^3_+)} \leq R_1(\chi,\sigma,r,\omega_0, \mathcal{U}_1).
    \end{align*}
    
    Set $\gamma_0 = \gamma_0(\chi,\sigma,r,\omega_0, \mathcal{U}_1,T^*) = \gamma_{\ref{TheoremExistenceAndUniquenessOnDomain},0}(\eta(\chi), R_1(\chi,\sigma,r,\omega_0, \mathcal{U}_1),T^*) \geq 1$, 
    where $\gamma_{\ref{TheoremExistenceAndUniquenessOnDomain},0}$ is the corresponding constant from Theorem~\ref{TheoremExistenceAndUniquenessOnDomain}. 
    As $u$ solves~\eqref{EquationIBVPForDerivative} with $\alpha = 0$, Theorem~\ref{TheoremExistenceAndUniquenessOnDomain} yields
    \begin{align*}
     \Gnormwg{0}{u}^2 &\leq e^{2 \gamma_0 T} \sup_{t \in (0,T)} \|e^{-\gamma_0 t} u(t)\|_{L^2(\R^3_+)}^2 \\
     &\leq C_{\ref{TheoremExistenceAndUniquenessOnDomain},0,0}(\eta,R_1) e^{2 \gamma_0 T^*} \Big(\Ltwohn{u_0}^2 + \|g\|_{E_{0,\gamma_0}(J \times \partial \R^3_+)}^2 \Big) \\
     &\hspace{2em} + C_{\ref{TheoremExistenceAndUniquenessOnDomain},0}(\eta,R_1,T^*) e^{2 \gamma_0 T^*} \frac{1}{\gamma_0} \|f\|_{L^2_{\gamma_0}(\Omega)}^2 \\
     &\leq C_0 \Big(\Ltwohn{u_0}^2 + \Enormwg{0}{g}^2 + \Ltwoanwgamma{f}^2 \Big),
    \end{align*}
    where $C_0 = C_0(\chi,\sigma,r,\omega_0,\mathcal{U}_1,T^*)$ and $C_{\ref{TheoremExistenceAndUniquenessOnDomain},0,0}$ respectively $C_{\ref{TheoremExistenceAndUniquenessOnDomain},0}$ 
    denote the corresponding constants from Theorem~\ref{TheoremExistenceAndUniquenessOnDomain}. 
    This inequality shows the claim~\eqref{EquationEstimateDerivativeForInductionStep} 
    for $k = 0$.
    
    Let $k \in \{1,\ldots,m\}$ and assume that~\eqref{EquationEstimateDerivativeForInductionStep} has been 
    shown for all $ 0 \leq j  \leq k-1$. We first claim that there are constants $C_{k,\alpha} = C_{k,\alpha}(\chi,\sigma,r,\omega_0,\mathcal{U}_1,T^*)$ such that
    \begin{align}
    \label{EquationHypothesisForMultiindexOfAbsoluteValuek}
     \Gnormwg{0}{\partial^\alpha u}^2 &\leq C_{k,\alpha} \Big(\sum_{j = 0}^{k-1} \Hhn{k-1-j}{\partial_t^j f(0)}^2 + \Enormwg{k}{g}^2 + \Hhn{k}{u_0}^2 \nonumber \\
     &\hspace{4em} + \Han{k}{f}^2 + \int_0^T \sum_{\beta  \in \N_0^4,  |\beta| = k} \Ltwohn{\partial^\beta u(s)}^2 ds \Big)
    \end{align}
    for all $\alpha \in \N_0^4$ with $|\alpha| = k$. We show~\eqref{EquationHypothesisForMultiindexOfAbsoluteValuek} 
    by another induction, this time with respect to $\alpha_3$.
    
    Let $\alpha \in \N_0^4$ with $|\alpha| = k$ and $\alpha_3 = 0$. In step I) we have seen that 
    $\partial^\alpha u$ solves the initial boundary value problem~\eqref{EquationIBVPForDerivative}. 
    Hence, Theorem~\ref{TheoremExistenceAndUniquenessOnDomain}  yields
    \begin{align*}
     &\Gnormwg{0}{\partial^\alpha u}^2 \leq e^{2 \gamma_0 T} \sup_{t \in (0,T)} \Ltwohn{e^{- \gamma_0 t} \partial^\alpha u(t)}^2 \\
     &\leq  C_{\ref{TheoremExistenceAndUniquenessOnDomain},0,0}(\eta,R_1) e^{2 \gamma_0 T^*} \Big(\Ltwohn{\partial^{(0,\alpha_1,\alpha_2,0)} S_{\chi,\sigma,\R^3_+,m,\alpha_0}(0,f,u_0)}^2  \\
     &\hspace{9em} + \|\partial^\alpha g\|_{E_{0,\gamma_0}(J \times \partial \R^3_+)}^2 \Big) + C_{\ref{TheoremExistenceAndUniquenessOnDomain},0}(\eta,R_1,T^*) e^{2 \gamma_0 T^*} \frac{1}{\gamma_0} \|f_\alpha\|_{L^2_{\gamma_0}(\Omega)}^2 \\
     &\leq C(\chi,\sigma,k,r,\omega_0,\mathcal{U}_1,T^*) \Big(\sum_{j = 0}^{k-1} \Hhn{k - 1 -j}{\partial_t^j f(0)}^2 + \|g\|_{E_{k}(J \times \partial \R^3_+)}^2 + \Hhn{k}{u_0}^2 \Big) \\
     &\hspace{12em} +  \Han{k}{f}^2 +  \Han{k}{u}^2 \Big),
    \end{align*}
    where we employed Lemma~\ref{LemmaNonlinearHigherOrderInitialValues} and~\eqref{EquationFinalEstimateForFalpha}.
    Using the induction hypothesis~\eqref{EquationEstimateDerivativeForInductionStep} for the derivatives of $u$ 
    of order smaller or equal than $k-1$, we 
    arrive at
    \begin{align*}
     &\Gnormwg{0}{\partial^\alpha u}^2 \leq C_{k,\alpha}(\chi,\sigma,r,\omega_0,\mathcal{U}_1,T^*) \Big(\sum_{j=0}^{k-1} \Hhn{k-1-j}{\partial_t^j f(0)}^2  + \|g\|_{E_{k}(J \times \partial \R^3_+)}^2   \\
      &\hspace{8 em} + \Hhn{k}{u_0}^2 + \Han{k}{f}^2 + \int_0^T \sum_{\beta \in \N_0^4, |\beta| = k}  \Ltwohn{\partial^\beta u(s)}^2 ds \Big),
    \end{align*}
    which is~\eqref{EquationHypothesisForMultiindexOfAbsoluteValuek} for all multiindices $\alpha$ with 
    $|\alpha| = k$ and $\alpha_3 = 0$.   
       
    Now, assume that there is a number $l \in \{1,\ldots, k\}$ such that~\eqref{EquationHypothesisForMultiindexOfAbsoluteValuek} 
    is true for all $\alpha \in \N_0^4$ with $|\alpha| = k$ and $\alpha_3 \in \{0,\ldots,l-1\}$.
    
    Take $\alpha \in \N_0^4$ with $|\alpha| = k$ and $\alpha_3 = l$. The multi-index  $\alpha' = \alpha - e_3$ 
    belongs to $\N_0^4$ and satisfies $|\alpha'| = k - 1 \leq m-1$. Due to step I), we know 
    that $\partial^{\alpha'} u$ solves the initial value problem~\eqref{EquationIVPForDerivative} with right-hand side 
    $f_{\alpha'}$ and initial value 
    \begin{equation*}
      \partial^{(0,\alpha_1,\alpha_2,\alpha_3-1)} S_{\chi,\sigma,\R^3_+,m,\alpha_0}(0,f,u_0).
    \end{equation*}
    As $|\alpha'| \leq m-1$, the function   $f_{\alpha'}$ belongs to $\Ha{1}$ by~\eqref{EquationEstimateForH1NormFalpha}, 
    the derivative of the higher order initial value $\partial^{(0,\alpha_1,\alpha_2,\alpha_3 -1)} S_{\chi,\sigma,G,m,\alpha_0}(0,f,u_0)$ to $\Hh{1}$ by Lemma~\ref{LemmaNonlinearHigherOrderInitialValues}, 
    and $\partial^{\alpha'} u$ to $\G{1}$. Moreover, $\chi(u)$ and $\sigma(u)$ are elements of $\Fupdwl{0}{\operatorname{c}}{\eta}$ respectively 
    $\Fuwl{0}{\operatorname{c}}$, 
    $A_1$ and $A_2$ belong to $\Fcoeff{m}{cp}$ and $A_3 = A_3^{\operatorname{co}}$.
    We can therefore apply Lemma~\ref{LemmaCentralEstimateInNormalDirection}. We choose $\gamma = 1$ to infer
    \begin{align}
    \label{EquationApplyingLemmaCentralEstimateNormalDirectionToDerivative}
     &\Gnormwg{0}{\partial^\alpha u}^2 = \Gnormwg{0}{\partial_3 \partial^{\alpha'} u}^2 \leq e^{2 T} \|\nabla \partial^{\alpha'} u\|_{G_{0,1}(\Omega)}^2 \nonumber\\
     &\leq e^{(2+C_1) T} \Big((C_{1,0} + T C_1) \Big( \sum_{j = 0}^2 \|\partial_j \partial^{\alpha'} u\|_{G_{0,1}(\Omega)}^2 
	    + \Ltwohn{f_{\alpha'}(0)}^2 \Big) \nonumber\\
     &\hspace{1.5em} + (C_{1,0} + T C_1) \Hhn{1}{\partial^{(0,\alpha_1,\alpha_2,\alpha_3 -1)} S_{\chi,\sigma,\R^3_+,m,\alpha_0}(0,f,u_0)}^2 + C_1 \| f_{\alpha'}\|_{H^1(\Omega)}^2 \Big) \nonumber\\
     &\leq C(\chi,\sigma,k,r,\omega_0,\mathcal{U}_1,T^*) \Big( \sum_{j = 0}^2 \|\partial_j \partial^{\alpha'} u\|_{G_{0}(\Omega)}^2 
	     +  \sum_{j = 0}^{k-1} \Hhn{k-1-j}{\partial_t^j f(0)}^2  \nonumber\\
     &\hspace{12em}  + \Hhn{k}{u_0}^2 + \Han{k}{f}^2 + \Han{k}{u}^2 \Big),
    \end{align}
    where we used~\eqref{EquationEstimateForFalphaInZero}, Lemma~\ref{LemmaNonlinearHigherOrderInitialValues}, 
    and~\eqref{EquationEstimateForH1NormFalpha} in the last estimate and where
    \begin{align*}
	C_{1,0} &= C_{1,0}(\chi,\sigma,r,\omega_0, \mathcal{U}_1) = C_{\ref{LemmaCentralEstimateInNormalDirection},1,0}(\eta(\chi),R_1(\chi,\sigma,r,\omega_0,\mathcal{U}_1)), \\
	C_1 &= C_1(\chi,\sigma,r,\omega_0,\mathcal{U}_1,T^*) = C_{\ref{LemmaCentralEstimateInNormalDirection},1}(\eta(\chi),R_1(\chi,\sigma,r,\omega_0,\mathcal{U}_1),T^*).
    \end{align*}
    Inserting the induction hypothesis for $\|\partial_j \partial^{\alpha'} u\|_{G_{0}(\Omega)}^2$ 
    and the induction hypothesis~\eqref{EquationEstimateDerivativeForInductionStep} for the derivatives 
    of $u$ of order smaller or equal than $k-1$, we obtain~\eqref{EquationHypothesisForMultiindexOfAbsoluteValuek} 
    for all $\alpha \in \N_0^4$ with $|\alpha| = k$ and $\alpha_3 = l$. By induction, we thus 
    infer that the estimate in~\eqref{EquationHypothesisForMultiindexOfAbsoluteValuek} is valid for all multiindices 
    $\alpha \in \N_0^4$ with $|\alpha| = k$. 
    
     We now sum in~\eqref{EquationHypothesisForMultiindexOfAbsoluteValuek} over all multiindices with $|\alpha| = k$, which yields
    \begin{align*}
     &\sum_{\alpha \in \N_0^4, |\alpha| = k} \Ltwohn{\partial^\alpha u(T)}^2 
     \leq \sum_{\alpha \in \N_0^4, |\alpha| = k} \Gnormwg{0}{\partial^\alpha u}^2 \\
     &\leq \Big(\sum_{\alpha \in \N_0^4,  |\alpha| = k}C_{k,\alpha}\Big)
     \Big( \sum_{j = 0}^{k-1} \Hhn{k-1-j}{\partial_t^j f(0)}^2 + \|g\|_{E_{k}((0,T) \times \partial \R^3_+)}^2  + \Hhn{k}{u_0}^2\\
      &\hspace{8em} + \|f\|_{H^k((0,T) \times \R^3_+)}^2  + \int_0^T \sum_{\beta \in \N_0^4, |\beta| = k}  \Ltwohn{\partial^\beta u(s)}^2 ds \Big).
    \end{align*}
    Recall that the time $T \in (0, T']$ was arbitrary. 
    Since $T \mapsto \|f\|_{H^k((0,T)\times\R^3_+)}^2$ and $T \mapsto \|g\|_{E_k((0,T) \times \partial \R^3_+)}^2$ are monotonically increasing, 
    Gronwall's inequality leads to
    \begin{align}
    \label{EquationEstimateAfterGronwall}
     \sum_{\alpha \in \N_0^4, |\alpha| = k} \Ltwohn{\partial^\alpha u(t)}^2 
     &\leq C_{k}' e^{C_k' t} \Big(\sum_{j = 0}^{k-1} \Hhn{k-1-j}{\partial_t^j f(0)}^2 + \|g\|_{E_k((0,t) \times \R^3_+)}^2 \nonumber \\
     &\hspace{6em} + \Hhn{k}{u_0}^2 + \|f\|_{H^k((0,t)\times\R^3_+)}^2 \Big) 
    \end{align}
    for all $t \in [0,T']$, where we defined 
    $C_{k}' = C_k'(\chi,\sigma,r,\omega_0,\mathcal{U}_1,T^*)$ by $\sum_{\alpha \in \N_0^4, |\alpha| = k} C_{k, \alpha}$. Defining $C_k = C_k(\chi,\sigma,r,\omega_0,\mathcal{U}_1,T^*)$ 
    by $C_k' e^{C_k' T^*}$ 
    and taking again a fixed time $T \in (0,T']$, we particularly obtain
    \begin{align*}
     \sum_{\alpha \in \N_0^4, |\alpha| = k} \Ltwohn{\partial^\alpha u(t)}^2 
     &\leq C_{k} \Big(\sum_{j = 0}^{k-1} \Hhn{k-1-j}{\partial_t^j f(0)}^2 + \|g\|_{E_k((0,T) \times \partial \R^3_+)}^2 \\
     &\hspace{4em}+ \Hhn{k}{u_0}^2 + \|f\|_{H^k((0,T)\times\R^3_+)}^2 \Big) 
    \end{align*}
    for all $t \in [0,T]$.
    
    We conclude that~\eqref{EquationEstimateDerivativeForInductionStep} is true for 
    all $\alpha \in \N_0^4$ with $|\alpha| = k$. Again by induction, we infer that~\eqref{EquationEstimateDerivativeForInductionStep} 
    is true for all $\alpha \in \N_0^4$ with $|\alpha| \in \{0,\ldots,m\}$. Summing over all multiindices 
    with absolute value between $0$ and $m$, the assertion of the proposition finally follows.
\end{proof}
  This proposition now easily implies the improved blow-up condition. We postpone the statement 
  and its proof to the full local wellposedness theorem below.
 \section{Continuous dependance and local wellposedness theorem}
 \label{SectionContinuousDependance}
 The investigation of continuous dependance for quasilinear problems is challenging because of a loss 
 of derivatives. It occurs since the difference of 
 two solutions satisfies an equation with a less regular right-hand side. 
 For the same reason 
 one can only hope for continuous (and not Lipschitz-continuous) dependance on the data. 
 We start with an approximation lemma in low regularity. It is the first step to overcome the 
 loss of derivatives. 
 \begin{lem}
   \label{LemmaApporixmationInLinearSituationInL2}
   Let $J \subset \R$ be an open interval, $t_0 \in \clJ$, and $\eta > 0$. 
   Take coefficients $A_{0,n}, A_0 \in \Fupdwl{3}{\operatorname{c}}{\eta}$, 
   $A_1, A_2 \in \Fcoeff{3}{cp}$, 
   $A_3  = A_3^{\operatorname{co}}$, and $D_n, D \in \Fuwl{3}{\operatorname{c}}$ for all $n \in \N$ 
   such that $(A_{0,n})_n$ and $(D_n)_n$ are bounded in $W^{1,\infty}(\Omega)$ and converge to $A_0$ 
   respectively $D$ in $L^\infty(\Omega)$. Set $B = B^{\operatorname{co}}$. 
   Choose data $u_0 \in \Ltwoh$, $g \in \E{0}$, and $f \in \Ltwoa$. Let $u_n$ denote the 
   weak solution of the linear initial boundary value problem~\eqref{IBVP}
   with differential operator 
   $L(A_{0,n},A_1, A_2, A_3, D_n)$ and these data  for all $n \in \N$
   and $u$ be the weak solution of~\eqref{IBVP} with differential operator 
   $L(A_0,\ldots, A_3,D)$ and the same data.
   Then $(u_n)_n$ converges to $u$ in $\G{0}$.
  \end{lem}

  \begin{proof}
   Without loss of generality we assume that $J = (0,T)$ for some $T > 0$ and $t_0 =0$. 
   Set $A_{0,0} = A_0$ and $D_0 = D$. Take $r > 0$ with 
   $\|A_{0,n}\|_{W^{1,\infty}(\Omega)} \leq r$ and $\|D_n\|_{W^{1,\infty}(\Omega)} \leq r$ for all $n \in \N_0$.
   
   I) We first assume that $u_0$ belongs to $\Hh{1}$, $g$ to $\E{1}$,  $f$ to $\Ha{1}$, and that 
   $B u_0 = g(0)$ on $\partial \R^3_+$. Then
   the tuples $(0, A_{0,n},A_1, A_2, A_3,B, D_{n}, f,g, u_0)$ 
   fulfill the linear compatibility conditions~\eqref{EquationCompatibilityConditionPrecised} of first order 
   for each $n \in \N_0$. The solutions $u_n$ and $u$ are thus contained in $\G{1}$ by 
   Theorem~\ref{TheoremExistenceAndUniquenessOnDomain}.
   The difference $u_n - u$ further solves the linear initial boundary value problem
   \begin{equation*}
    \left\{ 
      \begin{aligned}
       L(A_{0,n}, A_1, A_2, A_3, D_n)(u_n - u) &= f_n, \quad &&x \in \R^3_+, \quad &t \in J; \\
       B(u_n - u) &= 0, && x \in \partial \R^3_+, &t \in J; \\
	(u_n - u)(0) &= 0,	&& x \in \R^3_+;
      \end{aligned}
    \right.
   \end{equation*}
  where $f_n = (A_0 - A_{0,n}) \partial_t u + (D - D_n) u$ for all $n \in \N$. As $u$ is an element 
  of $\G{1}$, the right-hand side of the differential equation above 
  belongs to $\Ltwoa$. Theorem~\ref{TheoremExistenceAndUniquenessOnDomain}
  thus provides constants $\gamma = \gamma_{\ref{TheoremLocalExistenceNonlinear},0}(\eta,r)$ and 
  $C_0 = \max\{C_{\ref{TheoremLocalExistenceNonlinear},0,0}(\eta,r),C_{\ref{TheoremLocalExistenceNonlinear},0}(\eta,r,T)\}$ 
  such that
  \begin{align*}
   \Gnormwg{0}{u_n - u}^2 &\leq e^{2 \gamma T} \Gnorm{0}{u_n - u}^2 \leq C_0 e^{2 \gamma T} \Ltwoan{(A_0 \! - \! A_{0,n}) \partial_t u + (D \! - \! D_n) u}^2 \\
   &\leq 2 C_0 e^{2\gamma T}(\|A_{0,n} \! - \! A_0\|_{L^\infty(\Omega)}^2 \Ltwoan{\partial_t u}^2 + \|D_n \! - \! D\|_{L^\infty(\Omega)}^2 \Ltwoan{u}^2)
  \end{align*}
  for all $n \in \N$. Since $A_{0,n} \rightarrow A_0$ and $D_n \rightarrow D$ in $L^\infty(\Omega)$, 
  we conclude that the functions $u_n$ tend to $u$ in $\G{0}$ as $n \rightarrow \infty$.
  
  II) We now come to the general case. Take sequences $(f_j)_j$ in $\Ha{1}$,  $(g_j)_j$ in $\E{2}$, 
  and $(\tilde{u}_{0,j})_j$ in $C_c^\infty(\R^3_+)$ converging to $f$, $g$ and $u_0$ in $\Ltwoa$,  $\E{0}$, 
  and $\Ltwoh$ respectively. As $B$ is constant and has rank $2$, there is a sequence $(\tilde{h}_j)_j$ 
  in $\E{2}$ with $B \tilde{h}_j = g_j$  for all $j \in \N$. Extending 
  $\tilde{h}_j$ to $J \times \R^3_+$ via a suitable sequence of smooth cut-off functions 
  in $x_3$-direction, we obtain 
  functions $h_j$  in $C(\clJ, \Hh{1})$ such that $u_{0,j} = \tilde{u}_{0,j} + h_j(0)$ converges 
  to $0$ in $\Ltwoh$ as $j \rightarrow \infty$ and $B u_{0,j} = B h_j(0) = B \tilde{h}_j(0) = g_j(0)$ on $\partial \R^3_+$
  for all $j \in \N$. We refer to step~I) of the proof of Theorem~4.13 in~\cite{SpitzDissertation} for 
  the details of this construction.
  Note that the 
  tuples $(0,A_{0,n},A_1, A_2, A_3, D_n,B, f_j, g_j, u_{0,j})$ consequently fulfill the linear compatibility conditions~\eqref{EquationCompatibilityConditionPrecised}
  of order $1$ for all $n, j \in \N$.
  
  Let the function $u_n^j$ denote the weak solution of~\eqref{IBVP} 
  with differential operator $L(A_{0,n}, A_1, A_2, A_3, D_n)$ and data $f_j$,
  $g_j$, and $u_{0,j}$ as well as $u^j$ the weak solution of~\eqref{IBVP} with 
  differential operator $L(A_0, \ldots, A_3, D)$, and the same data for all $n,j \in \N$. 
  These solutions belong to $\G{1}$ by Theorem~\ref{TheoremExistenceAndUniquenessOnDomain}. 
  Observe that the difference $u_n^j - u_n$ solves~\eqref{IBVP} with differential operator 
  $L(A_{0,n},A_1, A_2, A_3, D_n)$, inhomogeneity $f_j - f$, boundary value $g_j - g$, 
  and initial value $u_{0,j} - u_0$, and the function $u^j - u$ solves~\eqref{IBVP} with 
  differential operator $L(A_0,A_1, A_2, A_3,D)$ and the same data. The a priori estimate in Theorem~\ref{TheoremExistenceAndUniquenessOnDomain} 
  thus shows
  \begin{align}
   &\Gnormwg{0}{u_n^j - u_n}^2 \leq e^{2 \gamma T} \Gnorm{0}{u_n^j - u_n}^2 \label{EquationEstimateForDifferenceunjun}\\
   &\leq C_0 \,e^{2 \gamma T} \Big(\Ltwohn{u_{0,j} - u_0}^2 + \Enorm{0}{g_j - g}^2 +  \Ltwoan{f_j - f}^2 \Big), \nonumber\\
   &\Gnormwg{0}{u^j - u}^2 \leq e^{2 \gamma T} \Gnorm{0}{u^j - u}^2 \label{EquationEstimateForDifferenceuju}\\
   &\leq C_0 \,e^{2 \gamma T} \Big( \Ltwohn{u_{0,j} - u_0}^2 + \Enorm{0}{g_j - g}^2 +  \Ltwoan{f_j - f}^2\Big) \nonumber
  \end{align}
  for all $n, j \in \N$, where $\gamma$ and $C_0$ were introduced in step I).

  Let $\epsilon > 0$. Because of the convergence of the data, 
  we find an index $j_0$ such that
  \begin{equation}
  \label{EquationFixingj0}
   C_0 \,e^{2 \gamma T} \Big(\Ltwohn{u_{0,j_0} - u_0}^2 + \Enorm{0}{g_{j_0} - g}^2 + \Ltwoan{f_{j_0} - f}^2\Big) \leq \frac{\epsilon^2}{9}.
  \end{equation}
  On the other hand, the tuple $(f_{j_0},g_{j_0}, u_{0,j_0})$ fulfills the assumptions of step I), which therefore implies
  $u_n^{j_0} \rightarrow u^{j_0}$ in $\G{0}$ as $n \rightarrow \infty$. Hence, there is an index $n_0 \in \N$ 
  such that 
  \begin{equation}
   \label{EquationEstimateDifferenceunjuj}
   \Gnormwg{0}{u_n^{j_0} - u^{j_0}} \leq \frac{\epsilon}{3}
  \end{equation}
  for all $n \geq n_0$. Combining~\eqref{EquationEstimateForDifferenceunjun} to~\eqref{EquationEstimateDifferenceunjuj}, 
  we arrive at
  \begin{align*}
   \Gnormwg{0}{u_n - u} &\leq \Gnormwg{0}{u_n - u_n^{j_0}} + \Gnormwg{0}{u_n^{j_0} - u^{j_0}} + \Gnormwg{0}{u^{j_0} - u} \leq  \epsilon
  \end{align*}
  for all $n \geq n_0$.
  \end{proof}
  
  The next lemma contains the core of the proof for the continuous dependance. It states that given
  a sequence of data converging in $H^m$ respectively $E_m$ and assuming that the sequence of 
  corresponding solutions of~\eqref{EquationNonlinearIBVP} converges in $G_{m-1}$, then the solutions 
  also converge in $G_m$. The proof involves techniques developed for the full space
  (see e.g.~\cite{BahouriCheminDanchin}) 
  which prevents to lose regularity because of the quasilinearity. Using methods
  from~\cite{SpitzMaxwellLinear}, we again exploit
  the structure of Maxwell's equations to avoid the loss of a derivative due to the characteristic boundary.
  \begin{lem}
  \label{LemmaConvergenceOfNonlinearSolutionInGmProvidedConvergenceInGmminus1}
   Let $J' \subseteq \R$ be an open and bounded interval, $t_0 \in \overline{J'}$, and $m \in \N$ with $m \geq 3$.
   Take functions $\chi \in \mlpdwl{m,6}{G}{\mathcal{U}}{c}$ and $\sigma \in \mlwl{m,6}{G}{\mathcal{U}}{c}$. Set
   \begin{align*}
   	B(x) = \begin{pmatrix}
             0 &\nu_3(x) &-\nu_2(x) &0 &0 &0 \\
      -\nu_3(x) &0 &\nu_1(x) &0 &0 &0 \\
      \nu_2(x) &-\nu_1(x) &0 &0 &0 &0 
            \end{pmatrix},
   \end{align*}
   for all $x \in \partial G$, where $\nu$ denotes the unit outer normal vector of $\partial G$. 
   Choose $f_n, f \in H^m(J' \times G)$, $g_n, g \in E_m(J' \times \partial G)$, 
   and $u_{0,n}, u_0 \in \Hhdom{m}$ 
   for all $n \in \N$ with 
   \begin{align*}
    \Hhndom{m}{u_{0,n} - u_0} \longrightarrow 0, \quad \|g_n - g\|_{E_m(J' \times \partial G)} \longrightarrow 0, \quad \|f_n - f\|_{H^m(J' \times G)} \longrightarrow 0,
   \end{align*}
   as $n \rightarrow \infty$.
   We further assume that~\eqref{EquationNonlinearIBVP} 
   with data $(t_0, f_n, g_n, u_{0,n})$ and $(t_0, f, g, u_0)$ have $G_m(J' \times G)$-solutions
   $u_n$ and $u$ for all $n \in \N$, that there is a compact subset 
   $\tilde{\mathcal{U}}_1$ of $\mathcal{U}$ 
   with $\image u(t) \subseteq \tilde{\mathcal{U}}_1$ for all $t \in J'$, that $(u_n)_n$ is bounded in $G_m(J' \times G)$, and that $(u_n)_n$ converges to $u$ 
   in $G_{m-1}(J' \times G)$.
   
   Then the functions $u_n$ converge to $u$ in $G_m(J' \times G)$.
  \end{lem}
  
  \begin{proof}
   Without loss of generality we assume that $t_0 = 0$, that $J' = (0, T')$ for a number $T' > 0$, and 
   that $\chi$ and $\sigma$ fulfill~\eqref{EquationPropertyForL2}, cf. Remark~\ref{RemarkChiuInFm}. 
   The proof is again reduced to the half-space case $G = \R^3_+$ via local charts. We do not give the 
   details of the localization procedure here but assume as   
   in the proof of Proposition~\ref{PropositionEstimateForNonlinearSolutionOnDomain} that $G = \R^3_+$ and that we have 
   spatial coefficients $A_1, A_2 \in \Fcoeff{m}{cp}$ $A_3 = A_3^{\operatorname{co}}$, and $B = B^{\operatorname{co}}$.
   We refer to~\cite[Section~2]{SpitzMaxwellLinear} and~\cite[Chapter~5]{SpitzDissertation} for the details.

   Let $T \in (0, T']$, $J = (0,T)$, and $\Omega = J \times \R^3_+$. 
   Sobolev's embedding yields a constant $C_S$ depending on the length of the interval $J'$ such that
   \begin{equation}
   \label{EquationConvergenceOffAndTimeDerivativesInZero}
    \sum_{j=0}^{m-1} \Hhn{m-j-1}{\partial_t^j f_n(0) - \partial_t^j f(0)} \leq m C_S \|f_n - f\|_{H^m(J' \times \R^3_+)} \longrightarrow 0
   \end{equation}
   as $n \rightarrow \infty$.
   We set $\overline{\N} = \N \cup \{\infty\}$, $u_\infty = u$, $f_\infty = f$, $g_\infty = g$, and $u_{0,\infty} = u_0$. 
   Throughout, let $n \in \overline{\N}$ and $\alpha \in \N_0^4$ with $|\alpha| \leq m$.
   By assumption, \eqref{EquationConvergenceOffAndTimeDerivativesInZero}, and Sobolev's embedding there is a 
   radius $r > 0$ such that 
   \begin{align}
   \label{EquationBoundForun}
   	&\| u_n \|_{G_m(J' \times \R^3_+)} + \| u_n \|_{L^\infty(J' \times \R^3_+)} + \sum_{j = 1}^2 \|A_j\|_{F_m(J' \times \R^3_+)} \leq r, \\
	&\sum_{j = 0}^{m-1} \Hhn{m-j-1}{\partial_t^j f_n(0)} + \Hhn{m}{u_{0,n}} + \|g_n\|_{E_m(J' \times \R^3_+)} + \|f_n\|_{H^m(J' \times \R^3_+)} \leq r. \label{EquationBoundForData} 	
   \end{align}
   Moreover,  $(u_n)_n$ converges 
   to $u$ in $L^\infty(J' \times G)$ so that there is a compact and connected set 
   $\mathcal{U}_1 \subseteq \mathcal{U}$ and an index $n_0$ such that 
   $\image u_n(t) \subseteq \mathcal{U}_1$ for all $t \in J'$ and $n \geq n_0$. 
   Without loss of generality we assume $n_0 = 1$. 
   Lemma~\ref{LemmaHigherOrderChainRule}~\ref{ItemHigherOrderChainRuleInG} then shows that 
   $\chi(u_n)$ and $\sigma(u_n)$ belong to $F_m^{\operatorname{c}}(J' \times \R^3_+)$ 
   and that there is a radius $R = R(\chi, \sigma,m, r,\mathcal{U}_1)$ with
   \begin{align}
   \label{EquationBoundForchiunsigmaun}
   	&\|\chi(u_n)\|_{F_m(J' \times \R^3_+)} + \| \sigma(u_n) \|_{F_m(J' \times \R^3_+)} \leq R.
   \end{align}
  Furthermore, $\chi(u_n)$ is symmetric and positive definite with $\chi(u_n) \geq \eta$.
  We use the operators and right-hand sides
   \begin{align}
	L_n &= L(\chi(u_n), A_1, A_2,A_3, \sigma(u_n)) , \nonumber \\
   	\label{EquationDefinitionfalphan}
   	f_{\alpha,n} &= \partial^\alpha f_n - \!\sum_{0 < \beta \leq \alpha} \binom{\alpha}{\beta} \partial^\beta \chi(u_n) \partial^{\alpha-\beta} \partial_t u_n - \sum_{j = 1}^2 \sum_{0 < \beta \leq \alpha} \binom{\alpha}{\beta} \partial^\beta A_j \partial^{\alpha - \beta} \partial_j u_n \nonumber \\
   	&\quad - \sum_{0 < \beta \leq \alpha} \binom{\alpha}{\beta} \partial^\beta \sigma(u_n) \partial^{\alpha - \beta} u_n.    
   \end{align}
  As in~\cite{SpitzMaxwellLinear}, the function $\partial^\alpha u_n$ then solves the linear initial value problem
   \begin{equation}
   \label{EquationInitialValueProblemForDerivative}
   		\left\{
   		\begin{aligned}
   			L_n v &= f_{\alpha,n}, \qquad && x \in \R^3_+, \quad && t \in J; \\
   			v(0) &= \partial^{(0,\alpha_1,\alpha_2,\alpha_3)} S_{\chi,\sigma,\R^3_+,m,\alpha_0}(0,f_{n},u_{0,n}), &&x \in \R^3_+;
   		\end{aligned} \right.
   \end{equation}
   and it is the solution 
   of the linear initial boundary value problem
   \begin{equation}
   \label{EquationInitialBoundaryValueProblemForDerivative}
   		\left\{
   		\begin{aligned}
   			L_n v &= f_{\alpha,n}, \qquad && x \in \R^3_+, \quad && t \in J; \\
   			B v &= \partial^\alpha g_{n}, &&x \in \partial\R^3_+, &&t \in J; \\
   			v(0) &= \partial^{(0,\alpha_1,\alpha_2,0)} S_{\chi,\sigma,\R^3_+,m,\alpha_0}(0,f_{n},u_{0,n}), &&x \in \R^3_+;
   		\end{aligned} \right.
   \end{equation}
   if also  $\alpha_3 = 0$. Here we exploited that $A_3$ and $B$ are constant.
   
   I) To estimate $f_{\alpha,n}$ and $f_{\alpha,n} - f_{\alpha, \infty}$, we introduce the quantity
   \begin{align*}
    h_n(t) &= \sum_{i=1}^3 \sum_{0 \leq j \leq m} \, \sum_{\substack{0 \leq \gamma \leq \alpha, \gamma_0 = 0 \\ |\gamma| = m-j}} \, \sum_{l_1, \ldots, l_{j} = 1}^6 \|(\partial_{y_{l_{j}}} \ldots \partial_{y_{l_1}} \partial_x^{(\gamma_1,\gamma_2,\gamma_3)}\theta_i)(u_n(t)) \nonumber \\
  &\hspace{14em} - (\partial_{y_{l_{j}}} \ldots \partial_{y_{l_1}} \partial_x^{(\gamma_1,\gamma_2,\gamma_3)}\theta_i)(u(t))\|_{L^\infty(\R^3_+)}
   \end{align*}
   for all $t \in \overline{J'}$ and $n \in \N$, where $\theta_1 = \chi$, $\theta_2 = \sigma$, and $\theta_3$ is the matrix inverse of $\chi$, cf. Corollary~\ref{CorollaryEstimateForDifferenceHigherOrder}.
   Recall that $(u_n)_n$ tends to $u$ uniformly as $n \rightarrow \infty$ and that these functions map 
   in the compact set $\mathcal{U}_1$. It follows
   \begin{equation}
    \label{EquationConvergenceToZeroOfhn} 
    h_n(t) \longrightarrow 0 \quad \text{for all } t \in \overline{J'} \qquad \text{and} \qquad  \int_{0}^{T'} h_n^2(t) dt \longrightarrow 0
   \end{equation}
   as $n \rightarrow \infty$. Using Lemma~2.1 of~\cite{SpitzMaxwellLinear} and Corollary~\ref{CorollaryEstimateForDifferenceHigherOrder} 
   we derive the bounds
  \begin{align}
   \label{EquationEstimateForDifferencefalphanfalphainfty}
   &\Ltwoanwgamma{f_{\alpha,n}} \leq C(\chi,\sigma,m,r,\mathcal{U}_1,T'), \nonumber\\
   &\Ltwoanwgamma{f_{\alpha,n} - f_{\alpha,\infty}}^2 = \int_{0}^T \Ltwohn{f_{\alpha,n}(s) - f_{\alpha,\infty}(s)}^2 ds \nonumber\\
   &\leq C(\chi,\sigma,m,r,\mathcal{U}_1,T')\Big(\Han{m}{f_n - f}^2 + \Gnormwg{m-1}{u_n - u}^2 + \delta_{|\alpha| m}\int_0^T h_n^2(s) ds \nonumber\\
   &\hspace{11em} +\int_0^T \sum_{\tilde{\alpha} \in \N_0^4, |\tilde{\alpha}| = m} \Ltwohn{\partial^{\tilde{\alpha}} u_n(s) - \partial^{\tilde{\alpha}} u(s)}^2 ds \Big).
  \end{align}
  Let $|\alpha| \leq m-1$. Using also~\eqref{EquationfalphakpartialkfalphaMinusErrorTerms}, we then obtain
   \begin{align}
   \label{EquationEstimateForDifferencefalphanfalphainftyInG0}
   &\Gnormwg{0}{f_{\alpha,n}} \leq C(\chi,\sigma,m,r,\mathcal{U}_1), \nonumber\\
   &\Gnormwg{0}{f_{\alpha,n} - f_{\alpha,\infty}} \leq \Gnormwg{m-1}{f_{n} - f} \! + \! C(\chi,\sigma,m,r,\mathcal{U}_1) \Gnormwg{m-1}{u_n - u}, \\
   \label{EquationEstimateForDifferencefalphanfalphainftyInH1}
   &\Han{1}{f_{\alpha,n}} \leq C(\chi,\sigma,m,r,\mathcal{U}_1,T'), \nonumber\\
   &\Han{1}{f_{\alpha,n} - f_{\alpha,\infty}}^2 \leq C(\chi,\sigma,m,r,\mathcal{U}_1,T')\Big(\Han{m}{f_n - f}^2 + \Gnormwg{m-1}{u_n - u}^2 \nonumber\\
   &\hspace{1em} + \delta_{|\alpha| m-1} \int_0^T h_n^2(s) ds +\int_0^T \sum_{\tilde{\alpha} \in \N_0^4, |\tilde{\alpha}| = m} \Ltwohn{\partial^{\tilde{\alpha}} u_n(s) - \partial^{\tilde{\alpha}} u(s)}^2 ds \Big).
  \end{align}
  (See~\cite{SpitzDissertation} for further details.)

  II) We now look at the tangential case $\alpha_3 = 0$. To split $\partial^\alpha u_n$, we define the functions
  \begin{align*}
   w_{0,n} = \partial^{(0,\alpha_1,\alpha_2,0)} S_{\chi,\sigma,\R^3_+, m,\alpha_0}(0,f_{n},u_{0,n}),
  \end{align*}
  which belong to $\Ltwoh$ by Lemma~\ref{LemmaNonlinearHigherOrderInitialValues}.
  Consider the linear initial boundary value problems
  \begin{equation}
   \label{EquationInitialBoundaryValueProblemForDerivativeWithFixedRightHandSide}
   		\left\{
   		\begin{aligned}
   			L_n v &= f_{\alpha,\infty}, \qquad && x \in \R^3_+, \quad && t \in J; \\
   			B v &= \partial^\alpha g_\infty, &&x \in \partial\R^3_+, &&t \in J; \\
   			v(0) &= w_{0,\infty}, &&x \in \R^3_+;
   		\end{aligned} \right.
   \end{equation}
   and
   \begin{equation}
   \label{EquationInitialBoundaryValueProblemForDerivativeWithDifferenceRightHandSide}
   		\left\{
   		\begin{aligned}
   			L_n v &= f_{\alpha,n} - f_{\alpha,\infty}, \qquad && x \in \R^3_+, \quad && t \in J; \\
   			B v &= \partial^\alpha g_{n} - \partial^\alpha g_\infty, &&x \in \partial\R^3_+, &&t \in J; \\
   			v(0) &= w_{0,n} - w_{0,\infty}, &&x \in \R^3_+.
   		\end{aligned} \right.
   \end{equation}
   Because of the above regularity statements,
   Theorem~\ref{TheoremExistenceAndUniquenessOnDomain} implies 
   that the problem~\eqref{EquationInitialBoundaryValueProblemForDerivativeWithFixedRightHandSide} has a unique 
   solution $w_n$ in $\G{0}$, the problem~\eqref{EquationInitialBoundaryValueProblemForDerivativeWithDifferenceRightHandSide} 
   has a unique solution $z_n$ in $\G{0}$, and that the sum $w_n + z_n$ uniquely solves~\eqref{EquationInitialBoundaryValueProblemForDerivative}.
   We thus obtain
   \begin{align}
   \label{Equationwnplusznequalspartialalphaun}
    w_n + z_n = \partial^\alpha u_n.
   \end{align}
   We point out that in the case $n = \infty$ the initial boundary value problems~\eqref{EquationInitialBoundaryValueProblemForDerivativeWithFixedRightHandSide} 
   and~\eqref{EquationInitialBoundaryValueProblemForDerivative} coincide. Since the latter is solved by 
   $\partial^\alpha u_n$ and solutions of that problem are unique by Theorem~\ref{TheoremExistenceAndUniquenessOnDomain}, we 
   conclude that
   \begin{align}
   \label{Equationwinftyequalspartialalphau}
    w_\infty = \partial^\alpha u_\infty = \partial^\alpha u.
   \end{align} 
   Since $(u_n)_n$ tends to $u$ uniformly and these functions map into the compact set $\mathcal{U}_1$, 
   we obtain the limits
   \begin{align*}
    \|\chi(u_n) - \chi(u)\|_{L^\infty(\Omega)} + \|\sigma(u_n) - \sigma(u)\|_{L^\infty(\Omega)} \leq  C(\chi,\sigma,\mathcal{U}_1) \Gnormwg{m-1}{u_n - u} \longrightarrow 0
   \end{align*}
   as $n \rightarrow \infty$. In view of~\eqref{EquationBoundForchiunsigmaun},
   Lemma~\ref{LemmaApporixmationInLinearSituationInL2} therefore tells us that
   \begin{align}
    \label{EquationConvergenceOfwnTowinfty}
      \Gnormwg{0}{w_n - \partial^\alpha u} = \Gnormwg{0}{w_n - w_\infty} \longrightarrow 0
   \end{align}
   as $n \rightarrow \infty$. 
   Define $\gamma = \gamma(\chi, \sigma,m, r,\mathcal{U}_1, T') \geq 1$ by
   \begin{align*}
    \gamma = \gamma_{\ref{TheoremExistenceAndUniquenessOnDomain},0}(\eta(\chi), R(\chi, \sigma, m, r,\mathcal{U}_1),T'),
   \end{align*}
   where $\gamma_{\ref{TheoremExistenceAndUniquenessOnDomain}, 0}$ is the corresponding constant from Theorem~\ref{TheoremExistenceAndUniquenessOnDomain}.
   The estimate from this theorem applied to~\eqref{EquationInitialBoundaryValueProblemForDerivativeWithDifferenceRightHandSide}
   then yields
   \begin{align}
   \label{EquationEstimateForzn}
    &\Gnormwg{0}{z_n}^2 \leq e^{2 \gamma T} \Gnorm{0}{z_n}^2 \\
    &\leq C_{0} e^{2\gamma T'} \Big( \Ltwohn{w_{0,n} \! - \! w_{0,\infty}}^2 \!+\! \Enorm{0}{\partial^\alpha g_{n} \! - \! \partial^\alpha g_\infty}^2
     \! + \! \Ltwoan{f_{\alpha,n} \! - \! f_{\alpha,\infty}}^2 \Big) \nonumber
   \end{align}
   where $C_{0}(\chi,\sigma,m,r, \mathcal{U}_1,T')$ is the maximum of the constants $C_0$ and $C_{0,0}$ appearing in 
   Theorem~\ref{TheoremExistenceAndUniquenessOnDomain}.
   Because of~\eqref{EquationBoundForData}, 
   Lemma~\ref{LemmaNonlinearHigherOrderInitialValues} provides a constant
   $C_{\ref{LemmaNonlinearHigherOrderInitialValues}} = C_{\ref{LemmaNonlinearHigherOrderInitialValues}}(\chi,\sigma,m,r,\mathcal{U}_1)$ 
   such that
   \begin{align*}
    &\Ltwohn{w_{0,n} - w_{0,\infty}} \\
    &= \Ltwohn{\partial^{(0,\alpha_1, \alpha_2, 0)} S_{\chi,\sigma,\R^3_+,m,\alpha_0}(0,f_n,u_{0,n}) 
	- \partial^{(0,\alpha_1, \alpha_2, 0)} S_{\chi,\sigma,\R^3_+,m,\alpha_0}(0,f,u_{0})} \\
    &\leq C_{\ref{LemmaNonlinearHigherOrderInitialValues}} 
      \Big( \sum_{j = 0}^{m-1} \Hhn{m-j-1}{\partial_t^j f_n(0) - \partial_t^j f(0)} + \Hhn{m}{u_{0,n} - u_0}\Big)
   \end{align*}
   for all $n \in \N$. Inserting this estimate together with~\eqref{EquationEstimateForDifferencefalphanfalphainfty} 
   into~\eqref{EquationEstimateForzn}, we derive
   \begin{align*}
    \Gnormwg{0}{z_n}^2 &\leq C_{\ref{EquationFinalEstimateForzn}} 
    \Big( \sum_{j = 0}^{m-1} \Hhn{m-j-1}{\partial_t^j f_n(0) - \partial_t^j f(0)}^2 + \Hhn{m}{u_{0,n} - u_0}^2 \\
      & \hspace{4em} + \Enormwg{m}{g_n - g}^2 + \Han{m}{f_n - f}^2 + \Gnormwg{m-1}{u_n - u}^2 \Big) \\
      &\quad +C_{\ref{EquationFinalEstimateForzn}} \int_0^T h_n^2(s) ds +  C_{\ref{EquationFinalEstimateForzn}}  \int_0^T \! \! \sum_{\tilde{\alpha} \in \N_0^4, |\tilde{\alpha}| = m} \! \! \Ltwohn{\partial^{\tilde{\alpha}} u_n(s) - \partial^{\tilde{\alpha}} u(s)}^2 ds, 
   \end{align*}
    for all $n \in \N$, where we introduce a constant $C_{\ref{EquationFinalEstimateForzn}} = C_{\ref{EquationFinalEstimateForzn}}(\chi, \sigma,m,r,\mathcal{U}_1,T')$. 
    We write $a_n'$ for the first 
    part of the above right-hand side. It follows
    \begin{align}
    \label{EquationFinalEstimateForzn}
     \Gnormwg{0}{z_n}^2 \leq a_n' + C_{\ref{EquationFinalEstimateForzn}} \int_0^T \sum_{\tilde{\alpha} \in \N_0^4, |\tilde{\alpha}| = m} \Ltwohn{\partial^{\tilde{\alpha}} u_n(s) - \partial^{\tilde{\alpha}} u(s)}^2 ds,
    \end{align}
    for all $n \in \N$. Observe that $a_n'$ converges to $0$ as $n \rightarrow \infty$ by our assumptions,
    \eqref{EquationConvergenceOffAndTimeDerivativesInZero}, 
    and~\eqref{EquationConvergenceToZeroOfhn}. Formula~\eqref{Equationwnplusznequalspartialalphaun} and 
    inequality~\eqref{EquationFinalEstimateForzn} imply that
    \begin{align}
     \label{EquationForpartialalphaunminuspartialalphauForAlphaWithNormalComponentZero}
     &\Gnormwg{0}{\partial^\alpha u_n - \partial^\alpha u}^2 = \Gnormwg{0}{w_n + z_n - \partial^\alpha u}^2
     \leq 2\Gnormwg{0}{w_n - \partial^\alpha u}^2 + 2 \Gnormwg{0}{z_n}^2 \nonumber\\
     &\leq 2 \Gnormwg{0}{w_n - \partial^\alpha u}^2 + 2 a_n' + 2 C_{\ref{EquationFinalEstimateForzn}} \int_0^T \sum_{\tilde{\alpha} \in \N_0^4, |\tilde{\alpha}| = m} \Ltwohn{\partial^{\tilde{\alpha}} u_n(s) - \partial^{\tilde{\alpha}} u(s)}^2 ds \nonumber\\
     &= a_{\alpha,n} + C_{\ref{EquationForpartialalphaunminuspartialalphauForAlphaWithNormalComponentZero}} 
	\int_0^T \sum_{\tilde{\alpha} \in \N_0^4, |\tilde{\alpha}| = m} \Ltwohn{\partial^{\tilde{\alpha}} u_n(s) - \partial^{\tilde{\alpha}} u(s)}^2 ds,
    \end{align}
    for all $n \in \N$. Here we set $C_{\ref{EquationForpartialalphaunminuspartialalphauForAlphaWithNormalComponentZero}} = C_{\ref{EquationForpartialalphaunminuspartialalphauForAlphaWithNormalComponentZero}}(\chi,\sigma,m,r,\mathcal{U}_1,T')$
    and note that
    \begin{align*}
     a_{\alpha,n}  := 2 \Gnormwg{0}{w_n - \partial^\alpha u}^2 + 2 a_n' \longrightarrow 0
    \end{align*}
    as $n \rightarrow \infty$ by~\eqref{EquationConvergenceOfwnTowinfty}.

    III) We claim that for all multiindices $\alpha \in \N_0^4$ with $|\alpha| = m$ there is a sequence
    $(a_{\alpha,n})_n$ and a constant $C_\alpha = C_\alpha(\chi,\sigma,m,r,\mathcal{U}_1,T')$ 
    such that
    \begin{equation}
    \label{EquationClaimForDifferenceOfAlphaDerivativeunu}
	  \Gnormwg{0}{\partial^\alpha u_n - \partial^\alpha u}^2 \leq a_{\alpha,n} + C_\alpha \int_0^T \sum_{\tilde{\alpha} \in \N_0^4, |\tilde{\alpha}| = m} \Ltwohn{\partial^{\tilde{\alpha}} u_n(s) - \partial^{\tilde{\alpha}} u(s)}^2 ds
    \end{equation}
    for all $n \in \N$ and 
    \begin{align}
     \label{EquationClaimConvergenceToZeroOfaalphan}
     a_{\alpha,n} \longrightarrow 0
    \end{align}
    as $n \rightarrow \infty$. One proves
    this assertion by induction with respect to $\alpha_3$. Observe that step~II) shows the claim for 
    the case $\alpha_3 = 0$.
    In the induction step one assumes that there is 
    an index $l \in \{1, \ldots, m\}$ such that the assertion is true for all $\alpha \in \N_0^4$ with 
    $|\alpha| = m$ and $\alpha_3 = l-1$. Take $\alpha \in \N_0^4$ with $|\alpha| = m$ and $\alpha_3 = l$. 
    We set $\alpha' = \alpha - e_3$.
    
    Unfortunately we cannot directly apply Lemma~\ref{LemmaCentralEstimateInNormalDirection} here, 
    since it was derived for a fixed differential operator. 
    If we apply only one such 
    operator to a difference of solutions we experience the typical loss of derivatives. Therefore, 
    one has to repeat the key step of the proof of Lemma~\ref{LemmaCentralEstimateInNormalDirection} 
    for the operators $L_n$ and the difference 
    $\partial^{\alpha'}u_n - \partial^{\alpha'} u$. 
    In this calculation we use results from step~II) such as estimate~\eqref{EquationEstimateForDifferencefalphanfalphainftyInH1}.
    Since in this very lengthy reasoning essentially the same arguments are employed as in~\cite{SpitzMaxwellLinear}, 
    we decided to omit these calculations here. The details can be found in steps~III) to~V) of the proof  
    of Lemma~7.22 in~\cite{SpitzDissertation}.
    	
	We define $a_{n}$ and $C_m = C_m(\chi,\sigma,r,\mathcal{U}_1,T')$ by
	\begin{align*}
		a_n&= \sum_{\tilde{\alpha} \in \N_0^4, |\tilde{\alpha}| = m} a_{\tilde{\alpha},n},\hspace{4em}
		C_m = \sum_{\tilde{\alpha} \in \N_0^4, |\tilde{\alpha}| = m} C_{\tilde{\alpha}},
	\end{align*}
	for all $n \in \N$. Summing~\eqref{EquationClaimForDifferenceOfAlphaDerivativeunu} over 
	all multiindices $\alpha \in \N_0^4$ with $|\alpha| = m$, we then derive
	\begin{align*}
		&\sum_{\tilde{\alpha} \in \N_0^4, |\tilde{\alpha}| = m} \Ltwohn{\partial^{\tilde{\alpha}} u_n(T) - \partial^{\tilde{\alpha}} u(T)}^2 \leq \sum_{\tilde{\alpha} \in \N_0^4, |\tilde{\alpha}| = m} \Gnormwg{0}{\partial^{\tilde{\alpha}} u_n - \partial^{\tilde{\alpha}} u}^2 \\
		&\leq a_{n} + C_m \int_0^T \sum_{\tilde{\alpha} \in \N_0^4, |\tilde{\alpha}| = m} \Ltwohn{\partial^{\tilde{\alpha}} u_n(s) - \partial^{\tilde{\alpha}} u(s)}^2 ds 
	\end{align*}
	for all $n \in \N$. Since $T \in (0,T']$ was arbitrary, Gronwall's lemma shows that
	\begin{align*}
		\sum_{\tilde{\alpha} \in \N_0^4, |\tilde{\alpha}| = m} \Ltwohn{\partial^{\tilde{\alpha}} u_n(T) - \partial^{\tilde{\alpha}} u(T)}^2 \leq a_{n} e^{C_m T}
	\end{align*}
	for all $T \in [0,T']$ and $n \in \N$. As $(a_n)_n$ converges to $0$ due to~\eqref{EquationClaimConvergenceToZeroOfaalphan}, we finally arrive at
	\begin{align*}
		\sum_{\tilde{\alpha} \in \N_0^4, |\tilde{\alpha}| = m} \|\partial^{\tilde{\alpha}} u_n - \partial^{\tilde{\alpha}} u\|_{G_0(J' \times \R^3_+)}^2 \leq a_{n} e^{C_m T'} \longrightarrow 0
	\end{align*}
	as $n \rightarrow \infty$. Since $\|u_n - u\|_{G_{m-1}(J' \times \R^3_+)}$ tends to zero
	as $n \rightarrow \infty$ by assumption, 
	we conclude that $(u_n)_n$ converges to $u$ in $G_{m}(J' \times \R^3_+)$.    
  \end{proof}
  
  Finally, we can prove the full local wellposedness theorem. 
  In the following we will write $B_{M}(x,r)$ for the ball of radius $r$ around a point $x$ from a
  metric space~$M$. For times $t_0 < T$ we further set 
  \begin{align*}
     &M_{\chi,\sigma,m}(t_0,T) = \{(\tilde{f},\tilde{g}, \tilde{u}_0) \in H^m((t_0,T) \times G) \times E_m((t_0,T) \times \partial G) \times \Hhdom{m} \colon \\
	&\hspace{15.5em} (\chi, \sigma, t_0,B, \tilde{f}, \tilde{g}, \tilde{u}_0) \text{ is compatible of order } m\}, \\
      &d((\tilde{f}_1, \tilde{g}_1, \tilde{u}_{0,1}), (\tilde{f}_2, \tilde{g}_2, \tilde{u}_{0,2})) \nonumber \\
      & = \max \{\|\tilde{f}_1 - \tilde{f}_2 \|_{H^m((t_0,T) \times G)}, \|\tilde{g}_1 - \tilde{g}_2\|_{E_m((t_0, T) \times \partial G)}, \Hhndom{m}{\tilde{u}_{0,1} - \tilde{u}_{0,2}}\}.
    \end{align*}

  \begin{theorem}
   \label{TheoremLocalWellposednessNonlinear}
   Let $m \in \N$ with $m \geq 3$ and fix $t_0 \in \R$. Take functions $\chi \in \mlpdwl{m,6}{G}{\mathcal{U}}{c}$ and 
   $\sigma \in \mlwl{m,6}{G}{\mathcal{U}}{c}$ and set
   \begin{align*}
   	B(x) = \begin{pmatrix}
             0 &\nu_3(x) &-\nu_2(x) &0 &0 &0 \\
      -\nu_3(x) &0 &\nu_1(x) &0 &0 &0 \\
      \nu_2(x) &-\nu_1(x) &0 &0 &0 &0 
            \end{pmatrix}
   \end{align*}
   for all $x \in \partial G$, where $\nu$ denotes the unit outer normal vector of $\partial G$.
   Choose data $u_0 \in \Hhdom{m}$, $g \in E_m((-T,T) \times \partial G)$, 
   and $f \in H^m((-T,T) \times G)$ 
   for all $T > 0$ such that $\overline{\image u_0} \subseteq \mathcal{U}$ and 
   the tuple $(\chi,\sigma,t_0,B,f,g,u_0)$ fulfills the compatibility 
   conditions~\eqref{EquationNonlinearCompatibilityConditions}
   of order $m$.
   For the maximal existence times from Definition~\ref{DefinitionMaximalIntervalOfExistence}
   we then have
   \begin{align*}
    T_+ &= T_+(m, t_0, f,g, u_0) = T_+(k, t_0, f,g, u_0), \\
    T_{-} &= T_{-}(m, t_0, f,g, u_0) = T_{-}(k, t_0, f,g, u_0)
   \end{align*}
   for all $k \in \{3, \ldots, m\}$.    
   The following assertions are true.
   \begin{enumerate}[leftmargin = 2em]
    \item \label{ItemUniqueSolution} There exists a unique maximal solution $u$ of~\eqref{EquationNonlinearIBVP}
    which belongs to the function space $\bigcap_{j = 0}^m C^{j}((T_{-}, T_+), \Hhdom{m-j})$.
    \item \label{ItemBlowUp} If $T_+ < \infty$, then
	\begin{enumerate}
	 \item \label{ItemLeavingEveryCompactSet} the solution $u$ leaves every compact subset of $\mathcal{U}$, or
	 \item \label{ItemWaveBreaking} $\limsup_{t \nearrow T_+} \|\nabla u(t)\|_{L^{\infty}(G)} = \infty$.
	\end{enumerate}
	The analogous result holds for $T_{-}$.
    \item \label{ItemContinuousDependence} 
    Let $T' \in (t_0, T_+)$. Then there is a number 
    $\delta > 0$ such that for all data $\tilde{f} \in H^m((t_0,T_+) \times G)$, $\tilde{g} \in E_m((t_0, T_+) \times \partial G)$, 
    and $\tilde{u}_0 \in \Hhdom{m}$ 
    fulfilling
    \begin{equation*}
    \|\tilde{f} - f\|_{H^m((t_0,T_+)\times G)} < \delta, \quad
    \|\tilde{g} - g\|_{E_m((t_0, T_+) \times \partial G)} < \delta,\quad  \Hhn{m}{\tilde{u}_0 - u_0} < \delta
    \end{equation*}
    and the compatibility conditions~\eqref{EquationNonlinearCompatibilityConditions} of order $m$, 
    the maximal existence time satisfies $T_+(m, t_0, \tilde{f},\tilde{g}, \tilde{u}_0) > T'$. 
    We denote by $u(\cdot; \tilde{f}, \tilde{g}, \tilde{u}_0)$ the corresponding maximal 
    solution of~\eqref{EquationNonlinearIBVP}.
    The flow map
    \begin{align*}
     \Psi\colon B_{M_{\chi,\sigma,m}(t_0, T_+)}((f,g,u_0), \delta)  &\rightarrow G_m((t_0,T') \times G) , \\
     (\tilde{f},\tilde{g}, \tilde{u}_0) &\mapsto u(\cdot; \tilde{f}, \tilde{g}, \tilde{u}_0), 
    \end{align*}
    is continuous. Moreover, there is a constant 
    $C = C(\chi, \sigma,m,r,T_+ - t_0, \kappa_0)$ such that
    \begin{align}
     \label{EquationLipschitzEstimateInContinuousDependenceAssertion}
     &\|\Psi(\tilde{f}_1, \tilde{g}_1, \tilde{u}_{0,1}) - \Psi(\tilde{f}_2, \tilde{g}_2, \tilde{u}_{0,2})\|_{G_{m-1}((t_0,T') \times G)} \nonumber \\
     &\leq C \sum_{j = 0}^{m-1} \Hhndom{m-j-1}{\partial_t^j \tilde{f}_1(t_0) - \partial_t^j \tilde{f}_2(t_0)} + C \|\tilde{g}_1 - \tilde{g}_2\|_{E_{m-1}((t_0,T') \times \partial G)} \nonumber\\
     &\qquad + C \Hhndom{m}{\tilde{u}_{0,1} - \tilde{u}_{0,2}} +C \| \tilde{f}_1 - \tilde{f}_2 \|_{H^{m-1}((t_0,T') \times G)}
    \end{align}
    for all $(\tilde{f}_1, \tilde{g}_1, \tilde{u}_{0,1}), (\tilde{f}_2, \tilde{g}_2, \tilde{u}_{0,2}) \in B_{M_{\chi,\sigma,m}(t_0, T_+)}((f,g,u_0), \delta)$, 
    where $\kappa_0 = \dist(\image u_0, \partial \mathcal{U})$.
    The analogous result is true for $T_{-}$.
   \end{enumerate}
  \end{theorem}
  
  \begin{proof}
  We show the assertion for $T_+$, the proofs for $T_{-}$ are analogous.
    Let $k \in \{3,\ldots, m-1\}$. We have $T_+ = T_+(m, t_0, f,g, u_0) \leq T_+(k, t_0, f,g, u_0)$ by definition. 
    Assume now that $T_+ < T_+(k, t_0, f,g, u_0)$. Then $T_+ < \infty$ and the maximal $\Hhdom{m}$-solution $u$ of~\eqref{EquationNonlinearIBVP}, 
    which exists on $(t_0, T_+)$, can be extended to a $\Hhdom{k}$-solution on $(t_0, T_+(k,t_0,f,g,u_0))$ by the 
    definition of the maximal existence time and Lemma~\ref{LemmaUniquenessOfNonlinearSolution}. 
    It follows that
    \begin{equation}
    \label{EquationProofMaximalExistenceTimeIndpendentOfmPositiveDistance}
     \sup_{t \in (t_0, T_+)} \Hhndom{k}{u(t)} < \infty \quad \text{and} \quad  \liminf_{t \nearrow T_+} \dist(\{u(t,x) \colon x \in G\}, \partial \mathcal{U}) > 0.
    \end{equation}
    Sobolev's embedding thus implies
    \begin{align*}
     \omega_0 := \sup_{t \in (t_0, T_+)} \|u(t)\|_{W^{1,\infty}(G)} < \infty.
    \end{align*} 
    Pick a radius $\rho > 0$ such that
    \begin{align*}
     \sum_{j = 0}^{m-1} \Hhndom{m-j-1}{\partial_t^j f(t_0)} \! + \! \|g\|_{E_m((t_0, T_+) \times \partial G)} \! + \! \Hhndom{m}{u_0} \! + \! \|f\|_{H^m((t_0,T_+) \times G)} < \rho.
    \end{align*}
    Due to~\eqref{EquationProofMaximalExistenceTimeIndpendentOfmPositiveDistance} and the boundedness of $u$ 
    there 
    is a compact subset $\mathcal{U}_1$ of $\mathcal{U}$ such that 
    $\image u(t) \subseteq \mathcal{U}_1$ for all $t \in [t_0, T_+]$.
    Proposition~\ref{PropositionEstimateForNonlinearSolutionOnDomain} then yields the bound
    \begin{align*}
     \sup_{t \in (t_0, T_+)}\Hhndom{m}{u(t)}^2 \leq C_{\ref{PropositionEstimateForNonlinearSolutionOnDomain}}(\chi, \sigma, m, \rho, \omega_0, \mathcal{U}_1,T_+ - t_0) \cdot C \rho^2.
    \end{align*}
    But by Lemma~\ref{LemmaBlowUpCriterion} and~\eqref{EquationProofMaximalExistenceTimeIndpendentOfmPositiveDistance} 
    we have $\lim_{t \nearrow T_+} \Hhn{m}{u(t)} = \infty$ and 
    thus a contradiction. We conclude that $T_+(k, t_0, f,g, u_0) = T_+$. 

  Assertion~\ref{ItemUniqueSolution} is just Proposition~\ref{PropositionMaximalExistence}.
  To show~\ref{ItemBlowUp}, assume that $T_+ < \infty$ and that properties~\ref{ItemLeavingEveryCompactSet} 
  and~\ref{ItemWaveBreaking} do not hold.
    We then have
    \begin{align*}
     \omega_0 := \sup_{t \in (t_0, T_+)} \|u(t)\|_{W^{1,\infty}(G)} < \infty
    \end{align*}
    and there is a compact subset $\mathcal{U}_1$ of $\mathcal{U}$ such that $\image u(t) \subseteq \mathcal{U}_1$ 
    for all $t \in [t_0, T_+]$.
    We apply Proposition~\ref{PropositionEstimateForNonlinearSolutionOnDomain} with $T^* = T_+$ again to deduce
    \begin{align*}
     \Hhn{m}{u(t)}^2 \leq C_{\ref{PropositionEstimateForNonlinearSolutionOnDomain}}(\chi,\sigma,m,r,\omega_0,\mathcal{U}_1,  T_+-t_0) \cdot C r^2
    \end{align*}
    for all $t \in (t_0, T_+)$ and thus $\sup_{t \in (t_0, T_+)} \Hhn{m}{u(t)} < \infty$. Lemma~\ref{LemmaBlowUpCriterion} however shows that $\lim_{t \nearrow T_+} \Hhndom{3}{u(t)} = \infty$.
    We thus obtain a contradiction.
    
   \ref{ItemContinuousDependence} 
   Let $T' \in (t_0, T_+)$.
   Without loss of generality we assume that $\chi$ and $\sigma$ satisfy~\eqref{EquationPropertyForL2}, 
   cf.~Remark~\ref{RemarkChiuInFm}.
	The difficulty in assertion~\ref{ItemContinuousDependence} is to make sure that the solutions to the data in the neighborhood we have to construct 
	exist at least till $T'$. To that purpose we use an iterative scheme that allows us to apply Theorem~\ref{TheoremLocalExistenceNonlinear} with 
	the same minimal time step size in each iteration. 
   
   Recall that by Sobolev's embedding there is a constant depending only on the length of the interval $[t_0, T_+)$ 
    such that
    \begin{align}
    \label{EquationSobolevInequalityOnMaximalInterval}
     \|\tilde{f}\|_{G_{m-1}((t_0, T_+) \times G)} \leq C_S \|\tilde{f}\|_{H^m((t_0, T_+) \times G)}
    \end{align}
    for all $\tilde{f} \in H^m((t_0, T_+) \times G)$. Fix a time $T^* \in (T', T_+)$.
   We pick two radii $0 < r_0 < r < \infty$ such that
   \begin{align}
   \label{EquationChoiceForrAndr0InContinuourDependance}
    &\Hhndom{m}{u_0} + \|f\|_{G_{m-1}((t_0,T_+) \times G)} + \|f\|_{H^m((t_0,T_+) \times G)} < r_0, \nonumber\\
    &C_S m r_0 < r, \quad \text{and} \quad \|u\|_{G_m((t_0, T^*) \times G)} < r.
   \end{align}
   Moreover, there is a compact subset $\mathcal{U}_1$ of $\mathcal{U}$ such that $\image u(t) \subseteq \mathcal{U}_1$ 
   for all $t \in [t_0, T^*]$.
   Lemma~\ref{LemmaHigherOrderChainRule} thus provides a number $\tilde{r} = \tilde{r}(\chi,\sigma,m,r, \mathcal{U}_1)$ 
   with 
   \begin{align}
   \label{EquationPreparationToApplyTheoremAPrioriEstimates}
    &\|\chi(u)\|_{F_m((t_0, T^*) \times G)} + \|\sigma(u)\|_{F_m((t_0, T^*) \times G)} \leq \tilde{r}, \nonumber\\
    &\max\{\Fvarnormdom{m-1}{\chi(u)(t_0)}, \max_{1 \leq j \leq m-1} \Hhndom{m-j-1}{\partial_t^j \chi(u)(t_0)}\} \leq \tilde{r}, \nonumber\\
    &\max\{\Fvarnormdom{m-1}{\sigma(u)(t_0)}, \max_{1 \leq j \leq m-1} \Hhndom{m-j-1}{\partial_t^j \sigma(u)(t_0)}\} \leq \tilde{r}.
   \end{align}

    I) Let $t' \in (t_0, T^*)$ and $(\tilde{f}, \tilde{g}, \tilde{u}_0) \in M_{\chi,\sigma,m}(t_0,T_+)$.
    Assume that the solution $\tilde{u}$ of~\eqref{EquationNonlinearIBVP} with data $\tilde{f}$, 
    $\tilde{g}$,  $\tilde{u}_0$ exists on $[t_0, t']$ and thus belongs to $G_m((t_0,t') \times G)$. Pick 
    a radius $R'$ and a compact subset $\tilde{\mathcal{U}}_1$ of $\mathcal{U}$ such that 
    $\|\tilde{u}\|_{G_m((t_0,t') \times G)} \leq R'$ and $\image u(t), \image \tilde{u}(t) \subseteq \tilde{\mathcal{U}}_1$ 
    for all $t \in [t_0, t']$. Set $\tilde{T} = T_+ - t_0$. 
    We first show the inequality
    \begin{align}
     \label{EquationDifferenceOfNonlinearSolutionsInGmminus1}
     &\|\tilde{u} - u\|_{G_{m-1}((t_0,t') \times G)}^2 \leq C \| \tilde{f} - f \|_{H^{m-1}((t_0,t') \times G)}^2 + C \| \tilde{g} - g\|_{E_{m-1}((t_0, t') \times \partial G)}\nonumber\\
     &\quad + C \Big( \sum_{j = 0}^{m-1} \Hhndom{m-j-1}{\partial_t^j \tilde{f}(t_0) - \partial_t^j f(t_0)}^2 + \Hhndom{m}{\tilde{u}_0 - u_0}^2 \Big),
    \end{align}
    for a constant $C = C(\chi,\sigma,m,r,R',\tilde{\mathcal{U}}_1, \tilde{T})$.
    To this aim, we apply the linear differential operator $L = L(\chi(u), A_1^{\operatorname{co}}, A_2^{\operatorname{co}}, A_3^{\operatorname{co}}, \sigma(u))$ to $\tilde{u} - u$.
    We obtain
    \begin{align*}
     L(\tilde{u} - u) = \tilde{f} + (\chi(u) - \chi(\tilde{u})) \partial_t \tilde{u} 
	+ (\sigma(u) - \sigma(\tilde{u})) \tilde{u} - f =: F.
    \end{align*}
    Lemma~\ref{LemmaHigherOrderChainRule} and \cite[Lemma~2.1]{SpitzMaxwellLinear} show that 
    $F$ is an element of $H^{m-1}((t_0,t') \times G)$. Set
    \begin{align*}
      \gamma_0 = \gamma_0(\chi, \sigma,  m, r, \mathcal{U}_1, \tilde{T}) = \gamma_{\ref{TheoremExistenceAndUniquenessOnDomain};0}(\eta(\chi),  \tilde{r}, \tilde{T}) \geq 1,
    \end{align*}
    where $\chi \geq \eta(\chi) > 0$ and $\gamma_{\ref{TheoremExistenceAndUniquenessOnDomain};0}$ is the 
    corresponding constant from Thereom~\ref{TheoremExistenceAndUniquenessOnDomain}. This theorem then yields
    \begin{align}
    \label{EquationAPrioriEstimatesAppliedToNonlinearDifference}
     &\|\tilde{u} - u\|_{G_{m-1,\gamma}((t_0,t')\times G)}^2 \nonumber\\
     &\leq (C_{\ref{TheoremExistenceAndUniquenessOnDomain};m,0} + \tilde{T} C_{\ref{TheoremExistenceAndUniquenessOnDomain};m}) 
	e^{m C_{\ref{TheoremExistenceAndUniquenessOnDomain};1} \tilde{T}} \Big(\sum_{j = 0}^{m-2} \Hhndom{m-2-j}{\partial_t^j F(t_0)}^2 
	 + \Hhndom{m-1}{\tilde{u}_0 - u_0}^2 \nonumber \\
	 &\hspace{2em} + \|\tilde{g} - g\|_{E_{m-1,\gamma}((t_0, t') \times \partial G)}^2 \Big) + \frac{C_{\ref{TheoremExistenceAndUniquenessOnDomain};m}}{\gamma} \|F\|_{H^{m-1}_\gamma((t_0, t') \times G)}^2
    \end{align}
    for all $\gamma \geq \gamma_0$, where 
    $C_{\ref{TheoremExistenceAndUniquenessOnDomain};m,0} = C_{\ref{TheoremExistenceAndUniquenessOnDomain};m,0}(\eta(\chi),   \tilde{r})$, $C_{\ref{TheoremExistenceAndUniquenessOnDomain};m} = C_{\ref{TheoremExistenceAndUniquenessOnDomain};m}(\eta(\chi), \tilde{r}, \tilde{T})$, and $C_{\ref{TheoremExistenceAndUniquenessOnDomain};1} = C_{\ref{TheoremExistenceAndUniquenessOnDomain};1}(\eta(\chi), \tilde{r}, \tilde{T})$
    are the corresponding constants from Theorem~\ref{TheoremExistenceAndUniquenessOnDomain}. We next apply Lemma~2.1~(2) 
    from~\cite{SpitzMaxwellLinear}
    and then Corollary~\ref{CorollaryEstimateForDifferenceHigherOrder}
    to obtain
    \begin{align}
    \label{EquationEstimateForBigFInHmminus1}
     &\|F\|_{H^{m-1}_\gamma((t_0, t') \times G)}^2 \\
     &\leq  C \|\tilde{f} - f\|_{H_{\gamma}^{m-1}((t_0,t') \times G)}^2 \! + \! C \tilde{T} \| \chi(\tilde{u}) - \chi(u)\|_{G_{m-1,\gamma}((t_0,t') \times G)}^2 \|\partial_t \tilde{u}\|_{G_{m-1}((t_0,t') \times G)}^2  \nonumber\\
     &\hspace{3em} + C \tilde{T} \| \sigma(\tilde{u}) - \sigma(u)\|_{G_{m-1,\gamma}((t_0,t') \times G)}^2 \| \tilde{u}\|_{G_{m-1}((t_0,t') \times G)}^2 \nonumber\\
     &\leq C  \|\tilde{f} - f\|_{H^{m-1}_{\gamma}((t_0,t') \times G)}^2 +C(\chi,\sigma,m,r,R', \tilde{\mathcal{U}}_1, \tilde{T}) \|\tilde{u} - u\|_{G_{m-1,\gamma}((t_0,t')\times G)}^2. \nonumber
    \end{align}
    Let $j \in \{0, \ldots, m-2\}$. To treat $\partial_t^j F(t_0)$, we employ Lemma~\ref{LemmaHigherOrderChainRule} and the definition of the $M_k^l$ in~\eqref{EquationDefinitionMkp}. The same arguments as in the proof of
        Lemma~\ref{LemmaNonlinearHigherOrderInitialValues} then show that
    \begin{align}
    \label{EquationEstimateForBigFInt0}
     &\Hhndom{m-2-j}{\partial_t^j F(t_0)} \leq \Hhndom{m-1-j}{\partial_t^j F(t_0)} \\
     &\leq C(\chi,\sigma,m,r, R', \tilde{\mathcal{U}}_1) \Big( \sum_{l = 0}^{m-1} \Hhndom{m-l-1}{\partial_t^l f(t_0) - \partial_t^l \tilde{f}(t_0)} + \Hhndom{m}{u_0 - \tilde{u}_0} \Big). \nonumber
    \end{align}
    We obtain a constant
    $C_{\ref{EquationIntermediateStepForNonlinearDifference}} = C_{\ref{EquationIntermediateStepForNonlinearDifference}}(\chi,\sigma,m,r,R',\tilde{\mathcal{U}}_1,\tilde{T})$
    and the bound
    \begin{align}
     \label{EquationIntermediateStepForNonlinearDifference}
     &\|\tilde{u} - u\|_{G_{m-1,\gamma}((t_0,t')\times G)}^2 \\
     &\leq C_{\ref{EquationIntermediateStepForNonlinearDifference}} \Big( \frac{1}{\gamma} \|\tilde{u} - u\|_{G_{m-1,\gamma}((t_0,t')\times \G)}^2 +\sum_{l = 0}^{m-1} \Hhn{m-l-1}{\partial_t^l \tilde{f}(t_0) - \partial_t^l f(t_0)}^2 \nonumber\\
     &\hspace{2em} + \|\tilde{g} - g\|_{E_{m-1,\gamma}((t_0,t') \times \partial G)}^2 + \Hhn{m}{\tilde{u}_0 - u_0}^2 +  \|\tilde{f} - f\|_{H^{m-1}_{\gamma}((t_0,t') \times G)}^2 \Big) \nonumber
    \end{align}
    for all $\gamma \geq \gamma_0$ by inserting~\eqref{EquationEstimateForBigFInHmminus1} and~\eqref{EquationEstimateForBigFInt0} 
    into~\eqref{EquationAPrioriEstimatesAppliedToNonlinearDifference}. We next fix a number $\gamma = \gamma(\chi,\sigma,m,r,R', \tilde{\mathcal{U}}_1,\tilde{T})$ with 
    $\gamma \geq \gamma_0$ and $C_{\ref{EquationIntermediateStepForNonlinearDifference}} \frac{1}{\gamma} \leq \frac{1}{2}$ to infer that~\eqref{EquationDifferenceOfNonlinearSolutionsInGmminus1} is true.
    
    II) Recall that $\mathcal{U}_1$ is a compact subset of $\mathcal{U}$ such that $\image u(t) \subseteq \mathcal{U}_1$ 
    for all $t \in [t_0, T^*]$. Pick a number $\kappa$ such that 
    \begin{equation}
    \label{EquationFixingKappaDistanceInContinuousDependance}
      2 \kappa < \dist(\mathcal{U}_1, \partial \mathcal{U}).
    \end{equation}
    Take the time step $\tau = \tau(\chi, \sigma, m, \tilde{T}, 4r,\kappa)$ from Theorem~\ref{TheoremLocalExistenceNonlinear}. 
    Choose an index $N \in \N$ 
   with 
    \begin{align*}
      t_0 + (N-1) \tau < T' \leq t_0 + N \tau.
    \end{align*} 
    We set $t_k = t_0 + k \tau$ for $k \in \{1, \ldots, N-1\}$. If
    $t_0 + N \tau < T^*$, we put $t_N = t_0 + N \tau$; else we take any $t_N$ from $(T', T^*)$.

    Let $0 < \delta_0 < r_0$ be so small that $C_{\operatorname{Sob}} \delta_0 < \kappa$, where 
    $C_{\operatorname{Sob}}$ is the norm of the
    embedding from $\Hhdom{2} \hookrightarrow L^\infty(G)$. As in~\eqref{EquationDefinitionUkappa} 
    and~\eqref{EquationDefinitionTildeUkappa} we define the compact sets
	\begin{equation*}
		\mathcal{U}_\kappa = \{y \in \mathcal{U} \colon \dist(y, \partial \mathcal{U}) \geq \kappa \} \cap B_{2 C_{\operatorname{Sob}} 4r}(0) \quad \text{and} \quad \tilde{\mathcal{U}}_\kappa = \mathcal{U}_\kappa + \overline{B}(0,\kappa/2).
	\end{equation*}    
     Take $(\tilde{f}, \tilde{g}, \tilde{u}_0) \in B_{M_{\chi,\sigma,m}(t_0,T_+)}((f,g,u_0),\delta_0)$.
    Using the choice of $r$ and $r_0$ in~\eqref{EquationChoiceForrAndr0InContinuourDependance}, we deduce that 
    \begin{align}
     &\Hhndom{m}{\tilde{u}_0} \leq \Hhndom{m}{u_0} + \Hhndom{m}{\tilde{u}_0 - u_0} \leq r_0 + \delta_0 < 2 r_0 < 2r, \nonumber\\
     &\|\tilde{g}\|_{E_m((t_0, T') \times \partial G)} \leq r_0 + \delta_0 < 2r, \quad
     \|\tilde{f}\|_{H^m((t_0, T') \times G)} \leq r_0 + \delta_0  < 2r, \label{EquationBoundForTildefInHm}\\
     &\sum_{j = 0}^{m-1} \Hhndom{m-1-j}{\partial_t^j \tilde{f}(t_0)} \leq m \|\tilde{f}\|_{G_{m-1}((t_0, T')\times G)} \leq C_S m \|\tilde{f}\|_{H^m((t_0, T') \times G)} \nonumber \\
     &\qquad < 2 C_S m r_0 < 2r. \label{EquationBoundForTildefInGmminus1}
    \end{align}
	Moreover, $\|\tilde{u}_0 - u_0\|_{L^\infty(G)} \leq C_{\operatorname{Sob}} \delta_0 < \kappa$ so that $\image \tilde{u}_0$ is contained in $\mathcal{U}_\kappa$.    
    So Theorem~\ref{TheoremLocalExistenceNonlinear} shows that the solution $\tilde{u}$ of~\eqref{EquationNonlinearIBVP}
    with data $\tilde{f}$, $\tilde{g}$, and $\tilde{u}_0$ at $t_0$ exists on $[t_0, t_1]$ and 
    belongs to $G_m((t_0,t_1) \times G)$. Moreover, the proof of this theorem yields a radius 
    $R = R_{\ref{TheoremLocalExistenceNonlinear}}(\chi,\sigma,m,4r,\kappa) > 4r$, see~\eqref{EquationRadiusForFixedPointSpace},
    such that $\|\tilde{u}\|_{G_m((t_0,t_1)\times G)} \leq R$. This proof also shows that 
    $\image \tilde{u}(t) \subseteq \tilde{\mathcal{U}}_\kappa$ for all $t \in [t_0,t_1]$, cf.~\eqref{EquationDefinitionTildeUkappa}.
    We conclude that~$\Psi$ maps $B_{M_{\chi,\sigma,m}(t_0,T_+)}((f,g,u_0),\delta_0)$ into 
    $B_{G_m((t_0,t_1)\times G)}(0,R)$. We further deduce from~\eqref{EquationDifferenceOfNonlinearSolutionsInGmminus1}
    that there is a constant
    \begin{align*} 
     C_{\ref{EquationLipschitzEstimateForFlowOnFirstTimeInterval}} 
     = C_{\ref{EquationLipschitzEstimateForFlowOnFirstTimeInterval}}(\chi,\sigma,m,r,\tilde{T},\kappa) 
     =  C_{\ref{EquationDifferenceOfNonlinearSolutionsInGmminus1}}(\chi,\sigma,m,r, R(\chi,\sigma,m,r,\kappa),\tilde{U}_\kappa, \tilde{T})
    \end{align*}
    such that
    \begin{align}
    \label{EquationLipschitzEstimateForFlowOnFirstTimeInterval}
     &\|\Psi(\tilde{f}, \tilde{g}, \tilde{u}_0) - \Psi(f,g,u_0)\|_{G_{m-1}((t_0,t_1) \times G)}^2 \nonumber \\
     &\leq C_{\ref{EquationLipschitzEstimateForFlowOnFirstTimeInterval}}  \| \tilde{f} - f \|_{H^{m-1}((t_0,t_1) \times G)}^2 + C_{\ref{EquationLipschitzEstimateForFlowOnFirstTimeInterval}} \|\tilde{g} - g\|_{E_{m-1}((t_0,t') \times \partial G)}^2 \nonumber\\
     &\hspace{2em} + C_{\ref{EquationLipschitzEstimateForFlowOnFirstTimeInterval}}  \Big(\sum_{j = 0}^{m-1} \Hhn{m-j-1}{\partial_t^j \tilde{f}(t_0) - \partial_t^j f(t_0)}^2 + \Hhn{m}{\tilde{u}_0 - u_0}^2 \Big) 
    \end{align}
    for all $(\tilde{f}, \tilde{g}, \tilde{u}_0) \in B_{M_{\chi,\sigma,m}(t_0,T_+)}((f,g,u_0),\delta_0)$.
    
    Next take a sequence $(f_n, g_n, u_{0,n})_n$ in 
    $B_{M_{\chi,\sigma,m}(t_0,T_+)}((f,g,u_0),\delta_0)$ which converges to $(f,g, u_0)$ in this space.
    Since
    \begin{align}
    \label{EquationConvergenceToZeroDueToSobolev}
     \sum_{j = 0}^{m-1} \Hhndom{m-j-1}{\partial_t^j f_n(t_0) - \partial_t^j f(t_0)}^2 
     &\leq m C_S \|f_n - f\|_{H^m((t_0, T_+) \times G)} \longrightarrow 0
    \end{align}
    as $n \rightarrow \infty$, estimate~\eqref{EquationLipschitzEstimateForFlowOnFirstTimeInterval} 
    yields the limit
    \begin{align*}
     \|\Psi(f_n,g_n, u_{0,n}) - \Psi(f,g,u_0)\|_{G_{m-1}((t_0,t_1)\times G)} \longrightarrow 0
    \end{align*}
    as $n \rightarrow \infty$. 
    Lemma~\ref{LemmaConvergenceOfNonlinearSolutionInGmProvidedConvergenceInGmminus1}
    thus shows that $(\Psi(f_n, g_n, u_{0,n}))_n$ converges to $\Psi(f,g,u_0)$ in $G_{m}((t_0,t_1) \times G)$. 
    We conclude that the map 
    \begin{align*}
     \Psi \colon B_{M_{\chi,\sigma,m}(t_0,T_+)}((f,g,u_0),\delta_0) \rightarrow G_m((t_0,t_1)\times G)
    \end{align*}
    is continuous at $(f,g, u_0)$. Using also~\eqref{EquationChoiceForrAndr0InContinuourDependance} 
    and~\eqref{EquationFixingKappaDistanceInContinuousDependance}, we find 
    a number $\delta_1 \in (0, \delta_0]$ such that 
    for all data $(\tilde{f}, \tilde{g}, \tilde{u}_0) \in B_{M_{\chi,\sigma,m}(t_0,T_+)}((f,g,u_0),\delta_1)$
    the function
    $\Psi(\tilde{f}, \tilde{g}, \tilde{u}_0)$ exists on $[t_0,t_1]$ and 
    satisfies~\eqref{EquationLipschitzEstimateForFlowOnFirstTimeInterval} and
    \begin{align*}
     &\| \Psi(\tilde{f}, \tilde{g}, \tilde{u}_0)\|_{G_m((t_0,t_1) \times G)} \\
     &\leq \| \Psi(\tilde{f},\tilde{g}, \tilde{u}_0) - \Psi(f,g,u_0)\|_{G_m((t_0,t_1) \times G)} + \|\Psi(f,g,u_0)\|_{G_m((t_0,t_1) \times G)} < 2r, \\
     &\dist(\image \Psi(\tilde{f}, \tilde{g}, \tilde{u}_0)(t),\partial \mathcal{U}) >  \kappa,
    \end{align*}
    for all $t \in [t_0,t_1]$.
    
    Now assume that there is an index $j \in \{1, \ldots, N-1\}$ and a number $\delta_j > 0$ 
    such that $\Psi(\tilde{f}, \tilde{g}, \tilde{u}_0)$ 
    exists on $[t_0,t_j]$ and satisfies
    \begin{align*}
     &\|\Psi(\tilde{f}, \tilde{g}, \tilde{u}_0)\|_{G_{m}((t_0,t_j) \times G)} < 2r \quad \text{and} \quad \dist(\image \Psi(\tilde{f}, \tilde{g}, \tilde{u}_0)(t), \partial \mathcal{U}) >  \kappa
    \end{align*}
    for all $t \in [t_0, t_j]$ and $(\tilde{f},\tilde{g}, \tilde{u}_0) \in B_{M_{\chi,\sigma,m}(t_0,T_+)}((f,g,u_0),\delta_j)$.
    
    Fix such a tuple $(\tilde{f}, \tilde{g}, \tilde{u}_0)$. Then the tuple $(\chi,\sigma,t_j, B,\tilde{f}, \tilde{g}, \Psi(\tilde{f},\tilde{g}, \tilde{u}_0)(t_j))$
    fulfills the nonlinear compatibility conditions~\eqref{EquationNonlinearCompatibilityConditions} 
    of order $m$ by~\eqref{EquationTimeDerivativesOfSolutionEqualSChiSigmaG} and 
    \begin{align*}
     &\Hhndom{m}{\Psi(\tilde{f}, \tilde{g}, \tilde{u}_0)(t_j)} \leq \|\Psi(\tilde{f}, \tilde{g}, \tilde{u}_0)\|_{G_m((t_0,t_j) \times G)} < 2r, \\
     &\dist(\image \Psi(\tilde{f}, \tilde{g}, \tilde{u}_0)(t_j), \partial \mathcal{U}) > \kappa.
    \end{align*}
    In view of~\eqref{EquationBoundForTildefInHm} and~\eqref{EquationBoundForTildefInGmminus1}, 
    Theorem~\ref{TheoremLocalExistenceNonlinear} shows that the problem~\eqref{EquationNonlinearIBVP} 
    with inhomogeneity $\tilde{f}$, boundary value $\tilde{g}$, and initial value $\Psi(\tilde{f}, \tilde{g}, \tilde{u}_0)(t_j)$ at initial time $t_j$ 
    has a unique solution $\tilde{u}^j$ on $[t_j, t_{j+1}]$, which is bounded by $R$ in 
    $G_m((t_j,t_{j+1}) \times G)$ and whose image is contained in $\tilde{U}_\kappa$. Concatenating $\Psi(\tilde{f}, \tilde{g},\tilde{u}_0)$ and $\tilde{u}^j$, 
    we obtain a solution of~\eqref{EquationNonlinearIBVP} with data $\tilde{f}$, $\tilde{g}$,
    and
     $\tilde{u}_0$ at initial time $t_0$, cf. Remark~\ref{RemarkNegativeTimes}. 
    This means that $\Psi(\tilde{f}, \tilde{g}, \tilde{u}_0)$ exists on $[t_0, t_{j+1}]$. 
    Uniqueness of solutions of~\eqref{EquationNonlinearIBVP}, i.e. Lemma~\ref{LemmaUniquenessOfNonlinearSolution}, further yields $\Psi(\tilde{f}, \tilde{g}, \tilde{u}_0)_{|[t_j, t_{j+1}]} = \tilde{u}^j$ 
    so that
    \begin{align*}
     \| \Psi(\tilde{f}, \tilde{g}, \tilde{u}_0) \|_{G_{m}((t_0, t_{j+1}) \times G)} 
     &\leq \max\{\| \Psi(\tilde{f}, \tilde{g}, \tilde{u}_0) \|_{G_{m}((t_0, t_{j})\times G)}, \|\tilde{u}^j\|_{G_m((t_j,t_{j+1}) \times G)}\} \\
     &\leq \max\{2r, R\} \leq R.
    \end{align*}
    As for the interval $[t_0, t_1]$, we obtain a number
     $\delta_{j+1} \in (0, \delta_j]$ such that
    \begin{align*}
     &\|\Psi(\tilde{f}, \tilde{g}, \tilde{u}_{0}) \|_{G_{m}((t_0, t_{j+1}) \times G)}  < 2r \quad \text{and} \quad 
     \dist(\image \Psi(\tilde{f}, \tilde{g}, \tilde{u}_0)(t), \partial \mathcal{U}) > \kappa
    \end{align*}
    for all $t \in [t_0, t_{j+1}]$ and $(\tilde{f}, \tilde{g}, \tilde{u}_0) \in B_{M_{\chi,\sigma,m}(t_0,T_+)}((f,g,u_0),\delta_{j+1})$.
    
    By induction, the above property holds for $j + 1 = N$, so that
    \begin{align*}
     T_+(m, t_0, \tilde{f}, \tilde{g}, \tilde{u}_0) > t_N \geq T'
    \end{align*}
    for all $(\tilde{f}, \tilde{g}, \tilde{u}_0) \in B_{M_{\chi,\sigma,m}(t_0,T_+)}((f,g,u_0),\delta_{N})$.
    
    Next fix 
    two tuples $(\tilde{f}_1, \tilde{g}_1, \tilde{u}_{0,1})$ and $(\tilde{f}_2, \tilde{g}_2, \tilde{u}_{0,2})$ in $B_{M_{\chi,\sigma,m}(t_0,T_+)}((f,g,u_0),\delta_{N})$.
    Replacing $u$ by $\Psi(\tilde{f}_2, \tilde{g}_2, \tilde{u}_{0,2})$ in step I), we deduce from~\eqref{EquationDifferenceOfNonlinearSolutionsInGmminus1} 
    that
    \begin{align}
    \label{EquationEstimateForDifferenceContinuousDependanceAllDataFromBall}
     &\|\Psi(\tilde{f}_1, \tilde{g}_1, \tilde{u}_{0,1}) - \Psi(\tilde{f}_2, \tilde{g}_2, \tilde{u}_{0,2})\|_{G_{m-1}((t_0,T') \times G)}^2 \nonumber \\
     &\leq C \| \tilde{f}_1 - \tilde{f}_2 \|_{H^{m-1}((t_0,T') \times G)}^2 + C\|\tilde{g}_1 - \tilde{g}_2 \|_{E_{m-1}((t_0, T') \times G)}^2 \nonumber\\
     &\quad + C \Big( \sum_{j = 0}^{m-1} \Hhndom{m-j-1}{\partial_t^j \tilde{f}_1(t_0) - \partial_t^j \tilde{f}_2(t_0)}^2 + \Hhndom{m}{\tilde{u}_{0,1} - \tilde{u}_{0,2}}^2 \Big),
    \end{align}
    where $C = C(\chi,\sigma,m,r, \tilde{T},\kappa) = C_{\ref{EquationDifferenceOfNonlinearSolutionsInGmminus1}}(\chi,\sigma,m,2r,2r,\mathcal{U}_\kappa,\tilde{T})$
    and the constant $C_{\ref{EquationDifferenceOfNonlinearSolutionsInGmminus1}}$ from~\eqref{EquationDifferenceOfNonlinearSolutionsInGmminus1}.
    This estimate implies~\eqref{EquationLipschitzEstimateInContinuousDependenceAssertion}. Finally, we 
    take a sequence $(\tilde{f}_n, \tilde{g}_n, \tilde{u}_{0,n})_n$ in $B_{M_{\chi,\sigma,m}(t_0,T_+)}((f,g,u_0),\delta_{N})$
    which converges to $(\tilde{f}_1, \tilde{g}_1, \tilde{u}_{0,1})$ in $M_{\chi,\sigma,m}(t_0,T_+)$. Employing~\eqref{EquationSobolevInequalityOnMaximalInterval}, 
    we deduce from~\eqref{EquationEstimateForDifferenceContinuousDependanceAllDataFromBall}
    that $\Psi(\tilde{f}_n, \tilde{g}_n, \tilde{u}_{0,n})$ tends to $\Psi(\tilde{f}_1, \tilde{g}_1, \tilde{u}_{0,1})$ 
    in the space $G_{m-1}((t_0,T') \times G)$ as $n \rightarrow \infty$. Lemma~\ref{LemmaConvergenceOfNonlinearSolutionInGmProvidedConvergenceInGmminus1}
    therefore implies that 
    \begin{align*}
     \| \Psi(\tilde{f}_n, \tilde{g}_n, \tilde{u}_{0,n}) - \Psi(\tilde{f}_1, \tilde{g}_1, \tilde{u}_{0,1})\|_{G_{m}((t_0,T') \times G)} \longrightarrow 0
    \end{align*}
    as $n \rightarrow \infty$. Consequently, the flow map 
    \begin{align*}
     \Psi \colon B_{M_{\chi,\sigma,m}(t_0,T_+)}((f,g,u_0),\delta_{N}) \rightarrow G_{m}((t_0,T') \times G)
    \end{align*}
    is continuous at $(\tilde{f}_1, \tilde{g}_1, \tilde{u}_{0,1})$ and thus on $B_{M_{\chi,\sigma,m}(t_0,T_+)}((f,g,u_0),\delta_{N})$.
  \end{proof}
 
 \section{Finite propagation speed}
 \label{SectionFinitePropagationSpeed}
 We finally prove that solutions of~\eqref{EquationNonlinearIBVP} have finite propagation speed, 
 i.e., that initial disturbances travel with finite speed. Several techniques to establish this property 
 have been developed in the literature, see e.g.~\cite{BahouriCheminDanchin, BenzoniGavage, Evans}. 
 While these references work on the full space, finite speed of propagation is proven for an 
 initial boundary value problem in~\cite{ChazarainPiriou}, making however several structural
 assumptions on 
 the problem which are not fulfilled by Maxwell's equations. We thus follow a different approach 
 and show that the technique of weighted energy estimates from~\cite{BahouriCheminDanchin} is 
 well adaptable to our setting.
 
 We first prove the finite propagation speed property for the 
 corresponding linear problem~\eqref{IBVP}. This result can be transferred to domains via 
 localization 
 as in the previous sections, see~\cite[Chapter~6]{SpitzDissertation} for details. We concentrate 
 on the half-space case here since in this case one sees much better how the maximal
 propagation speed depends on the coefficients. We however note that the 
 coefficients on the half-space arising after localization depend on the charts of $\partial G$ so 
 that the maximal propagation speed of the solution on domains depends on the shape of $\partial G$.
 We refer to Theorem~6.1 of~\cite{SpitzDissertation} for a result on $G$ itself.
 
 In this section we set $J = (0,T)$ for a time $T > 0$.
 \begin{theorem}
 \label{TheoremFinitePropagationSpeed}
 Let $\eta > 0$, $A_0 \in \Fupdwl{3}{c}{\eta}$, $A_1, A_2 \in \Fcoeff{3}{cp}$, $A_3 = A_3^{\operatorname{co}}$, 
 $D \in \Fuwl{3}{c}$, and $B = B^{\operatorname{co}}$. We set 
 \begin{align*}
  C_0 = \frac{1}{\eta} \, \sum_{j = 1}^3 \|A_j\|_{L^\infty(\Omega)}.
 \end{align*}
 Let $R > 0$ and $x_0 \in \overline{\R^3_+}$. We define the backward cone $\mathcal{C}$ by
 \begin{equation*}
  \mathcal{C} = \{(t,x) \in \R \times \R^3 \colon |x - x_0| < R - C_0 \,t \}.
 \end{equation*}
 Let $f \in \Ltwoa$, $g \in L^2(J, H^{1/2}(\partial \R^3_+))$, and $u_0 \in \Ltwoh$ satisfy
 \begin{equation*}
  \begin{aligned}
    f&= 0 \quad &&\text{on } \mathcal{C} \cap \Omega, \\
    g&= 0 	&&\text{on } \mathcal{C} \cap (J \times \partial \R^3_+), \\
    u_0&= 0 	&&\text{on } \mathcal{C}_{t=0} \cap \R^3_+,
  \end{aligned}
 \end{equation*}
 where 
 $\mathcal{C}_{t=0} = \{x \in \R^3 \colon (0,x) \in \mathcal{C}\}$. Then the unique solution 
 $u \in C(\clJ, \Ltwoh)$ of the linear initial boundary value problem~\eqref{IBVP} 
 with inhomogeneity $f$, boundary value $g$, and initial value $u_0$ vanishes on the cone $\mathcal{C}$, 
 i.e., 
 \begin{equation*}
  u(t,x) = 0 \quad \text{for almost all } (t,x) \in \mathcal{C} \cap \Omega.
 \end{equation*}
\end{theorem}

\begin{proof}
 I) Let $\epsilon > 0$ and set $K = C_0^{-1}$. We fix a function 
    $\psi \in C^{\infty}(\R^3)$ with
    \begin{align}
     &-2 \epsilon + K(R - |x - x_0|) \leq \psi(x) \leq - \epsilon + K (R - |x - x_0|) \quad \text{for all } x \in \R^3, \label{EquationEstimatesForWeightFunctionPsi} \\
     &\| \nabla \psi \|_{L^\infty(\R^3)} \leq K,
    \end{align}
    see step~I) of the proof of Theorem~6.1 in~\cite{SpitzDissertation}.
    We first assume that $f$ belongs to $\Ha{1}$, $g$ to 
    $\E{1}$, 
    and $u_0$ to $\Hh{1}$ and that $(t_0,A_0, \ldots, A_3, D, B, f, g, u_0)$ fulfills 
    the linear compatibility conditions of first order.
    Fix $\epsilon > 0$. We introduce the cone
    \begin{equation*}
     \mathcal{C}_\epsilon = \{(t,x) \in \R \times \R^3 \colon |x - x_0| < R - C_0 t - C_0 \epsilon \}.
    \end{equation*}
    The cones $\mathcal{C}_{t=0, \epsilon}$ are defined analogously.
    We set
    \begin{equation*}
     \Phi(t,x) = - t + \psi(x)
    \end{equation*}
    for all $(t,x) \in \R \times \R^3$. Note that $\Phi$ belongs to $C^\infty(\overline{J \times \R^3_+})$. We next 
    want to derive a weighted energy inequality for $u$. To that purpose, we define
    \begin{align*}
     u_\tau = e^{\tau \Phi} u, \quad f_\tau = e^{\tau \Phi} f, \quad g_\tau = e^{\tau \Phi(\cdot, 0)} g, \quad u_{0,\tau} = e^{\tau \Phi(0, \cdot)} u_0
    \end{align*}
    for all $\tau > 0$. Observe that there is a constant $C = C(\tau, \epsilon)$ such that
    \begin{equation*}
     e^{\tau \Phi(t,x)} + \sum_{j = 0}^3 |\partial_j e^{\tau \Phi(t,x)}| \leq C e^{-\tau K |x - x_0|}
    \end{equation*}
    for all $(t,x) \in [0,\infty) \times \R^3$ and $\tau > 0$. We further observe that 
    $u$ belongs to $\G{1}$ by Theorem~\ref{TheoremExistenceAndUniquenessOnDomain}. 
    Consequently, we infer that $u_\tau$ is contained in $\G{1}$, 
    $f_\tau$ in $\Ha{1}$, $g_\tau$ in $\E{1}$, and $u_{0,\tau}$ in $\Hh{1}$ 
    for all $\tau > 0$.
    With this amount of regularity we can compute
    \begin{align}
     A_0 \partial_t u_\tau &+ \sum_{j = 1}^3 A_j \partial_j u_\tau + D u_\tau = f_\tau - \tau \Big(A_0 - \sum_{j=1}^3 \partial_j \psi A_j \Big) u_\tau
     \label{EquationEquationForUTau}
    \end{align}
    for all $\tau > 0$. These functions also satisfy $u_\tau(0) = u_{0,\tau}$ on $\R^3_+$ and $B u_\tau = g_\tau$
    on $J \times \partial \R^3_+$ for all $\tau > 0$.
    
    Moreover, $A_0 - \sum_{j = 1}^3 \partial_j \psi A_j$ is uniformly positive semidefinite 
    since
    \begin{align*}
     &\Big(\big(A_0 - \sum_{j=1}^3 \partial_j \psi A_j \big) \xi, \xi \Big)_{\R^6 \times \R^6} 
     \geq \eta |\xi|^2 - \sum_{j = 1}^3 \|\partial_j \psi\|_{L^\infty(\R^3)} \|A_j\|_{L^\infty(\Omega)} |\xi|^2  \\
     &\geq \eta |\xi|^2 - \eta K C_0 |\xi|^2 = 0
    \end{align*}
    on $\Omega$ for $\xi \in \R^6$, where we used the definition of $C_0 = K^{-1}$. Identity~\eqref{EquationEquationForUTau} 
    in combination with this estimate then yields
    \begin{align*}
     &\partial_t \langle A_0 u_\tau, u_\tau \rangle_{\Ltwoh \times \Ltwoh} \\
      &= \langle \partial_t A_0 u_\tau, u_\tau \rangle_{\Ltwoh \times \Ltwoh} + 2 \Big\langle f_\tau - \sum_{j=1}^3 A_j \partial_j u_\tau - D u_\tau, u_\tau \Big\rangle_{\Ltwoh \times \Ltwoh} \\
      &\quad - 2 \tau \Big\langle \Big(A_0 - \sum_{j = 1}^3 \partial_j \psi A_j \Big) u_\tau, u_\tau \Big \rangle_{\Ltwoh \times \Ltwoh} \\
     &\leq \langle \partial_t A_0 u_\tau, u_\tau \rangle_{\Ltwoh \times \Ltwoh} - 2 \sum_{j = 1}^3 \langle A_j \partial_j u_\tau, u_\tau \rangle_{\Ltwoh \times \Ltwoh} \\
     &\quad - 2 \langle D u_\tau, u_\tau \rangle_{\Ltwoh \times \Ltwoh} + 2 \langle f_\tau, u_\tau \rangle_{\Ltwoh \times \Ltwoh}
    \end{align*}
    for almost all $t \in J$ and for all $\tau > 0$. Hence,
    \begin{align}
     &\eta \Ltwohn{u_\tau(t)}^2 \leq  \langle  A_0 u_\tau, u_\tau \rangle_{\Ltwoh \times \Ltwoh} \nonumber\\
     &= \langle  A_0(0) u_\tau(0), u_\tau(0) \rangle_{\Ltwoh \times \Ltwoh} + \int_0^t \partial_t  \langle A_0 u_\tau, u_\tau \rangle_{\Ltwoh \times \Ltwoh}(s) ds \nonumber\\
     &\leq \|A_0\|_{L^\infty(\Omega)} \Ltwohn{u_{0,\tau}}^2 + (\|\partial_t A_0\|_{L^\infty(\Omega)} + 2 \|D\|_{L^\infty(\Omega)}) \int_{0}^t  \Ltwohn{u_\tau(s)}^2 ds \nonumber\\
     &\quad - 2 \sum_{j = 1}^3 \int_0^t \langle A_j(s) \partial_j u_\tau(s), u_\tau(s) \rangle_{\Ltwoh \times \Ltwoh} ds \nonumber\\
     &\quad + 2 \int_0^t \Ltwohn{f_\tau(s)} \Ltwohn{u_\tau(s)} ds \label{EquationEstimateOfUTauForGronwall1}
    \end{align}
    for all $t \in J$ and $\tau > 0$. Since $u_\tau \in \G{1}$, the symmetry of the matrices 
    $A_j$ further implies
    \begin{align}
     &\langle A_j \partial_j u_\tau, u_\tau \rangle_{\Ltwoh \times \Ltwoh} = -\frac{1}{2} \langle \partial_j A_j  u_\tau, u_\tau \rangle_{\Ltwoh \times \Ltwoh}  
     \label{EquationSymmetryOfAj}, \\
      &\langle A_3 \partial_3 u_\tau, u_\tau \rangle_{\Ltwoh \times \Ltwoh} = -\frac{1}{2} \langle \partial_3 A_3  u_\tau, u_\tau \rangle_{\Ltwoh \times \Ltwoh} \nonumber\\
     &\hspace{14em} - \frac{1}{2} \int_{\partial \R^3_+} \tr(A_3 u_\tau)(\sigma) \tr (u_\tau)(\sigma) d\sigma
     \label{EquationSymmetryOfA3}
    \end{align}
    on $J$ for $j \in \{1,2\}$ by integration by parts. We set 
    \begin{equation*}
     C_1 = \frac{1}{\eta} \Big(\sum_{j=0}^3 \|A_j\|_{W^{1,\infty}(\Omega)} + 2 \|D\|_{L^\infty(\Omega)} + \frac{1}{\eta}\Big).
    \end{equation*}
    Inserting~\eqref{EquationSymmetryOfAj} and~\eqref{EquationSymmetryOfA3} into~\eqref{EquationEstimateOfUTauForGronwall1}, 
    we derive
    \begin{align}
     &\eta \Ltwohn{u_\tau(t)}^2 \nonumber \\
     &\leq \eta \, C_1 \Ltwohn{u_{0,\tau}}^2 + \eta \Ltwoanwgamma{f_\tau}^2 + \eta \, C_1 \int_0^t \Ltwohn{u_\tau(s)}^2 ds \nonumber \\
     &\quad +  \langle \tr(A_3 u_\tau), \tr u_\tau  \rangle_{L^2(\Gamma_t) \times L^2(\Gamma_t)} \label{EquationEstimateOfUTauForGronwall2}
    \end{align}
    for all $t \in J$ and $\tau > 0$, where we denote $(0,t) \times \partial \R^3_+$ by $\Gamma_t$. In order to estimate the last 
    term in~\eqref{EquationEstimateOfUTauForGronwall2}, we recall that the boundary matrix 
    $A_3 = A_3^{\operatorname{co}}$ decomposes as
    \begin{equation*}
     A_3^{\operatorname{co}} =  \frac{1}{2} C^{\operatorname{co} T} B^{\operatorname{co}} + \frac{1}{2}  B^{\operatorname{co} T} C^{\operatorname{co}},
    \end{equation*}
    see~\eqref{EquationDecompositionOfA3co}.
    Employing $B u_\tau = g_\tau$, $B = B^{\operatorname{co}}$, and $u \in \G{1}$, we thus infer
    \begin{align}
    \label{EquationEstimatingTheBoundaryTerm}
     &\langle \tr(A_3 u_\tau), \tr u_\tau  \rangle_{L^2(\Gamma_t) \times L^2(\Gamma_t)} = \langle C^{\operatorname{co}} \tr u_\tau, B^{\operatorname{co}} \tr u_\tau  \rangle_{L^2(\Gamma_t) \times L^2(\Gamma_t)} \\
     &= \langle C^{\operatorname{co}} \tr u_\tau, g_\tau  \rangle_{L^2(\Gamma_t) \times L^2(\Gamma_t)}  = \langle C^{\operatorname{co}} \tr u, g_{2\tau}  \rangle_{L^2(\Gamma_t) \times L^2(\Gamma_t)} \nonumber\\
     &\leq \|C^{\operatorname{co}} \tr u\|_{L^2(\Gamma_t)} \|g_{2\tau}\|_{L^2(\Gamma_t)} \leq \| \tr u \|_{L^2(\Gamma_t)} \|g_{2 \tau}\|_{L^2(\Gamma_t)} \leq \Han{1}{u} \|g_{2 \tau}\|_{L^2(\Gamma)} \nonumber 
    \end{align}
    for all $t \in J$ and $\tau > 0$, where $\Gamma$ denotes $J \times \partial \R^3_+$ as usual. We point out that
    $\Han{1}{u}$ is finite as $u \in \G{1}$. Estimate~\eqref{EquationEstimateOfUTauForGronwall2} and~\eqref{EquationEstimatingTheBoundaryTerm} 
    finally lead to
    \begin{align*}
     \Ltwohn{u_\tau(t)}^2 &\leq C_1 \Ltwohn{u_{0,\tau}}^2 + \Ltwoanwgamma{f_\tau}^2 + \frac{1}{\eta} \Han{1}{u} \|g_{2\tau}\|_{L^2(\Gamma)} \\
			  &\quad + C_1 \int_0^t \Ltwohn{u_\tau(s)}^2 ds
    \end{align*}
    for all $t \in J$ and $\tau > 0$ so that Gronwall's lemma implies
    \begin{align}
    \label{EquationWeightedEnergyEstimate}
     \sup_{t \in J} \Ltwohn{u_\tau(t)}^2 \leq \Big(C_1 \Ltwohn{u_{0,\tau}}^2 + \Ltwoanwgamma{f_\tau}^2 + \frac{1}{\eta} \Han{1}{u} \|g_{2\tau}\|_{L^2(\Gamma)} \Big) e^{C_1 T}
    \end{align}
    for all $\tau > 0$.
    
    III) To exploit the weighted energy estimate~\eqref{EquationWeightedEnergyEstimate}, we 
    take $(s,x) \in (J \times \R^3_+) \setminus \mathcal{C}$, i.e.,  
    $|x - x_0| \geq R - C_0 s$ 
    and hence
    \begin{equation*}
     -s + K(R - |x - x_0|) = -s + \frac{1}{C_0}(R - |x - x_0|) \leq 0.
    \end{equation*}
    It follows
    \begin{equation*}
     e^{\tau \Phi(s,x)} = e^{\tau\,(-s + \psi(x))} \leq e^{-\tau \epsilon}
    \end{equation*}
    for all $\tau > 0$.
    On the other hand, $f(s,x) = 0$ for almost all $(s,x) \in \mathcal{C}$, $g(s,x) = 0$ for almost all 
    $(s,x) \in \mathcal{C} \cap (J \times \partial \R^3_+)$, and $u_0(x) = 0$ for almost all $x \in \mathcal{C}_{t=0} \cap \R^3_+$. 
    We conclude that
    \begin{align*}
     &|f_\tau(s,x)| \leq |f(s,x)| \,\,\text{for all } \tau > 0 \quad \text{and} \quad |f_\tau(s,x)| \longrightarrow 0 \,\, \text{as } \tau \rightarrow \infty 
    \end{align*}
    for almost all $(s,x) \in J \times \R^3_+$ so that
    \begin{equation*}
     \Ltwoanwgamma{f_\tau} \longrightarrow 0
    \end{equation*}
    as $\tau \rightarrow \infty$. Analogously, we deduce
    \begin{equation*}
     \|g_{2 \tau}\|_{L^2(\Gamma)} \longrightarrow 0 \quad \text{and} \quad \Ltwohn{u_{0,\tau}} \longrightarrow 0
    \end{equation*}
    as $\tau \rightarrow \infty$. By~\eqref{EquationWeightedEnergyEstimate} the functions $u_\tau$ thus tend 
    to $0$ in $\G{0}$ as $\tau \rightarrow \infty$,
    so that
    \begin{equation}
    \label{EquationEstimateForUTauInSupNormInT}
     C_2 := \sup_{t \in J, \tau > 0} \Ltwohn{u_\tau(t)}^2 < \infty.
    \end{equation}
    Now take a point $(t,x)$ from $\mathcal{C}_{3 \epsilon}$. We then calculate
    \begin{align*}
     3 \epsilon &<  K (R - |x - x_0|) - t \leq -t + \psi(x) + 2 \epsilon = \Phi(t,x) + 2 \epsilon, \\
     \epsilon &< \Phi(t,x).
    \end{align*}
    Estimate~\eqref{EquationEstimateForUTauInSupNormInT} now implies
    \begin{align*}
     \int_{\mathcal{C}_{3 \epsilon} \cap \Omega} |u(t,x)|^2 dx dt &\leq e^{-2 \epsilon \tau} \int_{\mathcal{C}_{3 \epsilon} \cap \Omega} e^{2\tau \Phi(t,x)} |u(t,x)|^2 dx dt \\ 
     &\leq e^{-2 \epsilon \tau} T \sup_{t \in J} \Ltwohn{u_\tau(t)}^2 
     \leq C_2 T e^{-2 \epsilon \tau}
    \end{align*}
    for all $\tau > 0$. Letting $\tau \rightarrow \infty$, 
    we obtain
    $|u(t,x)| = 0$ for almost all $(t,x) \in \mathcal{C}_{3 \epsilon}$.
    
    Finally, we take a sequence $(\epsilon_n)_n$ in $(0,1)$ with $\epsilon_n \rightarrow 0$ as $n \rightarrow \infty$. 
    Since $u(t,x) = 0$ for almost all $(t,x) \in \mathcal{C}_{3 \epsilon_n}$ for all $n \in \N$, we conclude that
    \begin{equation*}
     u(t,x) = 0 \quad \text{for almost all } (t,x) \in \bigcup_{n \in \N} \mathcal{C}_{3 \epsilon_n} = \mathcal{C}.
    \end{equation*}
    
    IV) Now let $f$, $g$, and $u_0$ be as in the assertion. We take a family of functions $(f_\epsilon, g_\epsilon, u_{0,\epsilon})$ 
    in $\Ha{1} \times \E{1} \times \Hh{1}$ for $0 < \epsilon < 1$ such that $f_\epsilon$ converges to $f$ 
    in $\Ltwoa$, $g_\epsilon$ to $g$ in $\E{0}$, and $u_{0,\epsilon}$ to $u_0$ in $\Ltwoh$ as $\epsilon \rightarrow 0$ 
    and the tuples $(0, A_0, \ldots, A_3, D, B, f_\epsilon, g_\epsilon, u_{0,\epsilon})$ are 
    compatible of order $1$ for all $\epsilon \in (0,1)$. Such a family can be constructed as 
    explained at the beginning of step~II) of Lemma~\ref{LemmaApporixmationInLinearSituationInL2}. Using 
    a standard mollifier for the regularization as in step~I) of the proof of Theorem~4.13 in~\cite{SpitzDissertation}, 
    we can construct this family in such a way that    
    \begin{align*}
     &\supp f_\epsilon \subseteq \supp f + B(0,\epsilon), \quad \supp g_\epsilon \subseteq \supp g + B(0,\epsilon) \\
     &\supp u_{0,\epsilon} \subseteq (\supp u_0 \cup (\supp g(0) \times\{0\}) )+ B(0, \epsilon)
    \end{align*}
    for all $\epsilon \in (0,1)$.
    
    Let $\epsilon \in (0,1)$ and $(s,y) \in \R \times \R^3$ with $(s,y) \in \mathcal{C}^C + B(0,\epsilon)$, where $\mathcal{C}^C = \R^4 \setminus \mathcal{C}$.
    We then find $(t,x) \in \mathcal{C}^C$ such that
    $|(t,x) - (s,y)| < \epsilon$. Assume that
    $(s,y)$ belongs to $\mathcal{C}_{(2 + \frac{1}{C_0})\epsilon}$. We then obtain
    \begin{align*}
     R &\leq |x - x_0| + C_0 t \leq |x-y| + C_0(t-s) + |y - x_0| + C_0 s \\
     &\leq \epsilon + C_0 \epsilon + R - C_0\Big(2 + \frac{1}{C_0} \Big) \epsilon = R - C_0 \epsilon,
    \end{align*}
    which is a contradiction. This means that $\mathcal{C}^C + B(0,\epsilon) \subseteq \mathcal{C}_{(2 + \frac{1}{C_0})\epsilon}^C$ 
    for all $\epsilon > 0$. We thus arrive at
    \begin{equation*}
     \supp f_\epsilon \cap \Omega \subseteq (\supp f + B(0,\epsilon)) \cap \Omega \subseteq \mathcal{C}_{(2 + \frac{1}{C_0})\epsilon}^C \cap \Omega
    \end{equation*}
    for all $\epsilon > 0$. Analogously, we derive that
    \begin{align*}
     &\supp g_\epsilon \cap \Gamma \subseteq (\supp g + B(0,\epsilon)) \cap \Gamma \subseteq \mathcal{C}_{(2 + \frac{1}{C_0})\epsilon}^C \cap \Gamma, \\
     &\supp u_{0,\epsilon} \cap \R^3_+ \subseteq ((\supp u_0 \cup (\supp g(0) \times \{0\})) + B(0,\epsilon)) \cap \R^3_+ \\
     &\hspace{6.5em} \subseteq \mathcal{C}_{t = 0,(2 + \frac{1}{C_0})\epsilon}^C \cap \R^3_+,
    \end{align*}
    for all $\epsilon > 0$, where  
    $\mathcal{C}_{t = 0,(2 + \frac{1}{C_0})\epsilon}^C = \R^3 \setminus \mathcal{C}_{t = 0,(2 + \frac{1}{C_0})\epsilon}$. Steps II) and III) 
    now show that the unique solution $u_\epsilon \in C(\clJ, \Ltwoh)$ of~\eqref{IBVP} with inhomogeneity $f_\epsilon$, 
    boundary value $g_\epsilon$, and initial value $u_{0,\epsilon}$ vanishes on $\mathcal{C}_{(2 + \frac{1}{C_0})\epsilon}$, 
    i.e., $u_\epsilon(t,x) = 0$ for almost all $(t,x) \in \mathcal{C}_{(2 + \frac{1}{C_0})\epsilon}$ for each $\epsilon > 0$.
    
    Take a monotonically decreasing sequence $(\epsilon_n)_n$ in $(0,1)$ with $\epsilon_n \rightarrow 0$ as $n \rightarrow \infty$. 
    By Theorem~\ref{TheoremExistenceAndUniquenessOnDomain} there is a constant $C_3$ and a number $\gamma > 0$ such that
    \begin{align*}
     \Gnorm{0}{u_{\epsilon_n} \!-\! u}^2 &\leq C_3 \Ltwohn{u_{0, \epsilon_n} \! - \! u_0}^2 \! + \! C_3\|g_{\epsilon_n} \!- \! g\|_{E_{0,\gamma}}^2 \! + \!\frac{C_3}{\gamma} \Ltwoan{f_{\epsilon_n} \! - \! f}^2 \rightarrow 0
    \end{align*}
    as $n \rightarrow \infty$, in particular $(u_{\epsilon_n})_n$ tends to $u$ in $\Ltwoa$ as $n \rightarrow \infty$. Consequently, there is a 
    subsequence, which we again denote by $(u_{\epsilon_n})_n$, which  
    converges pointwise almost everywhere to $u$. Since $\mathcal{C}_{(2 + \frac{1}{C_0})\epsilon_n} \subseteq \mathcal{C}_{(2 + \frac{1}{C_0})\epsilon_m}$ 
    for all $m > n$, we infer that $u(t,x) = 0$ for almost all $(t,x) \in \mathcal{C}_{(2 + \frac{1}{C_0})\epsilon_n}$ 
    for all $n \in \N$. Hence, 
    \begin{equation*}
     u(t,x) = 0 \quad \text{for almost all } (t,x) \in \bigcup_{n \in \N} \mathcal{C}_{(2 + \frac{1}{C_0})\epsilon_n} = \mathcal{C}. \qedhere
    \end{equation*}
\end{proof}

We also formulate the finite propagation speed property using the forward light cone, cf.~\cite{BahouriCheminDanchin}. 
This version shows that if the data is supported on a forward light cone, then also the solution is supported on this cone.
\begin{cor}
 \label{CorollaryFinitePropagationSpeedForwardCone}
 Let $\eta > 0$, $A_0 \in \Fupdwl{3}{c}{\eta}$, $A_1, A_2 \in \Fcoeff{3}{cp}$, $A_3 = A_3^{\operatorname{co}}$,
 $D \in \Fuwl{3}{c}$, and $B = B^{\operatorname{co}}$. We set 
 \begin{align*}
  C_0 = \frac{1}{\eta} \, \sum_{j = 1}^3 \|A_j\|_{L^\infty(\Omega)}.
 \end{align*}
 Let $R > 0$ and $x_0 \in \overline{\R^3_+}$. We define the forward cone $\mathcal{K}$ by
 \begin{equation*}
  \mathcal{K} = \{(t,x) \in \R \times \R^3 \colon |x - x_0| \leq R + C_0 \,t \}.
 \end{equation*}
 Let $f \in \Ltwoa$, $g \in L^2(J, H^{1/2}(\partial \R^3_+))$, and $u_0 \in \Ltwoh$ such that
 \begin{equation*}
  \begin{aligned}
    f&= 0 \quad &&\text{on } \Omega \setminus \mathcal{K}, \\
    g&= 0 	&&\text{on } (J \times \partial \R^3_+) \setminus \mathcal{K} , \\
    u_0&= 0 	&&\text{on } \R^3_+ \setminus \mathcal{K}_{t=0},
  \end{aligned}
 \end{equation*}
 where 
 $\mathcal{K}_{t=0} = \{x \in \R^3 \colon (0,x) \in \mathcal{K}\}$. Then the unique solution $u \in C(\clJ, \Ltwoh)$ of the linear initial boundary value problem~\eqref{IBVP} 
 with inhomogeneity $f$, boundary value $g$, and initial value $u_0$ is supported in the cone $\mathcal{K}$, 
 i.e., 
 \begin{equation*}
  u(t,x) = 0 \quad \text{for almost all } (t,x) \in \Omega\setminus \mathcal{K}.
 \end{equation*}
\end{cor}
The assertion can be reduced to Theorem~\ref{TheoremFinitePropagationSpeed}, see~\cite[Corollary~6.2]{SpitzDissertation} 
for details.

\begin{rem}
   \label{RemarkFinitePropagationSpeedNonlinear}
   In the framework of Theorem~\ref{TheoremLocalWellposednessNonlinear} assume that the data vanish on a backward 
   light cone or outside of a forward light cone, see Theorem~\ref{TheoremFinitePropagationSpeed} 
   respectively Corollary~\ref{CorollaryFinitePropagationSpeedForwardCone} for the precise statement. Then also 
   the solution of the nonlinear problem~\eqref{EquationNonlinearIBVP} vanishes on the backward respectively 
   forward light cone. This assertion follows from the simple observation that 
   the solution $u$ of~\eqref{EquationNonlinearIBVP} also solves the linear problem~\eqref{EquationIBVPIntroduction} 
   respectively~\eqref{IBVP}  with coefficients $\chi(u)$ and $\sigma(u)$. 
   Theorem~\ref{TheoremFinitePropagationSpeed} respectively Corollary~\ref{CorollaryFinitePropagationSpeedForwardCone}
   then yield the assertion.
  \end{rem}
 
 
 \vspace{2em}
 \textbf{Acknowledgment:}
 I want to thank my advisor Roland Schnaubelt for helpful feedback on this research 
and valuable advice regarding this article. Moreover, I gratefully acknowledge 
financial support by the Deutsche Forschungsgemeinschaft (DFG) through CRC 1173.

\vspace{1em} 
 
 \emergencystretch = 1em
 \printbibliography
\vspace{2em}
\end{document}